\documentclass[10pt]{article}

\usepackage{bm}
\usepackage{framed}
\usepackage{color}
\usepackage{graphicx}
\usepackage{mathrsfs,amsfonts,latexsym,amssymb,amsthm,amsmath}
\usepackage{authblk}
\usepackage{anysize}\marginsize{2cm}{2cm}{1cm}{2.5cm}

\usepackage[font=footnotesize]{caption}
\usepackage[labelformat=simple]{subcaption}

\numberwithin{equation}{section}
\newtheorem{theorem}{Theorem}[section]
\newtheorem{lemma}[theorem]{Lemma}

\newtheorem{corollary}[theorem]{Corollary}
\newtheorem{definition}{Definition}[section]
\newtheorem{remark}[theorem]{Remark}

\makeatletter
\newcommand{\mylabel}[2]{#2\def\@currentlabel{#2}\label{#1}}
\makeatother

\newcommand{\Hgrad}{H^1(\Omega)}
\newcommand{\Hdiv}{H(\mbox{div}, \Omega)}
\newcommand{\Hcurl}{H(\mbox{curl}, \Omega)}
\newcommand{\Hcurlzero}{H_0(\mbox{curl}, \Omega)}
\newcommand{\Hdivzero}{H_0(\mbox{div}, \Omega)}
\newcommand{\BB}{\bm{B}}

\newcommand{\ppsi}{\bm{\psi}}
\newcommand{\mm}{\bm{m}}
\newcommand{\bb}{\bm{b}}
\newcommand{\ee}{\bm{e}}
\newcommand{\ff}{\bm{f}}
\newcommand{\nn}{\bm{n}}
\newcommand{\ii}{\bm{i}}
\newcommand{\jj}{\bm{j}}
\newcommand{\kk}{\bm{k}}
\newcommand{\pp}{\bm{p}}
\newcommand{\rr}{\bm{r}}
\newcommand{\uu}{\bm{u}}
\newcommand{\ruu}{\bm{{\rm u}}}
\newcommand{\vv}{\bm{v}}
\newcommand{\ww}{\bm{w}}
\newcommand{\xx}{\bm{x}}
\newcommand{\ttheta}{\bm{\theta}}
\newcommand{\vspan}{{\mbox{span}}}
\newcommand{\rank}{{\mbox{rank}}}
\newcommand{\lev}{{L}}
\newcommand{\vecu}{\ww}

\newcommand{\mfm}{\mathfrak{m}}
\newcommand{\mfa}{\mathfrak{a}}
\newcommand{\mfs}{\mathfrak{s}}

\newcommand{\mfmi}[1]{\mfm_{[#1]}}
\newcommand{\mfai}[1]{\mfa_{[#1]}}
\newcommand{\mfsi}[1]{\mfs_{[#1]}}

\newcommand{\Grad}{\nabla }
\newcommand{\Curl}{\nabla \times}
\newcommand{\Div}{\nabla \cdot}
\newcommand{\laplacian}{\mathfrak{L}}

\newcommand{\mysquare}{{\scriptscriptstyle\square}}
\newcommand{\mycircle}{{\text{\small o}}}

\date{}
\begin{document}

\title{Isogeometric analysis for 2D and 3D curl-div problems: \\ Spectral symbols and fast iterative solvers}

\author[1]{Mariarosa Mazza\footnote{Corresponding author}}

\author[2]{Carla Manni}

\author[1,3]{Ahmed Ratnani}

\author[4,5]{Stefano Serra-Capizzano}

\author[2]{Hendrik Speleers}

\affil[1]{\small Max-Planck Institut f\"ur Plasmaphysik, Boltzmannstra{\ss}e 2, 87548 Garching bei M{\"u}nchen, Germany}
\affil[2]{University of Rome ``Tor Vergata'', Via della Ricerca Scientifica 1, 00133 Rome, Italy}
\affil[3]{\small Technische Universit\"at M\"unchen, Boltzmannstra{\ss}e 3, 87548 Garching bei M\"unchen, Germany}

\affil[4]{\small University of Insubria, Via Valleggio 11, 22100 Como, Italy}
\affil[5]{\small Uppsala University, Box 337
SE-751 05, Uppsala, Sweden}

\maketitle

\begin{abstract}
\noindent Alfv\'en-like operators are of interest in magnetohydrodynamics, which is used in plasma physics to study the macroscopic behavior of plasma.
Motivated by this important and complex application, we focus on a parameter-dependent curl-div problem that can be seen as a prototype of an {Alfv\'en}-like operator, and we discretize it using isogeometric analysis based on tensor-product B-splines. The involved coefficient matrices can be very ill-conditioned, so that standard numerical solution methods perform quite poorly here.
In order to overcome the difficulties caused by such ill-conditioning, a two-step strategy is proposed. First, we conduct a detailed spectral study of the coefficient matrices, highlighting the critical dependence on the different physical and approximation parameters. Second, we exploit such spectral information to design fast iterative solvers for the corresponding linear systems.
For the first goal we apply the theory of (multilevel block) Toeplitz and generalized locally Toeplitz sequences, while for the second we use a combination of multigrid techniques and preconditioned Krylov solvers.
Several numerical tests are provided both for the study of the spectral problem and for the solution of the corresponding linear systems.
\end{abstract}

\noindent\textbf{Keywords: } Alfv\'en-like operator; isogeometric analysis; spectral symbol; GLT theory; multigrid techniques; Krylov preconditioning

\let\thefootnote\relax\footnote{\emph{Email addresses}: \texttt{\texttt{mariarosa.mazza@ipp.mpg.de}} (Mariarosa Mazza), \texttt{manni@mat.uniroma2.it} (Carla Manni),\\ \texttt{ahmed.ratnani@ipp.mpg.de} (Ahmed Ratnani), \texttt{stefano.serrac@uninsubria.it} (Stefano Serra-Capizzano),\\ \texttt{speleers@mat.uniroma2.it} (Hendrik Speleers)}

\section{Introduction}\label{sec:intro}

Plasmas are known to be the forth state of matter together with gas, liquid and solid. In fact, 99\% of the universe is composed of plasmas. Cold and hot plasmas pervade  many fields, including medical and waste processing, aerospace and aviation industries and nuclear fusion energy. The sun is a natural fusion reactor. Building such a device on earth is challenging: it requires the magnetic confinement of hot plasma particles; the temperature may reach $10^8$ degrees at the center and drops quickly when approaching the device wall. These huge gradients in temperature, pressure or density and the strong magnetic field lead to very high anisotropies reaching up to ten orders of magnitude.
Divertor tokamaks and stellarators are the only known devices that have chances to succeed in a nuclear fusion plant. The former have a  toroidal geometry but some undesirable instabilities, while the latter have more complicated geometries but less instabilities to control.
Understanding the physics of such devices requires robust software for numerical experiments, which are necessary because of the exorbitant  cost of the devices themselves.

In plasma physics, magnetohydrodynamics (MHD) is used to study the macroscopic behavior of the plasma. For instance, with additional physics extensions, many aspects of the large-scale instabilities that appear in a magnetic confined plasma  \cite{Huysmans2009a,Becoulet2014a} can be described satisfactorily in the MHD framework.
The set of equations that describe MHD are a combination of the Navier-Stokes equations of fluid dynamics and the Maxwell equations of electromagnetism \cite {Jardin2010}. Solving MHD equations globally in the (complicated) geometry of a divertor tokamak or a stellarator is a highly demanding task because of the strong temporal and spatial multi-scale nature of the problem and of the high anisotropies mentioned above. Due to these difficulties, explicit time integrators are in general not suitable because they lead to very small time steps, while the use of implicit methods leads to very ill-conditioned matrices.

Over the last decade, a promising technique, called \emph{physics-based preconditioning} \cite{Chacon2008,Chacon2016, Philip2008,Shadid2010}, leaded to new scalable MHD solvers. The proposed algorithm uses a multigrid as a preconditioner for a Jacobian-free Newton-Krylov method. The preconditioner is constructed by parabolization of the hyperbolic partial differential equations (PDEs). In order to understand how such physics-based preconditioner works, one first needs to unravel the spectral properties of every operator and also what are their dependency and pathologies with respect to both the discretization and physical parameters.

Among the different operators, one encounters the {Alfv\'en}-like operator \cite{Jardin}, which is characterized by a weighting of the curl and div operators. Its formal definition is given by
\begin{equation*}
\mathcal{L}_A \uu := \nu \uu - \beta\lambda \nabla (\Div \uu) - \lambda \left( \bb_0 \times \left(\nabla \times \Curl (\bb_0 \times \uu)\right) \right),
\end{equation*}
where $\nu>0$, $\beta\in\bigl(10^{-4},10^{-1}\bigr]$, $\lambda:=V_A \Delta t$ is the numerical Alfv\'en length,
and $\mathbf{b_0}:=\frac{\BB_0}{\| \BB_0 \|}$ with $\BB_0$ the magnetic field, $V_A:=\frac{\| \BB_0 \|}{\sqrt{\rho_0 \mu_0}}$ the Alfv\'en speed, $\Delta t$ the time step for the implicit time scheme, $\mu_0$ the permeability of the vacuum and $\rho_0$ the density. The operator $\mathcal{L}_A$ needs to be inverted in an optimal way, using an algorithm with a high scalability property over a computational domain $\Omega \subset \mathbb{R}^d$, $d=2,3$.
To this end, we need to understand the competition between the curl and div terms. Therefore, in this paper, we focus on the simpler weighted operator 
\begin{equation}
\mathcal{L}_{\alpha, \beta} \uu := - \beta \Grad \Div \uu + \alpha \Curl \Curl \uu,
\label{eq:alfven-op}
\end{equation}
with $\alpha \sim 1$ and $\beta \in \bigl(10^{-4},10^{-1}\bigr]$.
This problem contains the essential features of $\mathcal{L}_A$. Furthermore, we remark that such a parameter-dependent operator has an interest in itself and in fact it appears in other situations, including the Stokes equation and Maxwell equations \cite{Ciarlet2005}. It can also be seen as a \textit{weighted} Laplacian for vector fields (equivalently, Hodge Laplace for 1-forms).

The functional analysis framework traditionally involves the Sobolev spaces $\Hcurl$ and $\Hdiv$ and more generally the de Rham sequence \cite{Monk2003, Cessenat1996} when using a mixed formulation \cite{Arnold2012}.
The natural space for the unknown field $\uu$ is $\Hcurlzero \cap \Hdiv$ or $\Hcurl \cap \Hdivzero$.
In this paper we assume $\Omega=(0,1)^d$, $d=2,3$, and so both $\Hcurlzero \cap \Hdiv$ and $\Hcurl \cap \Hdivzero$ are continuously embedded into $\bigl(\Hgrad\bigr)^d$ \cite{Girault1986}.
Results on the well-posedness and approximation still hold in $\bigl(\Hgrad\bigr)^d$, and we refer the reader to \cite{Arnold2012,Auchmuty2005} and the references therein.
Therefore, we consider the following variational formulation of \eqref{eq:alfven-op} in a finite-dimensional vector space $\mathbb{V}_h \subset \bigl(\Hgrad\bigr)^d$,
\begin{equation}\label{eq:curl-div}
(\mathcal{L}_{\alpha, \beta} \uu_h, \vv_h) = \alpha(\Curl \uu_h,\Curl\vv_h) +\beta(\Div\uu_h,\Div\vv_h), \quad \forall \uu_h,\vv_h \in \mathbb{V}_h,
\end{equation}
to find an approximate solution of the problem $\mathcal{L}_{\alpha, \beta} \uu=\ff$ with suitable boundary conditions.
We focus on isogeometric analysis (IgA) as discretization technique. More precisely, we choose our approximation space $\mathbb{V}_h$ to be composed of vector fields whose components are linear combinations of tensor-product B-splines.

The discretization of problems based on the weighted operator \eqref{eq:curl-div} leads to solving linear systems,
where the involved coefficient matrices depend on many factors: the problem parameters $\alpha$, $\beta$, the basic curl and div operators, the fineness parameter and the degree of the B-spline approximation. Numerical experiments show that the linear algebra problems range from ill-conditioned to severely ill-conditioned, and hence standard numerical solution methods perform quite poorly on such problems.

In order to overcome the difficulties given by the observed ill-conditioning, a two-step strategy is proposed. First, we conduct a detailed spectral study of the coefficient matrices, highlighting the critical dependence on the different parameters, and then we use such spectral information to design fast iterative solvers for the corresponding linear systems.
For the first goal we apply the theory of (multilevel block) Toeplitz \cite{GSz,Tillinota,TyZ} and  generalized locally Toeplitz \cite{GS,GLT-bookII,GLT-bookIII} sequences, while for the second we use the spectral knowledge and a combination of multigrid techniques \cite{libroBramble} and preconditioned Krylov solvers \cite{libroSaad}.

The theoretical investigation and the numerics show that the important characteristics of the spectral distribution of the coefficient matrices obtained from the B-spline discretization of \eqref{eq:curl-div} can be properly estimated in terms of the  spectrum of the matrices arising from the same discretization for the standard Laplace operator \cite{GaroniNM,GMSSS} suitably weighted by the problem parameters $\alpha$ and $\beta$. This gives the theoretical foundations for the proposed iterative solvers for the corresponding linear systems, which turn out to be robust with respect to both the fineness parameter and the approximation degree of the used discretization.

For the sake of simplicity, we only consider homogeneous Dirichlet (no-slip) boundary conditions, i.e., $\uu = 0$ on $\partial \Omega$, imposed in strong form.
However, our spectral analysis involves solely internal knots, and therefore, applies to any kind of boundary conditions. In particular, it also applies to Dirichlet boundary conditions imposed weakly by a Nitsche method as described in \cite{Evans2013}. The Nitsche method is a good alternative to handle Dirichlet boundary conditions and domain decomposition, which may be a key point to devise a scalable solver for MHD in a complex geometry.

A similar symbol-based two-step strategy has already been successfully employed in 
\cite{BeSe,serra2,system-PDE-FEM,MAS,mog,glt} 
for different types of differential equations or systems of differential equations, discretized by various techniques such as finite differences, finite elements, and IgA.
In particular, the symbol-based approach has been investigated in \cite{serra2,cmame2,SINUM-IgA} for IgA discretizations of (scalar) second-order elliptic problems; alternative iterative solution methods can be found in \cite{BPX,Clemens,ST}.

The remainder of the paper is organized as follows. In Section~\ref{sec:spectral_analysis} we give notations and definitions  relevant for our spectral analysis, and in Section~\ref{sec:notation} we introduce the basics of B-spline discretizations. Section \ref{sec:discr} describes the B-spline discretization of \eqref{eq:curl-div} with homogeneous Dirichlet boundary conditions. In Section \ref{sec:spectral} we perform a detailed spectral analysis of the resulting matrices and discuss few numerical tests. In Section \ref{sec:proposals} we exploit such a spectral information to design ad hoc solvers for the corresponding linear systems, and we illustrate their performance with several numerical examples for both 2D and 3D cases. Finally, we conclude the paper in Section \ref{sec:end}.

\section{Preliminaries on spectral analysis}\label{sec:spectral_analysis}

In this section, we introduce some preliminary spectral tools used later on. First we recall the definition of spectral distribution of matrix-sequences, and then we focus on multilevel block Toeplitz matrices and the GLT theory.
Throughout the paper, we follow the standard convention for operations with multi-indices (see, e.g., \cite{lusin,Ty}).
Furthermore, given a multi-index $\nn:=(n_1,\ldots,n_d)\in\mathbb N^d$,
we say $\nn\rightarrow \infty$ if $n_i\rightarrow \infty$, $i=1,\ldots,d$.

\subsection{Spectral distribution and symbol of a matrix-sequence}\label{sub:symbol}

We begin with the formal definition of spectral distribution in the sense of the eigenvalues and singular values for a general matrix-sequence.

\begin{definition}\label{def-distribution}
	Let $f:G\to\mathbb{C}^{s\times s}$ be a measurable matrix-valued function, defined on a measurable set $G\subset\mathbb R^q$ with $q\ge 1$, $0<\mu_q(G)<\infty$, where $\mu_q$ is the Lebesgue measure. Let $\mathcal C_0(\mathbb K)$ be the set of continuous functions with compact support over $\mathbb K\in \{\mathbb C, \mathbb R_0^+\}$ and let $\{A_{\nn}\}_{\nn}$ be a matrix-sequence with ${\rm dim}(A_{\nn})=:d_{\nn}$ and $d_{\nn}\rightarrow\infty$ as $\nn\rightarrow \infty$.
	\begin{enumerate}
		\item[(a)] $\{A_{\nn}\}_{\nn}$ is {distributed as the pair $(f,G)$ in the sense of the eigenvalues,} denoted by
		$$\{A_{\nn}\}_{\nn}\sim_\lambda(f,G),$$ if the following
		limit relation holds for all $F\in\mathcal C_0(\mathbb C)$:
		\begin{align}\label{distribution:sv-eig}
		\lim_{\nn\to\infty}\frac{1}{d_{\nn}}\sum_{j=1}^{d_{\nn}}F(\lambda_j(A_{\nn}))=
		\frac1{\mu_q(G)}\int_G \frac{\sum_{i=1}^{s}F(\lambda_i(f({\bm t})))}{s} {\rm d}{\bm t},
		\end{align}
		where $\lambda_j(A_{\nn})$, $j=1,\ldots,d_{\nn}$ are the eigenvalues of $A_{\nn}$ and $\lambda_i(f)$, $i=1,\ldots,s$ are the eigenvalues of $f$.
		In this case, we say that $f$ is the {(spectral) symbol} of the matrix-sequence $\{A_{\nn}\}_{\nn}$.
		\item[(b)] $\{A_{\nn}\}_{\nn}$ is {distributed as the pair $(f,G)$ in the sense of the singular values,} denoted by
		$$\{A_{\nn}\}_{\nn}\sim_\sigma(f,G),$$ if the following
		limit relation holds for all $F\in\mathcal C_0(\mathbb R_0^+)$:
		\begin{align}\label{distribution:sv-eig-bis}
		\lim_{\nn\to\infty}\frac{1}{d_{\nn}}\sum_{j=1}^{d_{\nn}}F(\sigma_j(A_{\nn}))=
		\frac1{\mu_q(G)}\int_G \frac{\sum_{i=1}^sF(\sigma_i(f({\bm t})))}{s} {\rm d}{\bm t},
		\end{align}
		where $\sigma_j(A_{\nn})$, $j=1,\ldots,d_{\nn}$ are the singular values of $A_{\nn}$ and $\sigma_i(f)$, $i=1,\ldots,s$ are the singular values of $f$. In this case, we say that $f$ is the {singular value symbol} of the matrix-sequence $\{A_{\nn}\}_{\nn}$.
	\end{enumerate}
\end{definition}

If $f$ is smooth enough and the matrix-size of $A_{\nn}$ is sufficiently large, then the limit relation \eqref{distribution:sv-eig} (resp., \eqref{distribution:sv-eig-bis}) has the following informal interpretation:
a first set of $d_{\nn}/s$ eigenvalues (resp., singular values) of $A_{\nn}$ is approximated by a sampling of $\lambda_1(f)$ (resp., $\sigma_1(f)$) on a uniform equispaced grid of the domain $G$, a second set of $d_{\nn}/s$ eigenvalues (resp., singular values) of $A_{\nn}$ is approximated by a sampling of $\lambda_2(f)$ (resp., $\sigma_2(f)$) on a uniform equispaced grid of the domain $G$, and so on.

\begin{definition}
	Let $\{A_{\nn}\}_{\nn}$ be a matrix-sequence with ${\rm dim}(A_{\nn})=:d_{\nn}$ and $d_{\nn}\rightarrow\infty$ as $\nn\rightarrow \infty$. We say that $\{A_{\nn}\}_{\nn}$ is a {zero-distributed matrix-sequence} if $\{A_{\nn}\}_{\nn}\sim_\sigma(0,G)$.
\end{definition}

\subsection{Unilevel and multilevel Toeplitz sequences}\label{sub:toeplitz}
We now recall the definition of Toeplitz sequences generated by univariate functions in $L^1([-\pi, \pi])$.
\begin{definition}
	Let $f \in L^1([-\pi, \pi])$ and let $\hat f_k$ be its Fourier coefficients,
	\begin{align*}
	\hat f_k := \frac{1}{2 \pi} \int_{-\pi}^{\pi} f(\theta) e^{-{\rm i} k \theta} ~ {\rm d}\theta\in\mathbb{C},
	\quad k\in\mathbb Z.
	\end{align*}
	The $n$-th (unilevel) Toeplitz matrix associated with $f$ is the $n \times n$ matrix defined by
	\begin{align*}
	T_n(f):=\left[ \hat f_{i-j} \right]_{i,j=1}^{n}=
	\begin{bmatrix}
	\hat f_{0} &  \hat f_{-1} &   \hat f_{-2} & \ldots &  \hat f_{-(n-1)}  \\
	\hat f_{1} &  \hat f_{0} &   \hat f_{-1} & \ldots &   \\
	\hat f_{2} &  \hat f_{1} &   \hat f_{0} & \ldots & \vdots   \\
	\vdots &     &         & \ddots &  \\
	\hat f_{n-1} & \ldots &   & \ldots &  \hat f_{0}
	\end{bmatrix}.
	\end{align*}
	The matrix-sequence $\{T_n(f)\}_{n}$ is called the {Toeplitz sequence generated by $f$}, which
	in turn is referred to as the {generating function} of $\{T_{n}(f)\}_{n}$.
	
\end{definition}
The Fourier sequence $\{ \hat f_k \}_k$ determines uniquely the function $f$ and vice versa.
Therefore, the function $f$, if it exists, is also uniquely determined by the Toeplitz sequence $\{ T_n(f) \}_{n}$.


The notion of Toeplitz sequences can be generalized to multivariate matrix-valued generating functions. Let $L^1(d,s)$ denote the space of $s\times s$ matrix-valued $L^1$ functions over $[-\pi, \pi]^d$.
\begin{definition}
	Let $f\in L^1(d,s)$ and let $\hat f_{\kk}$ be its Fourier coefficients
	\begin{equation*}
	\hat f_{\kk}:=\frac1{(2\pi)^d}\int_{[-\pi,\pi]^d}f(\ttheta){\rm e}^{-{\rm i}\langle\kk,\ttheta\rangle}\,{\rm d}\ttheta\in\mathbb{C}^{s\times s},
	\quad \kk:=(k_1,\ldots,k_d)\in\mathbb Z^d, \quad \ttheta:=(\theta_1, \ldots, \theta_d)\in[-\pi,\pi]^d,
	\end{equation*}
	where $\langle\kk,\ttheta\rangle:=\sum_{l=1}^d k_l\,\theta_l$.
	Then, the $\nn$-th Toeplitz matrix associated with $f$ is the matrix of size $s\prod_{l=1}^dn_l$ defined by
	\begin{equation*}
	T_{\nn}(f):=\left[\hat f_{\ii-\jj}\right]_{\ii,\jj=\bm{1}}^{\nn}=\sum_{|j_1|<n_1}\cdots\sum_{|j_d|<n_d}\left[J_{n_1}^{(j_1)}\otimes\cdots\otimes
	J_{n_d}^{(j_d)}\right] \hat f_{\jj},
	\end{equation*}
	where $\bm{1}:=(1,\ldots,1)\in\mathbb{N}^d$, $\ii:=(i_1,\ldots,i_d)\in\mathbb{N}^d$, $\jj:=(j_1,\ldots,j_d)\in\mathbb{N}^d$, and $\otimes$ denotes the (Kronecker) tensor product of matrices. The term
	$J_m^{(l)}$ is the matrix of size $m$ whose $(i,j)$ entry equals $1$ if $i-j=l$ and zero otherwise.
	The matrix-sequence $\{T_{\nn}(f)\}_{\nn}$ is called the {$d$-level block Toeplitz sequence generated by $f$}, which
	in turn is referred to as the {generating function} of $\{T_{\nn}(f)\}_{\nn}$.
\end{definition}

Finally, we recall the following result on the spectral distribution of multilevel block Toeplitz sequences \cite{Tillinota}.
\begin{theorem}\label{szego-herm}
	If $f:[-\pi, \pi]^d\rightarrow \mathbb{C}^{s\times s}$ is a Hermitian $L^1(d,s)$ function, then $\{T_{\nn}(f)\}_{{\nn}}\sim_\lambda(f,[-\pi, \pi]^d)$.
\end{theorem}

\subsection{Essentials of the GLT theory}\label{sub:GLT}

Both zero-distributed matrix-sequences and multilevel block Toeplitz sequences introduced in the previous subsections belong to a larger class of matrix-sequences known as \emph{generalized locally Toeplitz (GLT) class}. In short, the GLT class is an algebra virtually containing any sequence of matrices coming from ``reasonable'' approximations by local PDE discretization methods (finite differences, finite elements, isogeometric analysis, etc.). The GLT algebra is especially useful when nonconstant coefficients occur in the considered PDE. More details can be found in the seminal work \cite{Tilli-lt} focusing on the spectrum of one-dimensional differential operators and in \cite{glt,serra1} containing a generalization to multivariate differential operators (see also the books \cite{GS,GLT-bookII,GLT-bookIII}).

Without going into details of the GLT algebra, here we list some interesting properties of GLT sequences. We will use them in the context of preconditioning and for proving that a sequence of (multilevel block) Toeplitz matrices, up to low-rank corrections, is a GLT sequence whose symbol is not affected by the low-rank perturbation.

\begin{description}
	\item[GLT1] Each GLT sequence has a symbol $f(\xx,\ttheta)$ for $(\xx,\ttheta)\in [0,1]^d\times [-\pi,\pi]^d$ in the  singular value sense according to Definition \ref{def-distribution} with $q=2d$.
	If the sequence is Hermitian, then the symbol also holds in the eigenvalue sense.
	The variables $\xx\in [0,1]^d$ and $\ttheta\in [-\pi,\pi]^d$ are called space and Fourier variables, respectively.
	
	\item[GLT2] The set of GLT sequences form a $*$-algebra, i.e., it is closed under linear combinations, products, inversion (whenever the symbol is singular in, at most, a set of zero Lebesgue measure), and conjugation. Hence, the sequence obtained via algebraic operations on a finite set of given GLT sequences is still a GLT sequence and its symbol is  obtained by performing the same algebraic manipulations on the corresponding symbols of the input GLT sequences.
	
	\item[GLT3] Every Toeplitz sequence generated by a function $f\in L^1([-\pi,\pi]^d)$ is a GLT sequence and its symbol is $f$, with the specifications mentioned in {\bf GLT1}. Note that such $f$ only depends on Fourier variables $\ttheta\in [-\pi,\pi]^d$ and not on space variables $\xx\in [0,1]^d$.
	
	\item[GLT4] Every zero-distributed matrix-sequence is a GLT sequence with symbol~$0$. In particular, any sequence in which the rank divided by the size tends to zero as the matrix-size tends to infinity and any sequence with asymptotically infinitesimal spectral norm have symbol~$0$.
\end{description}

\section{Preliminaries on IgA discretizations}\label{sec:notation}

In this section, we introduce some preliminary IgA tools. We start by recalling the definition of (cardinal) B-splines together with some relevant properties. Then, we collect some spectral results on matrices involved in the IgA discretization of 1D elliptic problems, which will be important in the IgA discretization of the curl-div problem \eqref{eq:curl-div} discussed in Section \ref{sec:discr}.

\subsection{B-splines and cardinal B-splines}\label{sub:cardinal}
For $p\ge0$ and $n\ge1$, consider the uniform knot sequence
\begin{equation*}\label{eq:knots}
\xi_1=\cdots=\xi_{p+1}:=0<\xi_{p+2}<\cdots<\xi_{p+n}<1=:\xi_{p+n+1}=\cdots=\xi_{2p+n+1},
\end{equation*}
where
\begin{equation*}
\xi_{i+p+1}:=\frac{i}{n}, \quad i=0,\ldots,n.
\end{equation*}
This knot sequence allows us to define $n+p$ B-splines of degree $p$.
\begin{definition}
	The B-splines of degree $p$ over a uniform mesh of $[0,1]$, consisting of $n$ intervals, are denoted by
	\begin{equation*}
	N_{i}^p:[0,1]\rightarrow \mathbb{R}, \quad i=1,\ldots,n+p,
	\end{equation*}
	and defined recursively as follows: for $1 \le i\le n+2p$,
	\begin{equation*}
	N_i^0(x):=\begin{cases}
	1, & x \in [\xi_i,\xi_{i+1}), \\
	0, & \text{otherwise};
	\end{cases}
	\end{equation*}
	for $1\le k\le p$ and $1\le i\le n+2p-k$,
	\begin{equation*}
	N_i^k(x):=\frac{x-\xi_i}{\xi_{i+k}-\xi_i}N_{i}^{k-1}(x)+\frac{\xi_{i+k+1}-x}{\xi_{i+k+1}-\xi_{i+1}}N_{i+1}^{k-1}(x),
	\end{equation*}
	where a fraction with zero denominator is assumed to be zero.
\end{definition}
It is well known that the B-splines $N_{i}^p$, $i=1,\ldots,n+p$, form a basis and they enjoy the following properties (see, e.g., \cite{deBoor}).
\begin{itemize}
	\item Local support:
	\begin{equation*}
	{\rm supp}(N_{i}^p)=[\xi_i,\xi_{i+p+1}], \quad i=1,\ldots,n+p;
	\end{equation*}
	\item Smoothness:
	\begin{equation*}
	N_{i}^p \in \mathcal{C}^{p-1}(0,1), \quad i=1,\ldots,n+p;
	\end{equation*}
	\item Differentiation:
	\begin{equation*}
	\left(N_{i}^p(x)\right)' = p\left(\frac{N_{i}^{p-1}(x)}{\xi_{i+p}-\xi_i}-
	\frac{N_{i+1}^{p-1}(x)}{\xi_{i+p+1}-\xi_{i+1}}\right), \quad i=1,\ldots,n+p, \quad p \geq 1;
	\end{equation*}
	\item Nonnegative partition of unity:
	\begin{equation*}
	N_{i}^p(x)\ge0, \quad i=1,\ldots,n+p, \qquad \sum_{i=1}^{n+p}N_{i}^p(x)=1;
	\end{equation*}
	\item Vanishing at the boundary:  
	\begin{equation} \label{eq:spline-boundary}
	N_{i}^p(0)=N_{i}^p(1)=0, \quad i=2,\ldots,n+p-1.
	\end{equation}
\end{itemize}

The central B-splines $N_i^p$, $i=p+1,\ldots,n$, are uniformly shifted and scaled versions of a single shape function, the so-called cardinal B-spline $\phi_p:\mathbb{R}\rightarrow \mathbb{R}$,
\begin{equation*}
\phi_0(t) := \begin{cases}
1, & t \in [0, 1), \\
0, & \text{otherwise},
\end{cases}
\end{equation*}
and
\begin{equation*}
\phi_p (t) := \frac{t}{p} \phi_{p-1}(t) + \frac{p+1-t}{p} \phi_{p-1}(t-1), \quad p \geq 1.
\end{equation*}
More precisely, we have
\begin{equation*}
N^{p}_i(x) =\phi_{p}(nx-i+p+1), \quad i=p+1,\ldots,n,
\end{equation*}
and
\begin{equation*}
\left(N^{p}_i(x)\right)' =n\phi'_{p}(nx-i+p+1), \quad i=p+1,\ldots,n.
\end{equation*}
The cardinal B-spline has the following properties (see, e.g., \cite[Section3.1]{GaroniNM} and references therein).
\begin{itemize}
	\item Local support:
	\begin{equation*}
	{\rm supp}(\phi_p)=[0,p+1];
	\end{equation*}
	\item Smoothness:
	\begin{equation*}
	\phi_p \in \mathcal{C}^{p-1}(\mathbb{R});
	\end{equation*}
	\item Differentiation:
	\begin{equation*}
	{\phi}_p' (t) = \phi_{p-1}(t) - \phi_{p-1}(t-1), \quad p \geq 1;
	\end{equation*}
	\item Symmetry:
	\begin{equation*}
	\phi_p(t) = \phi_p(p+1-t);
	\end{equation*}
	\item Inner product: 
	\begin{equation*}
	\int_{\mathbb{R}} \phi_{p_1}^{(r_1)}(t) \phi_{p_2}^{(r_2)}(t+\tau) ~ {\rm d}t = (-1)^{r_1} \phi_{p_1+p_2+1}^{(r_1+r_2)}(p_1+1+\tau) = (-1)^{r_2} \phi_{p_1+p_2+1}^{(r_1+r_2)}(p_2+1-\tau).
	\end{equation*}
\end{itemize}

Finally, we recall the definition of tensor-product B-splines.
\begin{definition}
	The tensor-product B-splines of degrees $\pp:=(p_1,\ldots,p_d)$ over a uniform mesh of $[0,1]^d$, consisting of $\nn:=(n_1,\ldots,n_d)$ intervals in each direction, are denoted by
	\begin{equation*}
	N_{\ii}^{\pp}:[0,1]^d\rightarrow \mathbb{R}, \quad \ii=\bm{1},\ldots,\nn+\pp,
	\end{equation*}
	and defined as
	\begin{equation*}
	N_{\ii}^{\pp}(\xx):=\prod_{l=1}^d N_{i_l}^{p_l}(x_{l}), \quad \xx:=(x_1,\ldots,x_d)\in[0,1]^d,
	\end{equation*}
	where $\bm{1}:=(1,\ldots,1)\in\mathbb{N}^d$, $\ii:=(i_1,\ldots,i_d)\in\mathbb{N}^d$.
\end{definition}
We define the tensor-product spline space $\mathbb{S}^{\bm{p}}_{\nn}$ as
\begin{equation}\label{eq:tensor2D}
\mathbb{S}^{\bm{p}}_{\nn}:=\vspan\left\{N_{\ii}^{\pp}:\ii=\bm{2},\ldots,\bm{n+p-1}\right\}.
\end{equation}
Note that all the elements of this space vanish at the boundary of $[0,1]^d$;
see \eqref{eq:spline-boundary};
hence, the space is suited for dealing with homogeneous Dirichlet boundary conditions.
In the remainder of this paper, we will focus on the spline spaces $\mathbb{S}^{p}_{n}$, $\mathbb{S}^{p,p}_{n_1,n_2}$ and $\mathbb{S}^{p,p,p}_{n_1,n_2,n_3}$ in 1D, 2D and 3D, respectively.

\subsection{IgA mass, advection, and stiffness matrices}\label{sub:IgAmatr}
In the context of IgA discretization of elliptic problems, we often deal with the following mass, advection, and stiffness matrices
\begin{align}
M^p_n &:= \left[\int_0^1 N_{i+1}^p(x) ~ N_{j+1}^p(x) ~{\rm d}x\right]_{i, j=1}^{n+p-2}, \label{eq:matr_M}
\\
A^p_n &:= \left[\int_0^1 N_{i+1}^p(x) ~ \left(N_{j+1}^p(x)\right)' ~{\rm d}x\right]_{i, j=1}^{n+p-2},\label{eq:matr_A}
\\
S^p_n &:= \left[\int_0^1 \left(N_{i+1}^p(x)\right)' ~ \left(N_{j+1}^p(x)\right)' ~{\rm d}x\right]_{i, j=1}^{n+p-2}.
\label{eq:matr_S}
\end{align}
The matrices $M^p_n$ and $S^p_n$ are symmetric, and $A^p_n$ is skew-symmetric. Furthermore, using the results of Section \ref{sub:cardinal}, the central parts of these matrices can be expressed as
\begin{align*}
(M^p_n)_{i,j} &= \frac{1}{n}~ \phi_{2p+1}(p+1-(i-j)),
\\
(A^p_n)_{i,j} &= - \, {\phi}'_{2p+1}(p+1-(i-j)),
\\
(S^p_n)_{i,j} &= -\, n {\phi}''_{2p+1}(p+1-(i-j)),
\end{align*}
for $i,j=2p,\ldots,n-p-1$. This means that they are Toeplitz matrices up to a low-rank correction, and the following results on the spectral distribution of mass, advection, and stiffness matrix-sequences hold \cite{GaroniNM,MC-collocation}. For completeness, we provide a compact proof based on the GLT theory.
%
\begin{theorem} \label{thm:spline-distr}
	It holds that
	\begin{enumerate}
		\item[\mylabel{eq:Mdistr}{(a)}]
		$\{n M^p_n\}_n \sim_{\lambda}(\mfm_p,[-\pi,\pi])$, where the symbol $\mfm_p$ is given by
		\begin{equation*}
		\mfm_p(\theta) := \phi_{2p+1}(p+1) + 2 \sum_{k=1}^p \phi_{2p+1}(p+1-k) \cos(k \theta);
		\end{equation*}
		\item[\mylabel{eq:Adistr}{(b)}]
		$\{-{\rm i} A^p_n\}_n \sim_{\lambda}(\mfa_p,[-\pi,\pi])$, where the symbol $\mfa_p$ is given by
		\begin{equation*}
		\mfa_p(\theta) := - 2 \sum_{k=1}^p {\phi}'_{2p+1}(p+1-k) \sin(k \theta);
		\end{equation*}
		\item[\mylabel{eq:Sdistr}{(c)}]
		$\{\frac{1}{n} S^p_n\}_n \sim_{\lambda}(\mfs_p,[-\pi,\pi])$, where the symbol $\mfs_p$ is given by
		\begin{equation*}
		\mfs_p(\theta) := - {\phi}''_{2p+1}(p+1) - 2 \sum_{k=1}^p {\phi}''_{2p+1}(p+1-k) \cos(k \theta).
		\end{equation*}
	\end{enumerate}
\end{theorem}
\begin{proof}
	Let $m:=n+p-2$. By direct inspection, we observe that
	\begin{equation*}
	n M^p_n = T_m( \mfm_p) + R_m^{\rm M},
	\quad
	-{\rm i} A^p_n =  T_m( \mfa_p) + R_m^{\rm A},
	\quad
	\frac{1}{n} S^p_n =  T_m( \mfa_p) + R_m^{\rm S},
	\end{equation*}
	with $\rank(R_m^{\rm M})$, $\rank(R_m^{\rm A})$, $\rank(R_m^{\rm S})\leq 4p$.
	Now, by {\bf GLT3}, the sequences $\{T_m(\mfm_p)\}_n$,  $\{T_m(\mfa_p)\}_n$,
	$\{T_m(\mfs_p)\}_n$ are GLT sequences with symbols $\mfm_p$, $\mfa_p$, $\mfs_p$, respectively,
	while by {\bf GLT4} the sequences  $\{R_{m}^{\rm M}\}_n$,  $\{R_{m}^{\rm A}\}_n$,
	$\{R_{m}^{\rm M}\}_n$ are GLT sequences with symbol $0$, due to their low rank.
	As a consequence of {\bf GLT2}, also the sequences $\{n M^p_n\}_n$, $\{-{\rm i} A^p_n\}_n$, $\{\frac{1}{n} S^p_n\}_n$ are GLT sequences with symbols $\mfm_p$, $\mfa_p$, $\mfs_p$, respectively. Since all the involved matrices are Hermitian, {\bf GLT1} concludes the proof.
\end{proof}

The symbols $\mfm_p$, $\mfa_p$, and $\mfs_p$ enjoy the following properties (see \cite{GaroniNM,SINUM-IgA,MC-collocation,lusin}).
\begin{lemma} \label{lem:h-f-e}
	It holds that
	\begin{enumerate}
		\item[\mylabel{eq:mp}{(a)}] $\mfm_p(\theta)>0$ for all $\theta\in[-\pi,\pi]$ and $\left(2/{\pi}\right)^{2p}\leq\mfm_{p}(\theta)\leq\mfm_{p}(0)=1$;
		\item[\mylabel{eq:sp}{(b)}] $\mfs_p(\theta)=\mfm_{p-1}(\theta)(2-2\cos(\theta))\ge0$ for all $\theta\in[-\pi,\pi]$ and $\mfs_p(\pi)\leq 2^{2-p}{\max_{\theta\in[0,\pi]}\mfs_p(\theta)}$;
		\item[\mylabel{eq:ap}{(c)}] $\mfm_p(\theta)\mfs_p(\theta)\geq(\mfa_p(\theta))^2$ for all $\theta\in[-\pi,\pi]$, with only equality at $\theta=0$.
	\end{enumerate}
\end{lemma}

As a result of Lemma~\ref{lem:h-f-e}, the function $\mfs_p(\theta)$ has a unique zero of order $2$ at $\theta=0$ (like the function $2-2\cos(\theta)$). In addition, there is an exponential decay in $p$ at $\theta=\pm\pi$, implying that ${\mfs_p(\theta)}$ becomes very small and behaves like a ``numerical zero'' at $\theta=\pm\pi$ for large $p$.

\section{IgA discretization of the curl-div operator}\label{sec:discr}

In this section, we detail the IgA discretization of \eqref{eq:curl-div} considering homogeneous Dirichlet conditions on the boundary of $\Omega=(0,1)^d$, $d=2,3$. More precisely, we describe the corresponding matrices obtained by using tensor-product B-splines in 2D and 3D.

\subsection{2D case}

In the 2D case, the approximation space is given by
\begin{equation*}
\mathbb{V}_h:=\begin{bmatrix}
\mathbb{S}^{p,p}_{n_1,n_2}\\[0.2cm] \mathbb{S}^{p,p}_{n_1,n_2}
\end{bmatrix}
=\vspan\left\{\ppsi^{p,1}_{i_1,i_2},\ppsi^{p,2}_{j_1,j_2}:i_l,j_l=2,\ldots,n_l+p-1; \ l=1,2\right\},
\end{equation*}
where $\mathbb{S}^{p,p}_{n_1,n_2}$ is defined in \eqref{eq:tensor2D} and
\begin{equation*}
\ppsi^{p,1}_{i_1,i_2}:=\begin{bmatrix}
N^{p}_{i_1}(x_1)N^{p}_{i_2}(x_2) \\[0.2cm] 0
\end{bmatrix}, \quad
\ppsi^{p,2}_{j_1,j_2}:=\begin{bmatrix}
0 \\[0.2cm] N^{p}_{j_1}(x_1)N^{p}_{j_2}(x_2)
\end{bmatrix}.
\end{equation*}
Expanding $\uu_h$ in terms of the B-spline representation,
\begin{equation*}\label{eq:u_expansion}
\uu_h=
\begin{bmatrix}
u^1_h \\[0.2cm] u^2_h
\end{bmatrix}=
\begin{bmatrix}
\sum_{i_1=2}^{n_1+p-1}\sum_{i_2=2}^{n_2+p-1} u^1_{i_1,i_2} N^{p}_{i_1}(x_1)N^{p}_{i_2}(x_2) \\[0.2cm]
\sum_{j_1=2}^{n_1+p-1}\sum_{j_2=2}^{n_2+p-1} u^2_{j_1,j_2} N^{p}_{j_1}(x_1)N^{p}_{j_2}(x_2)
\end{bmatrix},
\end{equation*}
and choosing $\vv_h=\ppsi^{p,1}_{k_1,k_2}$, we obtain
\begin{align}\label{eq:curl1}
\notag(\Curl\uu_h,\Curl \ppsi^{p,1}_{k_1,k_2})&=
-\sum_{j_1=2}^{n_1+p-1}\sum_{j_2=2}^{n_2+p-1} u^2_{j_1,j_2} \int_{\Omega} \left(N^{p}_{j_1}(x_1)\right)'N^{p}_{j_2}(x_2)N^{p}_{k_1}(x_1) \left(N^{p}_{k_2}(x_2)\right)'  {\rm d}x_1 {\rm d}x_2 \\
\notag& \quad +
\sum_{i_1=2}^{n_1+p-1}\sum_{i_2=2}^{n_2+p-1} u^1_{i_1,i_2} \int_{\Omega} N^{p}_{i_1}(x_1)\left(N^{p}_{i_2}(x_2)\right)'
N^{p}_{k_1}(x_1)\left(N^{p}_{k_2}(x_2)\right)'{\rm d}x_1 {\rm d}x_2\\
&=-\sum_{j_1=2}^{n_1+p-1}\sum_{j_2=2}^{n_2+p-1}u^2_{j_1,j_2}((A^{p}_{n_1})^T\otimes A^{p}_{n_2})_{k_1,k_2,j_1,j_2}+ \sum_{i_1=2}^{n_1+p-1}\sum_{i_2=2}^{n_2+p-1}u^1_{i_1,i_2}(M^{p}_{n_1}\otimes S^{p}_{n_2})_{k_1,k_2,i_1,i_2},
\end{align}
where $M^{p}_{n_1}$, $A^{p}_{n_1}$, $A^{p}_{n_2}$ and $S^{p}_{n_2}$ are defined in \eqref{eq:matr_M},  \eqref{eq:matr_A} and  \eqref{eq:matr_S}.
Equation \eqref{eq:curl1} can be compactly expressed as
\begin{equation}\label{eq:eq_1}
M_{n_1}^{p}\otimes S_{n_2}^p\,\ruu^1 -(A^{p}_{n_1})^T\otimes A^{p}_{n_2}\,\ruu^2,
\end{equation}
with $\ruu^1:=[u^1_{2,2},\ldots,u^1_{n_1+p-1,n_2+p-1}]^T$ and $\ruu^2:=[u^2_{2,2},\ldots,u^2_{n_1+p-1,n_2+p-1}]^T$.
On the other hand, if we choose $\vv_h=\ppsi^{p,2}_{l_1,l_2}$, then we obtain
\begin{align}\label{eq:curl2}
\notag(\Curl \uu_h,\Curl \ppsi^{p,2}_{l_1,l_2})&=
\sum_{j_1=2}^{n_1+p-1}\sum_{j_2=2}^{n_2+p-1} u^2_{j_1,j_2} \int_{\Omega} \left(N^{p}_{j_1}(x_1)\right)'N^{p}_{j_2}(x_2)\left(N^{p}_{l_1}(x_1)\right)' N^{p}_{l_2}(x_2)  \,{\rm d}x_1 {\rm d}x_2 \\
\notag & \quad
-\sum_{i_1=2}^{n_1+p-1}\sum_{i_2=2}^{n_2+p-1} u^1_{i_1,i_2} \int_{\Omega} N^{p}_{i_1}(x_1)\left(N^{p}_{i_2}(x_2)\right)'
\left(N^{p}_{l_1}(x_1)\right)'N^{p}_{l_2}(x_2)\,{\rm d}x_1 {\rm d}x_2\\
&=\sum_{j_1=2}^{n_1+p-1}\sum_{j_2=2}^{n_2+p-1}u^2_{j_1,j_2}(S^{p}_{n_1}\otimes M^{p}_{n_2})_{l_1,l_2,j_1,j_2}- \sum_{i_1=2}^{n_1+p-1}\sum_{i_2=2}^{n_2+p-1}u^1_{i_1,i_2}(A^{p}_{n_1}\otimes (A^{p}_{n_2})^T)_{l_1,l_2,i_1,i_2},
\end{align}
where $M^{p}_{n_2}$, $A^{p}_{n_1}$, $A^{p}_{n_2}$ and $S^{p}_{n_1}$ are defined in \eqref{eq:matr_M},  \eqref{eq:matr_A} and \eqref{eq:matr_S}.
We can compactly express equation \eqref{eq:curl2} as
\begin{equation}\label{eq:eq_2}
S_{n_1}^{p}\otimes M_{n_2}^{p}\,\ruu^2 -A^{p}_{n_1}\otimes (A^{p}_{n_2})^T\,\ruu^1.
\end{equation}
Putting together equations \eqref{eq:eq_1} and \eqref{eq:eq_2}, we arrive at the following $2\times2$ block matrix
\begin{equation}\label{eq:matr_Acurl}
\mathcal{A}^{p,{\rm curl}}_{\nn}:=\begin{bmatrix}
M^{p}_{n_1}\otimes S^p_{n_2} & -(A^{p}_{n_1})^T\otimes A^p_{n_2} \\
-A^{p}_{n_1}\otimes (A^p_{n_2})^T & S^{p}_{n_1}\otimes M^{p}_{n_2}
\end{bmatrix}.
\end{equation}
Such a matrix is the result of the IgA discretization of the curl-curl operator $(\Curl \cdot, \Curl\cdot)$ appearing in \eqref{eq:curl-div}. Similarly, the matrix arising from the IgA discretization of the divergence term $(\Div,\Div)$ can be written as
\begin{equation}\label{eq:matr_Adiv}
\mathcal{A}^{p,{\rm div}}_{\nn}:=\begin{bmatrix}
S^{p}_{n_1}\otimes M^p_{n_2} & A^p_{n_1}\otimes (A^p_{n_2})^T \\
(A^p_{n_1})^T\otimes A^p_{n_2} & M^{p}_{n_1}\otimes S^{p}_{n_2}
\end{bmatrix}.
\end{equation}
Therefore, the complete IgA discretization of \eqref{eq:curl-div} for $d=2$ 
leads to the following $2\times2$ block symmetric coefficient matrix
\begin{equation}\label{eq:matr_Amu}
\mathcal{A}^{p,\alpha,\beta}_{\nn}:=\alpha\mathcal{A}^{p,{\rm curl}}_{\nn}
+ \beta\mathcal{A}^{p,{\rm div}}_{\nn}=\alpha\begin{bmatrix}
M^{p}_{n_1}\otimes S^p_{n_2} & -(A^{p}_{n_1})^T\otimes A^p_{n_2} \\
-A^{p}_{n_1}\otimes (A^p_{n_2})^T & S^{p}_{n_1}\otimes M^{p}_{n_2}
\end{bmatrix}+\beta\begin{bmatrix}
S^{p}_{n_1}\otimes M^p_{n_2} & A^p_{n_1}\otimes (A^p_{n_2})^T \\
(A^p_{n_1})^T\otimes A^p_{n_2} & M^{p}_{n_1}\otimes S^{p}_{n_2}
\end{bmatrix}.
\end{equation}

\subsection{3D case}
In the 3D case, the approximation space is given by
\begin{eqnarray*}
	\mathbb{V}_h:=\begin{bmatrix}
		\mathbb{S}^{p,p,p}_{n_1,n_2,n_3}\\[0.2cm] \mathbb{S}^{p,p,p}_{n_1,n_2,n_3}\\[0.2cm] \mathbb{S}^{p,p,p}_{n_1,n_2,n_3}
	\end{bmatrix}
	=\vspan\left\{\ppsi^{p,1}_{i_1,i_2,i_3},\ppsi^{p,2}_{j_1,j_2,j_3},\ppsi^{p,1}_{k_1,k_2,k_3}: i_l,j_l,k_l=2,\ldots,n_l+p-1; \ l=1,2,3\right\},
\end{eqnarray*}
where $\mathbb{S}^{p,p,p}_{n_1,n_2,n_3}$ is defined in \eqref{eq:tensor2D} and
\begin{align*}
\ppsi^{p,1}_{i_1,i_2,i_3}:=\left[\begin{smallmatrix}
N^{p}_{i_1}(x_1)N^{p}_{i_2}(x_2)N^{p}_{i_3}(x_3)\\ 0 \\ 0
\end{smallmatrix}\right], \quad
\ppsi^{p,2}_{j_1,j_2,j_3}:=\left[\begin{smallmatrix}
0 \\ N^{p}_{j_1}(x_1)N^{p}_{j_2}(x_2)N^{p}_{j_3}(x_3)\\0
\end{smallmatrix}\right], \quad
\ppsi^{p,3}_{k_1,k_2,k_3}:=\left[\begin{smallmatrix}
0 \\ 0 \\ N^{p}_{k_1}(x_1)N^{p}_{k_2}(x_2)N^{p}_{k_3}(x_3)
\end{smallmatrix}\right].
\end{align*}
Following the same reasoning as in the 2D case, we obtain here the $3\times3$ block symmetric matrix
\begin{align}\label{eq:matr_Amu3D}
\begin{split}
&\mathcal{A}^{p,\alpha,\beta}_{\nn}:=\\
&\alpha\left[\begin{smallmatrix}
M^{p}_{n_1}\otimes M^{p}_{n_2}\otimes S^p_{n_3}+
M^{p}_{n_1}\otimes S^{p}_{n_2}\otimes M^p_{n_3}
& -(A^{p}_{n_1})^T\otimes A^p_{n_2} \otimes M^p_{n_3}
& -(A^{p}_{n_1})^T\otimes M^p_{n_2} \otimes A^p_{n_3} \\
-A^{p}_{n_1}\otimes (A^p_{n_2})^T\otimes M^p_{n_3}
& 
S^{p}_{n_1}\otimes M^{p}_{n_2}\otimes M^p_{n_3}+
M^{p}_{n_1}\otimes M^{p}_{n_2}\otimes S^p_{n_3}
& -M^{p}_{n_1}\otimes (A^p_{n_2})^T \otimes A^p_{n_3}\\
-A^{p}_{n_1}\otimes M^p_{n_2} \otimes (A^p_{n_3})^T
& -M^{p}_{n_1}\otimes A^p_{n_2} \otimes (A^p_{n_3})^T
& 
S^{p}_{n_1}\otimes M^{p}_{n_2}\otimes M^p_{n_3}+
M^{p}_{n_1}\otimes S^{p}_{n_2}\otimes M^p_{n_3}
\end{smallmatrix}\right]\\
&+\beta\left[\begin{smallmatrix}
S^{p}_{n_1}\otimes M^p_{n_2}\otimes M^p_{n_3} & A^p_{n_1}\otimes (A^p_{n_2})^T\otimes M^p_{n_3}
& A^p_{n_1}\otimes  M^p_{n_2}\otimes (A^p_{n_3})^T\\
(A^{p}_{n_1})^T\otimes A^p_{n_2} \otimes M^p_{n_3} & M^{p}_{n_1}\otimes S^{p}_{n_2} \otimes M^p_{n_3}
& M^{p}_{n_1}\otimes A^p_{n_2} \otimes (A^p_{n_3})^T\\
(A^{p}_{n_1})^T\otimes M^p_{n_2} \otimes A^p_{n_3} & M^{p}_{n_1}\otimes (A^p_{n_2})^T \otimes A^p_{n_3}
&M^{p}_{n_1}\otimes M^{p}_{n_2} \otimes S^p_{n_3}
\end{smallmatrix}\right],
\end{split}
\end{align}
where $M^{p}_{n_i}$, $A^{p}_{n_i}$ and $S^{p}_{n_i}$ are defined in \eqref{eq:matr_M}, \eqref{eq:matr_A} and \eqref{eq:matr_S}.

\section{Spectral analysis of matrices $\mathcal{A}^{p,\alpha,\beta}_{\nn}$}\label{sec:spectral}

We are interested in the spectral behavior of the matrices $\mathcal{A}^{p,\alpha,\beta}_{\nn}$  defined in \eqref{eq:matr_Amu} and \eqref{eq:matr_Amu3D} for $d=2,3$, respectively.
Therefore, in this section, we analyze the spectral distribution of the matrix-sequences $\bigl\{n^{d-2}\mathcal{A}^{p,\alpha,\beta}_{\nn}\bigr\}_{\nn}$ where $\nn=n{\bm \nu}$ for increasing $n\in\mathbb{N}$ and fixed ${\bm\nu}\in\mathbb{Q}^d$ with $d=2,3$. In particular, we prove that
\begin{equation}\label{eq:distr}
\left\{n^{d-2}\mathcal{A}^{p,\alpha,\beta}_{\nn}\right\}_{\nn} \sim_\lambda (f^{p,\alpha,\beta}, [-\pi,\pi]^d), \quad d=2,3,
\end{equation}
for a specific $d\times d$ matrix-valued function $f^{p,\alpha,\beta}$; see Theorems \ref{thm:distr2D} and \ref{thm:distr3D}. As in Section \ref{sec:discr}, we start with the 2D case and then extend all our spectral findings to the 3D case. Finally, we end with some numerical experiments.

\subsection{2D case}\label{sub:distr2D}
Let us set $\mm:=(m_1,m_2):=(n_1+p-2,n_2+p-2)$ and $\nn:=(n_1,n_2)$, so the dimension of $\mathcal{A}^{p,\alpha,\beta}_{\nn}$ in \eqref{eq:matr_Amu} equals $2m_1m_2$.
In order to show the spectral distribution in \eqref{eq:distr} for $d=2$, we first prove that the matrix $\mathcal{A}^{p,\alpha,\beta}_{\nn}$ can be seen as a permutation of a $2$-level block Toeplitz matrix up to a low-rank correction. Then, by applying the GLT theory, we arrive at the following spectral result.

\begin{theorem}\label{thm:distr2D}
	For $d=2$ and $\nn:=(n_1,n_2)=(n\nu_1,n\nu_2)$, the symbol in \eqref{eq:distr} is given by
	\begin{equation}\label{eq:f2D}
	f^{p,\alpha,\beta}:[-\pi,\pi]^2\rightarrow \mathbb{C}^{2\times2}, \quad
	f^{p,\alpha,\beta}(\theta_1, \theta_2) :=\begin{bmatrix}
	f^{p,\alpha,\beta}_{1,1}(\theta_1, \theta_2) & f^{p,\alpha,\beta}_{1,2}(\theta_1, \theta_2)\\
	f^{p,\alpha,\beta}_{2,1}(\theta_1, \theta_2) & f^{p,\alpha,\beta}_{2,2}(\theta_1, \theta_2)
	\end{bmatrix},
	\end{equation}
	where
	\begin{align*}
	f^{p,\alpha,\beta}_{1,1}(\theta_1, \theta_2)&:=\alpha\frac{\nu_2}{\nu_1}\mfm_{p}(\theta_1)\mfs_{p}(\theta_2)+ \beta\frac{\nu_1}{\nu_2}\mfs_{p}(\theta_1)\mfm_{p}(\theta_2),\\
	f^{p,\alpha,\beta}_{1,2}(\theta_1, \theta_2)&:=f^{p,\alpha,\beta}_{2,1}(\theta_1, \theta_2):=-(\alpha-\beta)\mfa_{p}(\theta_1)\mfa_{p}(\theta_2),\\
	f^{p,\alpha,\beta}_{2,2}(\theta_1, \theta_2)&:=\alpha\frac{\nu_1}{\nu_2}\mfs_{p}(\theta_1)\mfm_{p}(\theta_2)+\beta
	\frac{\nu_2}{\nu_1}\mfm_{p}(\theta_1)\mfs_{p}(\theta_2).
	\end{align*}
\end{theorem}
\begin{proof}
	From 
	Theorem \ref{thm:spline-distr}
	we immediately deduce that the blocks of the matrices $\mathcal{A}^{{\rm curl}}_{\nn}$ and $\mathcal{A}^{{\rm div}}_{\nn}$ in \eqref{eq:matr_Acurl}--\eqref{eq:matr_Adiv} behave spectrally like
	\begin{align*}
	\left\{(\mathcal{A}^{p,{\rm curl}}_{\nn})^{(1,1)}\right\}_{\nn}
	=\left\{(\mathcal{A}^{p,{\rm div}}_{\nn})^{(2,2)}\right\}_{\nn}
	=\left\{M^{p}_{n_1}\otimes S^{p}_{n_2}\right\}_{\nn}&\sim_{\lambda}\left(\frac{\nu_2}{\nu_1}\mfm_{p}(\theta_1)\mfs_{p}(\theta_2),[-\pi,\pi]^2\right),
	\\
	\left\{(\mathcal{A}^{{p,\rm curl}}_{\nn})^{(1,2)}\right\}_{\nn}
	=\left\{-(\mathcal{A}^{p,{\rm div}}_{\nn})^{(2,1)}\right\}_{\nn}
	=\left\{-(A^{p}_{n_1})^T\otimes A^{p}_{n_2}\right\}_{\nn}&\sim_{\lambda}\left(-\mfa_{p}(\theta_1)\mfa_{p}(\theta_2),[-\pi,\pi]^2\right),
	\\
	\left\{(\mathcal{A}^{p,{\rm curl}}_{\nn})^{(2,1)}\right\}_{\nn}
	=\left\{-(\mathcal{A}^{p,{\rm div}}_{\nn})^{(1,2)}\right\}_{\nn}
	=\left\{-A^{p}_{n_1}\otimes (A^{p}_{n_2})^T\right\}_{\nn}&\sim_{\lambda}\left(-\mfa_{p}(\theta_1)\mfa_{p}(\theta_2),[-\pi,\pi]^2\right),
	\\
	\left\{(\mathcal{A}^{p,{\rm curl}}_{\nn})^{(2,2)}\right\}_{\nn}
	=\left\{(\mathcal{A}^{p,{\rm div}}_{\nn})^{(1,1)}\right\}_{\nn}
	=\left\{S^{p}_{n_1}\otimes M^{p}_{n_2}\right\}_{\nn}&\sim_{\lambda}\left(\frac{\nu_1}{\nu_2}\mfs_{p}(\theta_1)\mfm_{p}(\theta_2),[-\pi,\pi]^2\right).
	\end{align*}
	Then, recalling the discussion in Section \ref{sub:IgAmatr}, it is easy to see that each block of $\mathcal{A}^{p,\alpha,\beta}_{\nn}$ in \eqref{eq:matr_Amu} is a $2$-level Toeplitz matrix up to a low-rank correction. More precisely,
	$$
	(\mathcal{A}^{p,\alpha,\beta}_{\nn})^{(i,j)}=T_{\mm}(f^{p,\alpha,\beta}_{i,j})+R_{\mm}^{(i,j)},
	\quad i,j=1,2,
	$$
	where $T_{\mm}(f^{p,\alpha,\beta}_{ij})$ are $2$-level Toeplitz matrices generated
	by the functions $f^{p,\alpha,\beta}_{i,j}$ in \eqref{eq:f2D}.
	We are now ready to understand that the full matrix $\mathcal{A}^{p,\alpha,\beta}_{\nn}$ in \eqref{eq:matr_Amu} is a permutation of a $2$-level block Toeplitz matrix up to a low-rank correction. Indeed,
	\begin{equation*}\label{permut}
	\Pi_{\mm}\,\mathcal{A}^{p,\alpha,\beta}_{\nn}\,\Pi_{\mm}^{T}
	=\Pi_{\mm}\begin{bmatrix}
	T_{\mm}(f^{p,\alpha,\beta}_{1,1})+R_{\mm}^{(1,1)} &
	T_{\mm}(f^{p,\alpha,\beta}_{1,2})+R_{\mm}^{(1,2)}\\
	T_{\mm}(f^{p,\alpha,\beta}_{2,1})+R_{\mm}^{(2,1)} &
	T_{\mm}(f^{p,\alpha,\beta}_{22})+R_{\mm}^{(2,2)}
	\end{bmatrix}\Pi_{\mm}^{T}= T_{\mm}(f^{p,\alpha,\beta}) + R_{\mm},
	\end{equation*}
	where $\Pi_{\mm}$ is a permutation matrix of size $2m_1m_2$ and $R_{\mm}$ is a low-rank matrix whose rank is $o(2m_1m_2)$.
	Note that the first $m_1m_2$ columns of $\Pi_{\mm}$ are the odd columns of the identity matrix of size $2m_1m_2$, while the remaining ones are the even columns of the same identity matrix.
	
	Finally, we apply the GLT theory to conclude the spectral distribution in \eqref{eq:distr}. More precisely, here we use the extension of properties {\bf GLT1}--{\bf GLT4} given in Section \ref{sub:GLT} to the case where the symbol is a matrix-valued function (see \cite{GLT-bookIII} for a detailed discussion on block GLTs). For the sake of simplicity, we refer to these properties again as {\bf GLT1}--{\bf GLT4}.
	By {\bf GLT3}, $\{T_{\mm}(f^{p,\alpha,\beta})\}_{\nn}$ is a GLT sequence with symbol $f^{p,\alpha,\beta}$, while $\{R_{\mm}\}_{\nn}$ is a GLT sequence with symbol $0$, due to {\bf GLT4}.
	Then, by {\bf GLT2}, we deduce that $\{\Pi \mathcal{A}^{p,\alpha,\beta}_{\nn}\Pi^T\}_{\nn}$ is a GLT sequence with the same symbol as $\{T_{\mm}(f^{p,\alpha,\beta})\}_{\nn}$, that is $f^{p,\alpha,\beta}$.
	Since $f^{p,\alpha,\beta}$ is a Hermitian matrix-valued function, it follows from {\bf GLT1} that $\{\Pi_{\mm}\,\mathcal{A}^{p,\alpha,\beta}_{\nn}\,\Pi_{\mm}^T\}_{\nn}$ is distributed as $f^{p,\alpha,\beta}$ in the sense of the eigenvalues.
	This is also true for $\{\mathcal{A}^{p,\alpha,\beta}_{\nn}\}_{\nn}$ because the matrices $\Pi_{\mm}\, \mathcal{A}^{p,\alpha,\beta}_{\nn}\,\Pi_{\mm}^T$ and $\mathcal{A}^{p,\alpha,\beta}_{\nn}$ are similar for all $p,\alpha,\beta, \nn$.
\end{proof}

The symbol $f^{p,\alpha,\beta}$ defined in \eqref{eq:f2D} is a $2\times2$ matrix-valued function, so we have to study its two eigenvalue functions:
\begin{align*}
\lambda_{1}(f^{p,\alpha,\beta})&=\frac{1}{2}\left((f^{p,\alpha,\beta}_{1,1}+f^{p,\alpha,\beta}_{2,2})-\sqrt{(f^{p,\alpha,\beta}_{1,1}-f^{p,\alpha,\beta}_{22})^2+4f^{p,\alpha,\beta}_{1,2}f^{p,\alpha,\beta}_{2,1}}\right), \\
\lambda_{2}(f^{p,\alpha,\beta})&=\frac{1}{2}\left((f^{p,\alpha,\beta}_{1,1}+f^{p,\alpha,\beta}_{2,2})+\sqrt{(f^{p,\alpha,\beta}_{1,1}-f^{p,\alpha,\beta}_{22})^2+4f^{p,\alpha,\beta}_{1,2}f^{p,\alpha,\beta}_{2,1}}\right).
\end{align*}
It is easy to check that
\begin{align*}
f^{p,\alpha,\beta}_{1,1}(\theta_1,\theta_2)+f^{p,\alpha,\beta}_{2,2}(\theta_1,\theta_2)&=(\alpha+\beta)\laplacian_p^+(\theta_1,\theta_2),\quad \\
f^{p,\alpha,\beta}_{1,1}(\theta_1,\theta_2)-f^{p,\alpha,\beta}_{2,2}(\theta_1,\theta_2)&=(\alpha-\beta)\laplacian_p^-(\theta_1,\theta_2),\quad \\
f^{p,\alpha,\beta}_{1,2}\,(\theta_1,\theta_2)f^{p,\alpha\beta}_{2,1}(\theta_1,\theta_2)
&=(\alpha-\beta)^2\mathfrak{a}^2_p(\theta_1)\mathfrak{a}^2_p(\theta_2),
\end{align*}
with
\begin{align*}
\laplacian_p^\pm(\theta_1,\theta_2)&:=\frac{\nu_2}{\nu_1}\mfm_{p}(\theta_1)\mfs_{p}(\theta_2)\pm \frac{\nu_1}{\nu_2}\mfs_{p}(\theta_1)\mfm_{p}(\theta_2).
\end{align*}
This results in
\begin{align}
\lambda_{1}(f^{p,\alpha,\beta}(\theta_1,\theta_2))=\frac{1}{2}\left((\alpha+\beta)\laplacian_p^+(\theta_1,\theta_2)-
|\alpha-\beta|\sqrt{(\laplacian_p^-(\theta_1,\theta_2))^2+4\mfa_p^2(\theta_1)\mfa_p^2(\theta_2)}\right),
\label{eq:expres:eig:1} \\
\lambda_{2}(f^{p,\alpha,\beta}(\theta_1,\theta_2))=\frac{1}{2}\left((\alpha+\beta)\laplacian_p^+(\theta_1,\theta_2)+
|\alpha-\beta|\sqrt{(\laplacian_p^-(\theta_1,\theta_2))^2+4\mfa_p^2(\theta_1)\mfa_p^2(\theta_2)}\right).
\label{eq:expres:eig:2}
\end{align}
Note that $\laplacian_p^+(\theta_1,\theta_2)$ coincides with the symbol of the 2D Laplace operator, denoted by $\laplacian_p(\theta_1,\theta_2)$, obtained after discretization by means of tensor-product B-splines (see \cite{GaroniNM}).

When $\alpha=\beta$, the eigenvalue functions in \eqref{eq:expres:eig:1}--\eqref{eq:expres:eig:2} are a multiple of the 2D Laplacian symbol, i.e.,
\begin{equation*}
\lambda_{1}(f^{p,\alpha,\alpha}(\theta_1,\theta_2))=
\lambda_{2}(f^{p,\alpha,\alpha}(\theta_1,\theta_2))=\alpha \laplacian_p(\theta_1,\theta_2).
\end{equation*}
We also expect that if $|\alpha-\beta|\ll 1$ then both eigenvalue functions approximately behave like multiples of $\laplacian_p(\theta_1,\theta_2)$. Actually, a similar behavior is observed for general $\alpha$ and $\beta$ as stated in the following theorem.
\begin{theorem}\label{thm:bounds2D}
	We have
	\begin{equation}\label{eq:behavior_eig2D}
	0\leq \min(\alpha,\beta)\,\laplacian_p(\theta_1,\theta_2)
	\leq\lambda_{1}(f^{p,\alpha,\beta}(\theta_1,\theta_2))\leq\lambda_{2}(f^{p,\alpha,\beta}(\theta_1,\theta_2))
	\leq \max(\alpha,\beta)\,\laplacian_p(\theta_1,\theta_2).
	\end{equation}
\end{theorem}
\begin{proof}
	Thanks to items \ref{eq:mp} and \ref{eq:sp} of Lemma \ref{lem:h-f-e}, we obtain
	\begin{equation} \label{eq:lapl2D}
	\laplacian_p(\theta_1,\theta_2):=\frac{\nu_2}{\nu_1}\mfm_{p}(\theta_1)\mfs_{p}(\theta_2)+ \frac{\nu_1}{\nu_2}\mfs_{p}(\theta_1)\mfm_{p}(\theta_2)\geq 0.
	\end{equation}
	Moreover, by item \ref{eq:ap} of the same lemma, we have
	$\mfs_{p}(\theta)\mfm_{p}(\theta)-\mathfrak{a}^2_{p}(\theta)\geq 0.$
	Thus,
	\begin{align*}
	(\laplacian_p^-(\theta_1,\theta_2))^2 &+ 4\mfa_p^2(\theta_1)\mfa_p^2(\theta_2)
	\\
	&=\left(\frac{\nu_2}{\nu_1}\right)^2\mathfrak{m}^2_{p}(\theta_1)\mathfrak{s}^2_{p}(\theta_2)+ \left(\frac{\nu_1}{\nu_2}\right)^2\mathfrak{s}^2_{p}(\theta_1)\mathfrak{m}^2_{p}(\theta_2)-
	2\mfm_{p}(\theta_1)\mfs_{p}(\theta_1)\mfm_{p}(\theta_2)\mfs_{p}(\theta_2)+4\mfa_p^2(\theta_1)\mfa_p^2(\theta_2)
	\\
	&\leq\left(\frac{\nu_2}{\nu_1}\right)^2\mathfrak{m}^2_{p}(\theta_1)\mathfrak{s}^2_{p}(\theta_2) + \left(\frac{\nu_1}{\nu_2}\right)^2\mathfrak{s}^2_{p}(\theta_1)\mathfrak{m}^2_{p}(\theta_2)+
	2\mfm_{p}(\theta_1)\mfs_{p}(\theta_1)\mfm_{p}(\theta_2)\mfs_{p}(\theta_2)= \laplacian_p(\theta_1,\theta_2)^2.
	\end{align*}
	Suppose now $\alpha\geq \beta$. We deduce from the previous inequality in combination with \eqref{eq:expres:eig:1} that
	\begin{align*}
	\lambda_{1}(f^{p,\alpha,\beta}(\theta_1,\theta_2))
	&\geq\frac{1}{2}\left((\alpha+\beta)\laplacian_p(\theta_1,\theta_2)-
	(\alpha-\beta)\laplacian_p(\theta_1,\theta_2)\right)=\beta \laplacian_p(\theta_1,\theta_2),
	\end{align*}
	and with \eqref{eq:expres:eig:2} that
	\begin{align*}
	\lambda_{2}(f^{p,\alpha,\beta}(\theta_1,\theta_2))
	&\leq\frac{1}{2}\left((\alpha+\beta)\laplacian_p(\theta_1,\theta_2)+
	(\alpha-\beta)\laplacian_p(\theta_1,\theta_2)\right)=\alpha \laplacian_p(\theta_1,\theta_2).
	\end{align*}
	Similar bounds hold in case $\beta\geq\alpha$.
	Finally, from \eqref{eq:expres:eig:1}--\eqref{eq:expres:eig:2} it follows $\lambda_{1}(f^{p,\alpha,\beta})\leq\lambda_{2}(f^{p,\alpha,\beta})$, and consequently we arrive at \eqref{eq:behavior_eig2D}.
\end{proof}

\begin{remark} \label{rmk:sharp2D}
	The bounds in \eqref{eq:behavior_eig2D} are not just bounds, but provide a quite precise description of the two eigenvalue functions. Indeed, the errors
	\begin{equation*}
	| \lambda_{1}(f^{p,\alpha,\beta}(\theta_1,\theta_2))-\min(\alpha,\beta)\,\laplacian_p(\theta_1,\theta_2)|,
	\quad | \lambda_{2}(f^{p,\alpha,\beta}(\theta_1,\theta_2))-\max(\alpha,\beta)\,\laplacian_p(\theta_1,\theta_2)|
	\end{equation*}
	are small as can be observed in our numerical experiments (see Section~\ref{sub:numer_symbol}).
	They could be estimated by using the bounds for $\mfm_p$, $\mfa_p$, and $\mfs_p$ in \cite[Lemmas 3.4--3.6]{MC-collocation}, keeping in mind the relation between the Galerkin and collocation symbols \cite[Footnote 2]{lusin}.
\end{remark}

Thanks to Theorem \ref{thm:bounds2D} and the properties of the 2D Laplacian symbol (see, e.g., \cite{SINUM-IgA,serra2} and also Lemma \ref{lem:h-f-e}), we immediately deduce that the eigenvalue functions $\lambda_i(f^{p,\alpha,\beta})$, $i=1,2$ have the following vanishing behavior.
\begin{corollary} \label{cor:zeros2D}
	When $\alpha,\beta>0$, it holds that for $i=1,2$,
	\begin{enumerate}
		\item[\mylabel{zero-an2D}{(a)}] $\lambda_i(f^{p,\alpha,\beta}(\theta_1,\theta_2))$ has a unique zero of order $2$ at $(\theta_1,\theta_2)=(0,0)$;
		\item[\mylabel{zero-nm2D}{(b)}] $\lambda_i(f^{p,\alpha,\beta}(\theta_1,\theta_2))$ presents an exponential decay in $p$ at all points $(\pm\pi,\theta_2)$ and $(\theta_1,\pm\pi)$, implying ``numerical zeros'' for large $p$.
	\end{enumerate}
\end{corollary}
The knowledge of the vanishing behavior of the eigenvalue functions in Corollary \ref{cor:zeros2D} is important in the design of fast iterative solvers for linear systems involving $\mathcal{A}^{p,\alpha,\beta}_{\nn}$ as coefficient matrix. This will be illustrated in Section~\ref{sub:multi}.

\subsection{3D case}\label{sub:distr3D}
We now address the spectral symbol in \eqref{eq:distr} for $d=3$ considering the matrices $\mathcal{A}^{p,\alpha,\beta}_{\nn}$ in \eqref{eq:matr_Amu3D}.

\begin{theorem}\label{thm:distr3D}
	For $d=3$ and $\nn:=(n_1,n_2,n_3)=(n\nu_1,n\nu_2,n\nu_3)$, the symbol in \eqref{eq:distr} is given by
	\begin{equation}\label{eq:f3D}
	f^{p,\alpha,\beta}:[-\pi,\pi]^3\rightarrow \mathbb{C}^{3\times3}, \quad
	f^{p,\alpha,\beta}(\theta_1, \theta_2, \theta_3):=\begin{bmatrix}
	f^{p,\alpha,\beta}_{1,1}(\theta_1, \theta_2, \theta_3) & f^{p,\alpha,\beta}_{1,2}(\theta_1, \theta_2, \theta_3) & f^{p,\alpha,\beta}_{1,3}(\theta_1, \theta_2, \theta_3)\\
	f^{p,\alpha,\beta}_{2,1}(\theta_1, \theta_2, \theta_3) & f^{p,\alpha,\beta}_{2,2}(\theta_1, \theta_2, \theta_3) & f^{p,\alpha,\beta}_{2,3}(\theta_1, \theta_2, \theta_3)\\
	f^{p,\alpha,\beta}_{3,1}(\theta_1, \theta_2, \theta_3) & f^{p,\alpha,\beta}_{3,2}(\theta_1, \theta_2, \theta_3) & f^{p,\alpha,\beta}_{3,3}(\theta_1, \theta_2, \theta_3)
	\end{bmatrix},
	\end{equation}
	where
	\begin{align*}
	f^{p,\alpha,\beta}_{1,1}(\theta_1, \theta_2, \theta_3) &:= \beta \gamma_1\mfs_{p}(\theta_1)\mfm_{p}(\theta_2)\mfm_{p}(\theta_3)
	+ \alpha \gamma_2 \mfm_{p}(\theta_1)\mfs_{p}(\theta_2)\mfm_{p}(\theta_3)
	+ \alpha \gamma_3 \mfm_{p}(\theta_1)\mfm_{p}(\theta_2)\mfs_{p}(\theta_3),\\
	f^{p,\alpha,\beta}_{1,2}(\theta_1, \theta_2, \theta_3) &:= f^{p,\alpha,\beta}_{2,1}(\theta_1, \theta_2, \theta_3) := -(\alpha-\beta)\delta_3\mfa_{p}(\theta_1)\mfa_{p}(\theta_2)\mfm_{p}(\theta_3),\\
	f^{p,\alpha,\beta}_{1,3}(\theta_1, \theta_2, \theta_3) &:= f^{p,\alpha,\beta}_{3,1}(\theta_1, \theta_2, \theta_3) := -(\alpha-\beta)\delta_2\mfa_{p}(\theta_1)\mfm_{p}(\theta_2)\mfa_{p}(\theta_3),\\
	f^{p,\alpha,\beta}_{2,2}(\theta_1, \theta_2, \theta_3) &:= \alpha \gamma_1\mfs_{p}(\theta_1)\mfm_{p}(\theta_2)\mfm_{p}(\theta_3)
	+ \beta \gamma_2\mfm_{p}(\theta_1)\mfs_{p}(\theta_2)\mfm_{p}(\theta_3)
	+ \alpha \gamma_3\mfm_{p}(\theta_1)\mfm_{p}(\theta_2)\mfs_{p}(\theta_3),\\
	f^{p,\alpha,\beta}_{2,3}(\theta_1, \theta_2, \theta_3) &:= f^{p,\alpha,\beta}_{3,2}(\theta_1, \theta_2, \theta_3) := -(\alpha-\beta)\delta_1\mfm_{p}(\theta_1)\mfa_{p}(\theta_2)\mfa_{p}(\theta_3),\\
	f^{p,\alpha,\beta}_{3,3}(\theta_1, \theta_2, \theta_3) &:= \alpha \gamma_1\mfs_{p}(\theta_1)\mfm_{p}(\theta_2)\mfm_{p}(\theta_3)
	+ \alpha \gamma_2\mfm_{p}(\theta_1)\mfs_{p}(\theta_2)\mfm_{p}(\theta_3)
	+ \beta \gamma_3\mfm_{p}(\theta_1)\mfm_{p}(\theta_2)\mfs_{p}(\theta_3),
	\end{align*}
	with
	\begin{equation*}
	\gamma_i:=\frac{\nu_i^2}{\nu_1\nu_2\nu_3}, \quad \delta_i:=\frac{1}{\nu_i}, \quad i=1,2,3.
	\end{equation*}
\end{theorem}
\begin{proof}
	The proof follows the same line of arguments as the one of Theorem \ref{thm:distr2D}, so we can omit the details. First, it can be shown that $n\mathcal{A}^{p,\alpha,\beta}_{\nn}$ is a permutation of a $3$-level block Toeplitz matrix up to a low-rank correction, and that the generating function of the Toeplitz part is the Hermitian matrix-valued function \eqref{eq:f3D}.
	Then, using {\bf GLT1}--{\bf GLT4} in their matrix-valued form exactly as done in the 2D setting, we can conclude \eqref{eq:distr} with symbol \eqref{eq:f3D}.
\end{proof}

Similar to the 2D case (Theorem \ref{thm:bounds2D}), we can formulate bounds for the eigenvalue functions of the $3\times3$ matrix-valued symbol $f^{p,\alpha,\beta}$ defined in \eqref{eq:f3D} in terms of the 3D Laplacian symbol given by (see \cite{serra2})
\begin{equation}\label{eq:lapl3D}
\laplacian_p(\theta_1,\theta_2,\theta_3):=\gamma_1\mfs_{p}(\theta_1)\mfm_{p}(\theta_2)\mfm_{p}(\theta_3)
+ \gamma_2\mfm_{p}(\theta_1)\mfs_{p}(\theta_2)\mfm_{p}(\theta_3)
+ \gamma_3\mfm_{p}(\theta_1)\mfm_{p}(\theta_2)\mfs_{p}(\theta_3).
\end{equation}
\begin{theorem}\label{thm:bounds3D}
	We have
	\begin{equation}\label{eq:behavior_eig3D}
	0 \leq \min(\alpha,\beta)\,\laplacian_p(\theta_1,\theta_2,\theta_3)
	\leq\lambda_i(f^{p,\alpha,\beta}(\theta_1,\theta_2,\theta_3))
	\leq \max(\alpha,\beta)\,\laplacian_p(\theta_1,\theta_2,\theta_3), \quad i=1,2,3.
	\end{equation}
\end{theorem}
\begin{proof}
	The bounds in \eqref{eq:behavior_eig3D} obviously hold for $\alpha=\beta$, because in this case the symbol takes the form of a diagonal matrix, i.e.,
	\begin{equation*}
	f^{p,\alpha,\alpha}(\theta_1,\theta_2,\theta_3)=\begin{bmatrix}
	\alpha \laplacian_p(\theta_1,\theta_2,\theta_3) &  &  \\
	& \alpha \laplacian_p(\theta_1,\theta_2,\theta_3) &  \\
	&  &\alpha \laplacian_p(\theta_1,\theta_2,\theta_3)
	\end{bmatrix}.
	\end{equation*}
	Note that, here and in the following, we might ignore the visualization of zero entries in a matrix when there is no confusion.
	Therefore, it remains to prove the bounds for $\alpha\neq\beta$.
	We first address the case $\alpha>\beta$.
	It suffices to show that for any $(\theta_1, \theta_2, \theta_3)\in [-\pi, \pi]^3$ we have
	\begin{equation}
	\label{eq:det_lower}
	\det(f^{p,\alpha,\beta}(\theta_1, \theta_2, \theta_3)-\lambda\, I_3)> 0, \quad \forall \lambda<\beta \laplacian_p(\theta_1,\theta_2,\theta_3),
	\end{equation}
	and
	\begin{equation}
	\label{eq:det_upper}
	\det(\lambda\, I_3 - f^{p,\alpha,\beta}(\theta_1, \theta_2, \theta_3))> 0, \quad \forall \lambda>\alpha \laplacian_p(\theta_1,\theta_2,\theta_3),
	\end{equation}
	where $I_3$ is the $3\times3$ identity matrix.
	To simplify the notation, along this proof we use
	$$
	\mfmi{i}:=\mfm_p(\theta_i), \quad \mfai{i}:=\mfa_p(\theta_i), \quad \mfsi{i}:=\mfs_p(\theta_i), \quad i=1,2,3.
	$$
	We start by showing \eqref{eq:det_lower}.
	Let
	\begin{equation}
	\label{eq:epsilon}
	\epsilon:=\frac{1}{\alpha-\beta}\bigl(\beta \laplacian_p(\theta_1,\theta_2,\theta_3)-\lambda\bigr),
	\end{equation}
	and
	\begin{align*}
	B&:=\frac{1}{\alpha-\beta}\Bigl(f^{p,\alpha,\beta}(\theta_1,\theta_2,\theta_3)-\beta \laplacian_p(\theta_1,\theta_2,\theta_3)\,I_3\Bigr)
	\\
	&=\begin{bmatrix}
	\gamma_2\mfmi{1}\mfsi{2}\mfmi{3}+\gamma_3\mfmi{1}\mfmi{2}\mfsi{3}& -\delta_3\mfai{1}\mfai{2}\mfmi{3} & -\delta_2\mfai{1}\mfmi{2}\mfai{3}
	\\
	-\delta_3\mfai{1}\mfai{2}\mfmi{3}& \gamma_1\mfsi{1}\mfmi{2}\mfmi{3}+\gamma_3\mfmi{1}\mfmi{2}\mfsi{3}  & -\delta_1\mfmi{1}\mfai{2}\mfai{3}
	\\
	-\delta_2\mfai{1}\mfmi{2}\mfai{3}& -\delta_1\mfmi{1}\mfai{2}\mfai{3}& \gamma_1\mfsi{1}\mfmi{2}\mfmi{3}+\gamma_2\mfmi{1}\mfsi{2}\mfmi{3}
	\end{bmatrix}.
	\end{align*}
	Then,
	$$
	\det(f^{p,\alpha,\beta}(\theta_1,\theta_2,\theta_3)-\lambda\, I_3)=(\alpha-\beta)^3\det(B+\epsilon\, I_3).
	$$
	We observe that for any $3\times 3$ matrix $B=[b_{i,j}]_{i,j=1}^3$ we have
	\begin{equation}
	\label{eq:det-perturbation}
	\det (B+\epsilon\, I_3)=\det (B)+\epsilon(\det(B_{1,1})+\det(B_{2,2})+\det(B_{3,3}))+\epsilon^2 (b_{1,1}+b_{2,2}+b_{3,3})+\epsilon^3,
	\end{equation}
	where $B_{i,j}$ denotes the submatrix obtained from $ B$ by removing the $i$-th row and the $j$-th column.
	From \eqref{eq:epsilon} we see that $\epsilon>0$ for all $\lambda<\beta \laplacian_p(\theta_1,\theta_2,\theta_3)$.
	In view of \eqref{eq:det-perturbation}, we are interested in proving that all the principal submatrices of $B$  have  nonnegative  determinant.
	Taking into account the inequalities for $\mfm_p$ and $\mfs_p$ in items \ref{eq:mp} and \ref{eq:sp} of Lemma~\ref{lem:h-f-e}, it is clear that the diagonal elements of $B$ are nonnegative. Moreover, a direct computation and a combined use of the inequalities for $\mfm_p$, $\mfa_p$, $\mfs_p$ in items \ref{eq:mp}--\ref{eq:ap} of Lemma~\ref{lem:h-f-e} show
	$$
	\det(B)\geq
	2\delta_1\delta_2\delta_3\mfmi{1}\mfmi{2}\mfmi{3}\bigl(\mfmi{1}\mfsi{1}\mfmi{2}\mfsi{2}\mfmi{3}\mfsi{3}-(\mfai{1}\mfai{2}\mfai{3})^2\bigr)\geq 0,
	$$
	and
	$$
	\det(B_{3,3})\geq \gamma_3\mfmi{1}\mfmi{2}\mfsi{3}
	\bigl(\gamma_1\mfsi{1}\mfmi{2}\mfmi{3}+ \gamma_2\mfmi{1}\mfsi{2}\mfmi{3}+\gamma_3\mfmi{1}\mfmi{2}\mfsi{3}\bigr)=\gamma_3\mfmi{1}\mfmi{2}\mfsi{3}\laplacian_p(\theta_1,\theta_2,\theta_3)\geq 0.
	$$
	We find analogous expressions for $\det(B_{2,2})$ and $\det(B_{1,1})$. This proves \eqref{eq:det_lower}.
	By using similar arguments, we also deduce \eqref{eq:det_upper}.
	The case $\alpha< \beta$ can be addressed in a completely symmetric way by swapping the lower and upper bounds.
\end{proof}

\begin{remark} \label{rmk:sharp3D}
	The bounds in \eqref{eq:behavior_eig3D} provide a quite good description of the three eigenvalue functions. More precisely, two eigenvalue functions are well approximated by $\alpha\laplacian_p$ and the third one by $\beta\laplacian_p$. A numerical confirmation is given in Section \ref{sub:numer_symbol}. This behavior is completely similar to the 2D case; see Remark~\ref{rmk:sharp2D}.
\end{remark}

Again, similar to the 2D case (Corollary \ref{cor:zeros2D}), we can deduce the following vanishing behavior of the eigenvalue functions $\lambda_i(f^{p,\alpha,\beta})$, $i=1,2,3$.
\begin{corollary} \label{cor:zeros3D}
	When $\alpha,\beta>0$, it holds that for $i=1,2,3$,
	\begin{enumerate}
		\item[\mylabel{zero-an3D}{(a)}] $\lambda_i(f^{p,\alpha,\beta}(\theta_1,\theta_2,\theta_3))$ has a unique zero of order $2$ at $(\theta_1,\theta_2,\theta_3)=(0,0,0)$;
		\item[\mylabel{zero-nm3D}{(b)}] $\lambda_i(f^{p,\alpha,\beta}(\theta_1,\theta_2,\theta_3))$ presents an exponential decay in $p$ at all points $(\pm\pi,\theta_2,\theta_3)$, $(\theta_1,\pm\pi,\theta_3)$ and $(\theta_1,\theta_2,\pm\pi)$, implying ``numerical zeros'' for large $p$.
	\end{enumerate}
\end{corollary}
The knowledge of the above vanishing behavior is important in the design of fast iterative solvers for linear systems involving $\mathcal{A}^{p,\alpha,\beta}_{\nn}$ as coefficient matrix (see Section~\ref{sub:multi}).

\subsection{Numerical examples} \label{sub:numer_symbol}
In the following, we verify the spectral results obtained in Sections \ref{sub:distr2D} and \ref{sub:distr3D} through several numerical examples. More precisely, we illustrate that
\begin{itemize}
	\item relation \eqref{eq:distr} holds for $d=2,3$ and $f^{p,\alpha,\beta}$ given by \eqref{eq:f2D} and \eqref{eq:f3D}, respectively;
	\item the eigenvalue functions $\lambda_i(f^{p,\alpha,\beta})$, $i=1,\ldots,d$ satisfy the bounds in \eqref{eq:behavior_eig2D} for $d=2$, and the bounds in \eqref{eq:behavior_eig3D} for $d=3$.
\end{itemize}
Let us fix $\nn:=(n,\ldots,n)\in\mathbb{N}^d$, $\pp:=(p,\ldots,p)\in\mathbb{N}^d$,
and $\mm:=\nn+\pp-{\bm 2}$. We start by defining the following equispaced grid on $[0,\pi]^d$,
\begin{equation*}\label{Gnp}
\Gamma:=\left\{\ttheta_{\kk}:=\frac{\kk\pi}{\mm}: \kk={\bm 1}, \dots, \mm\right\}.
\end{equation*}
Then, we denote by $\Lambda_i$ the set of all evaluations of $\lambda_i(f^{p,\alpha,\beta})$ on $\Gamma$ for fixed $i\in\{1,\ldots,d\}$, i.e.,
\begin{equation*}
\Lambda_i:=\left\{\lambda_i (f^{p,\alpha,\beta}(\ttheta_{\bm 1})), \dots ,\lambda_i (f^{p,\alpha,\beta}(\ttheta_{\mm}))\right\},
\end{equation*}
and by $\Lambda$ the set of all evaluations of $\lambda_i(f^{p,\alpha,\beta})$ on $\Gamma$ varying $i$, i.e.,
\begin{equation*}
\Lambda:=\left\{\Lambda_1,\Lambda_2,\ldots,\Lambda_d\right\}.
\end{equation*}
Note that it suffices to consider only the subdomain $[0,\pi]^d$ because the matrix-valued symbols \eqref{eq:f2D} and \eqref{eq:f3D} are symmetric on $[-\pi,\pi]^d$, and hence also their eigenvalue functions.
Finally, we consider
\begin{equation*}\label{L}
\Delta:=\left\{\laplacian_p(\ttheta_{\bm 1}), \dots ,\laplacian_p(\ttheta_{\mm})\right\},
\end{equation*}
the set of evaluations of the $d$-variate Laplacian symbol $\laplacian_p(\theta_1,\ldots,\theta_d)$ on $\Gamma$,
defined in \eqref{eq:lapl2D} and \eqref{eq:lapl3D} for $d=2,3$.

\begin{figure}[t!]
	\centering
	\begin{subfigure}[c]{.45\textwidth}
		\includegraphics[width=\textwidth]{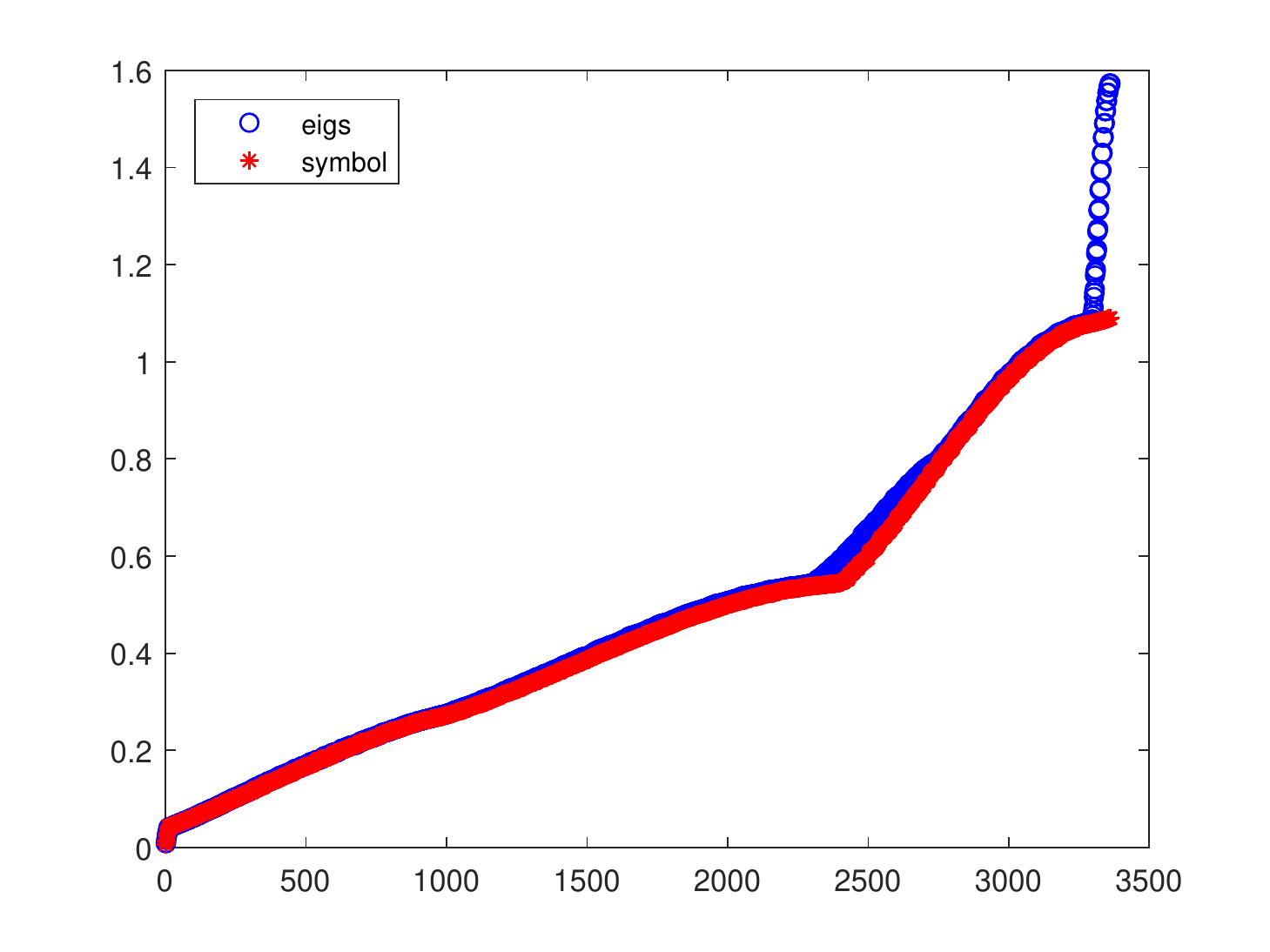}
		\subcaption{$\alpha=1$, $\beta=0.5$}
	\end{subfigure}
	\begin{subfigure}[c]{.45\textwidth}
		\includegraphics[width=\textwidth]{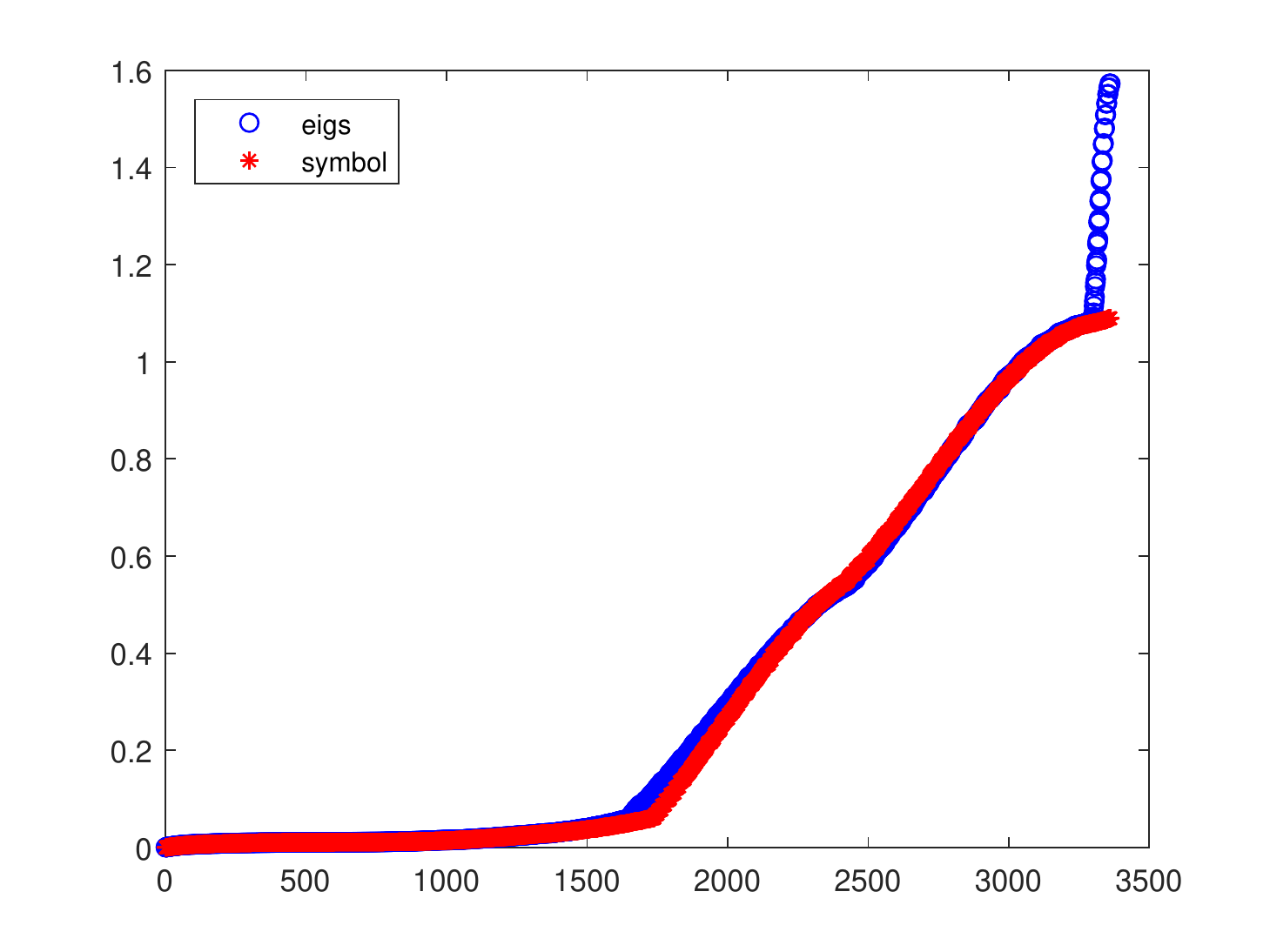}
		\subcaption{$\alpha=1$, $\beta=0.01$}
	\end{subfigure}
	\caption{2D case. Comparison of the eigenvalues of $\mathcal{A}^{p,\alpha,\beta}_{\nn}$ (\textcolor{blue}{$\mycircle$}) with $\Lambda=\{\Lambda_1,\Lambda_2\}$ collecting uniform samples of $\lambda_i(f^{p,\alpha,\beta})$, $i=1,2$ ordered in ascending way (\textcolor{red}{$\ast$}), for $n=40$, $p=3$, $\alpha=1$, and $\beta\in\{0.5,0.01\}$.} 
	\label{fig:comp2d}
	\medskip
	\centering
	\begin{subfigure}[c]{.45\textwidth}
		\includegraphics[width=\textwidth]{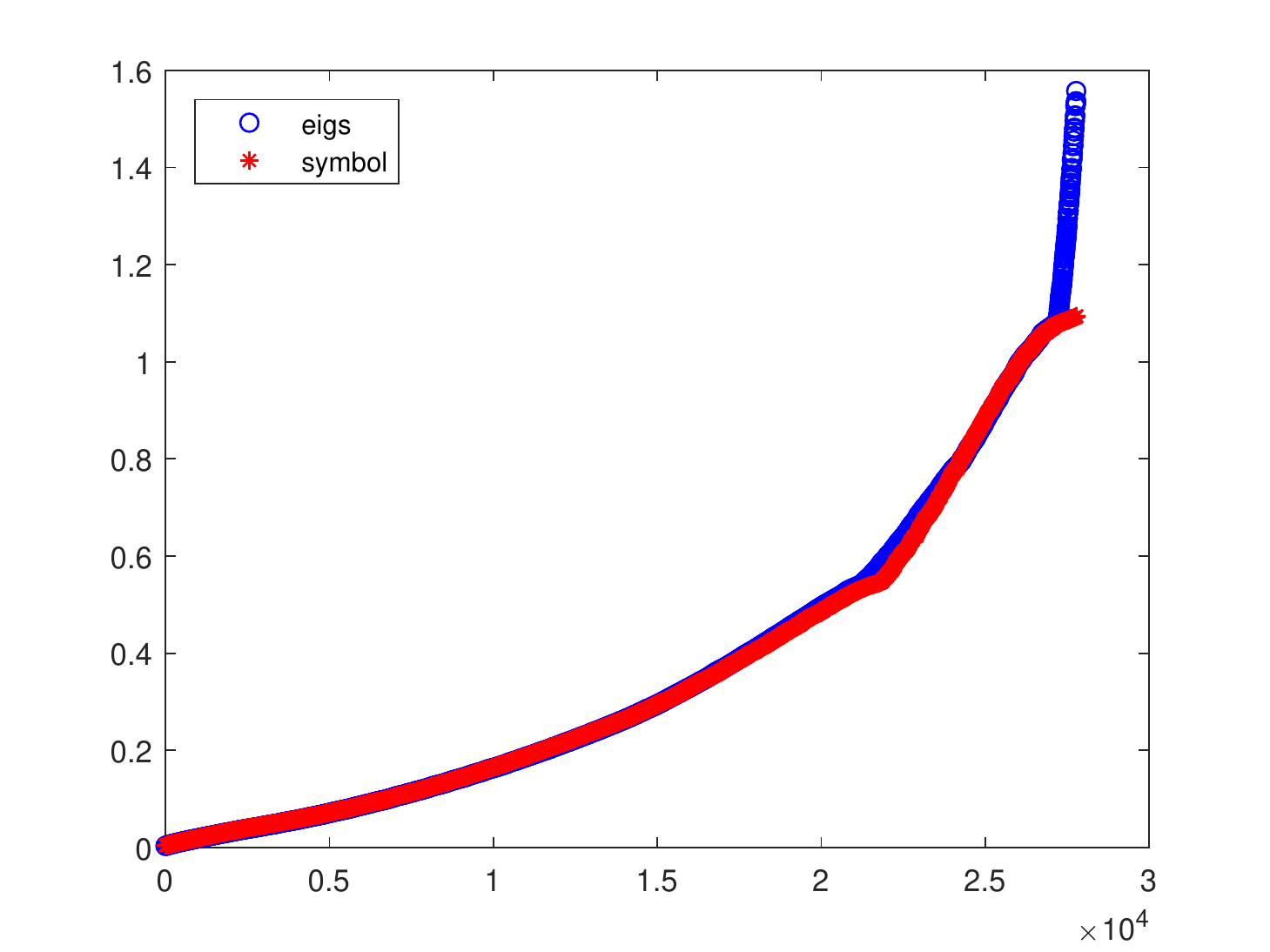}
		\subcaption{$\alpha=1$, $\beta=0.5$}
	\end{subfigure}
	\begin{subfigure}[c]{.45\textwidth}
		\includegraphics[width=\textwidth]{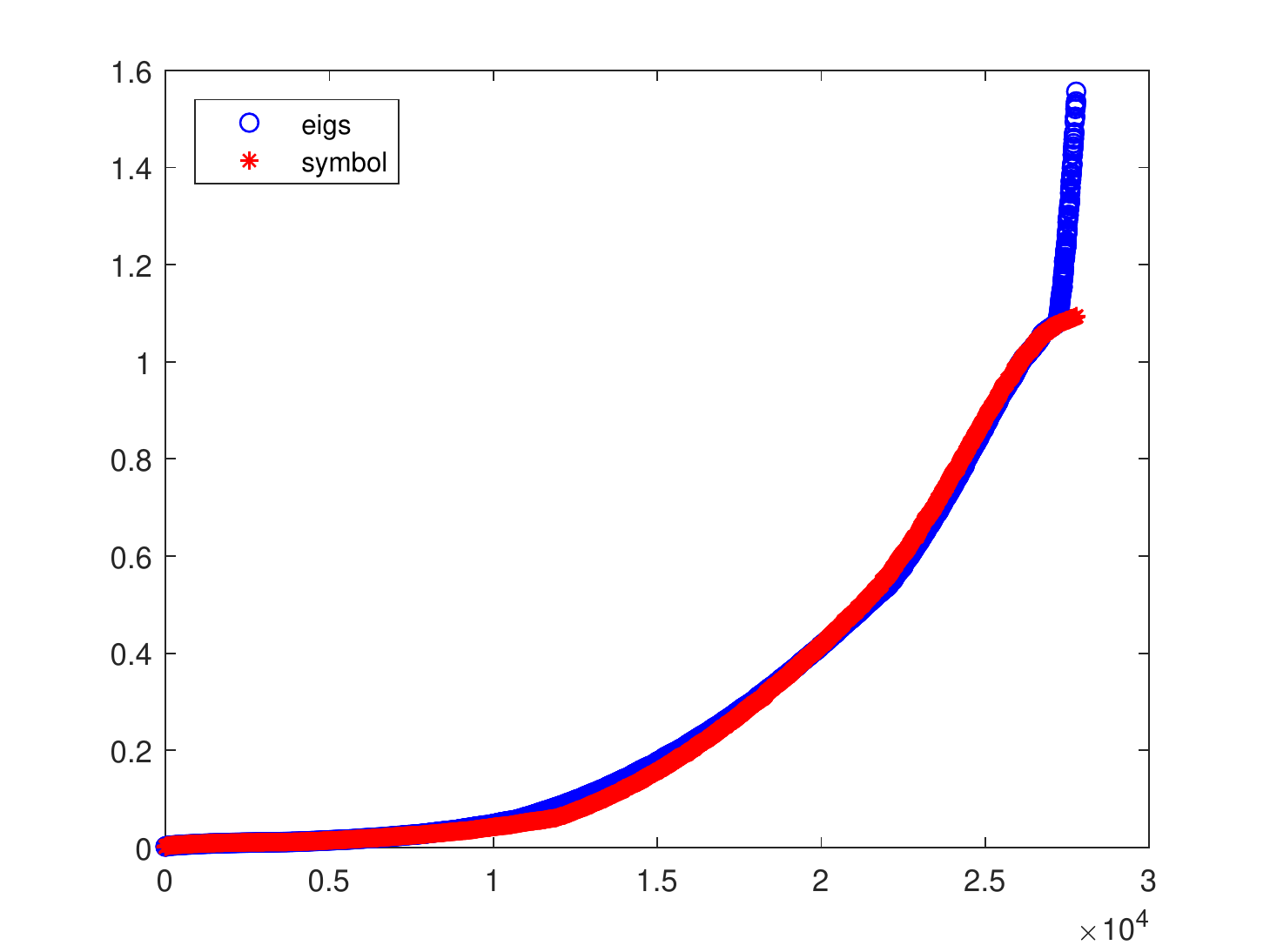}
		\subcaption{$\alpha=1$, $\beta=0.01$}
	\end{subfigure}
	\caption{3D case. Comparison of the eigenvalues of $n\mathcal{A}^{p,\alpha,\beta}_{\nn}$ (\textcolor{blue}{$\mycircle$}) with $\Lambda=\{\Lambda_1,\Lambda_2,\Lambda_3\}$ collecting uniform samples of $\lambda_i(f^{p,\alpha,\beta})$, $i=1,2,3$ ordered in ascending way (\textcolor{red}{$\ast$}), for $n=20$, $p=3$, $\alpha=1$, and $\beta\in\{0.5,0.01\}$.} 
	\label{fig:comp3d}
\end{figure}

\paragraph{Spectral distribution}
Referring to Section~\ref{sub:symbol}, we verify numerically relation \eqref{eq:distr} by comparing the eigenvalues of $n^{d-2}\mathcal{A}^{p,\alpha,\beta}_{\nn}$ with a uniform sampling of the eigenvalue functions $\lambda_i(f^{p,\alpha,\beta})$, $i=1,\ldots,d$ and $d=2,3$. In this view, for $d=2$, Figure \ref{fig:comp2d} depicts both the eigenvalues of $\mathcal{A}^{p,\alpha,\beta}_{\nn}$ defined in \eqref{eq:matr_Amu} and the values collected in $\Lambda=\{\Lambda_1,\Lambda_2\}$, ordered in ascending way. Here, $n=40$, $p=3$, $\alpha=1$, and $\beta\in\{0.5,0.01\}$.
We notice that, as predicted by the theory, independently of $\beta$,  the considered sampling of $\lambda_i(f^{p,\alpha,\beta})$, $i=1,2$ 
describes accurately the behavior of the eigenvalues of $\mathcal{A}^{p,\alpha,\beta}_{\nn}$, up to some outliers.
Furthermore, when $\beta=0.01$, we observe that Figure \ref{fig:comp2d}(b) clearly shows not only the ill-conditioning of the resulting matrices, but also the intrinsic ill-posed nature of the problem as $\beta$ tends to zero. Indeed, we see that about half of the spectrum is almost zero in perfect agreement with our theoretical findings.

Similarly, for $d=3$, in Figure \ref{fig:comp3d} we show both the eigenvalues of $n\mathcal{A}^{p,\alpha,\beta}_{\nn}$, with $\mathcal{A}^{p,\alpha,\beta}_{\nn}$ defined as in \eqref{eq:matr_Amu3D}, and the values collected in $\Lambda=\{\Lambda_1,\Lambda_2,\Lambda_3\}$, ordered in ascending way.
Here, $n=20$, $p=3$, $\alpha=1$, and $\beta\in\{0.5,0.01\}$.
We skipped the analytical computation of the eigenvalue functions $\lambda_i(f^{p,\alpha,\beta})$, $i=1,2,3$ and we directly evaluated them on $\Gamma\subset[0,\pi]^3$ according to the following algorithm:
\begin{itemize}
	\item evaluate the matrix-valued symbol $f^{p,\alpha,\beta}$ at all points of $\Gamma$;
	\item compute the eigenvalues of the obtained $3\times3$ matrix for every fixed grid point;
	\item collect all the smallest eigenvalues of these matrices to obtain a sampling of the eigenvalue function $\lambda_1(f^{p,\alpha,\beta})$ (stored in the set $\Lambda_1$), and so on, until the largest eigenvalues to get a sampling of the eigenvalue function $\lambda_3(f^{p,\alpha,\beta})$ (stored in the set $\Lambda_3$).
\end{itemize}
As expected, also for $d=3$, the considered sampling of $\lambda_i(f^{p,\alpha,\beta})$, $i=1,2,3$ describes accurately the behavior of the eigenvalues of $n\mathcal{A}^{p,\alpha,\beta}_{\nn}$, up to some outliers. Furthermore, for $\beta$ tending to zero, we are dealing with an ill-posed problem, and already for $\beta=0.01$, we see that one third of the spectrum is almost zero.


\begin{figure}[t!]
	\centering
	\begin{subfigure}[c]{.47\textwidth}
		\includegraphics[width=\textwidth]{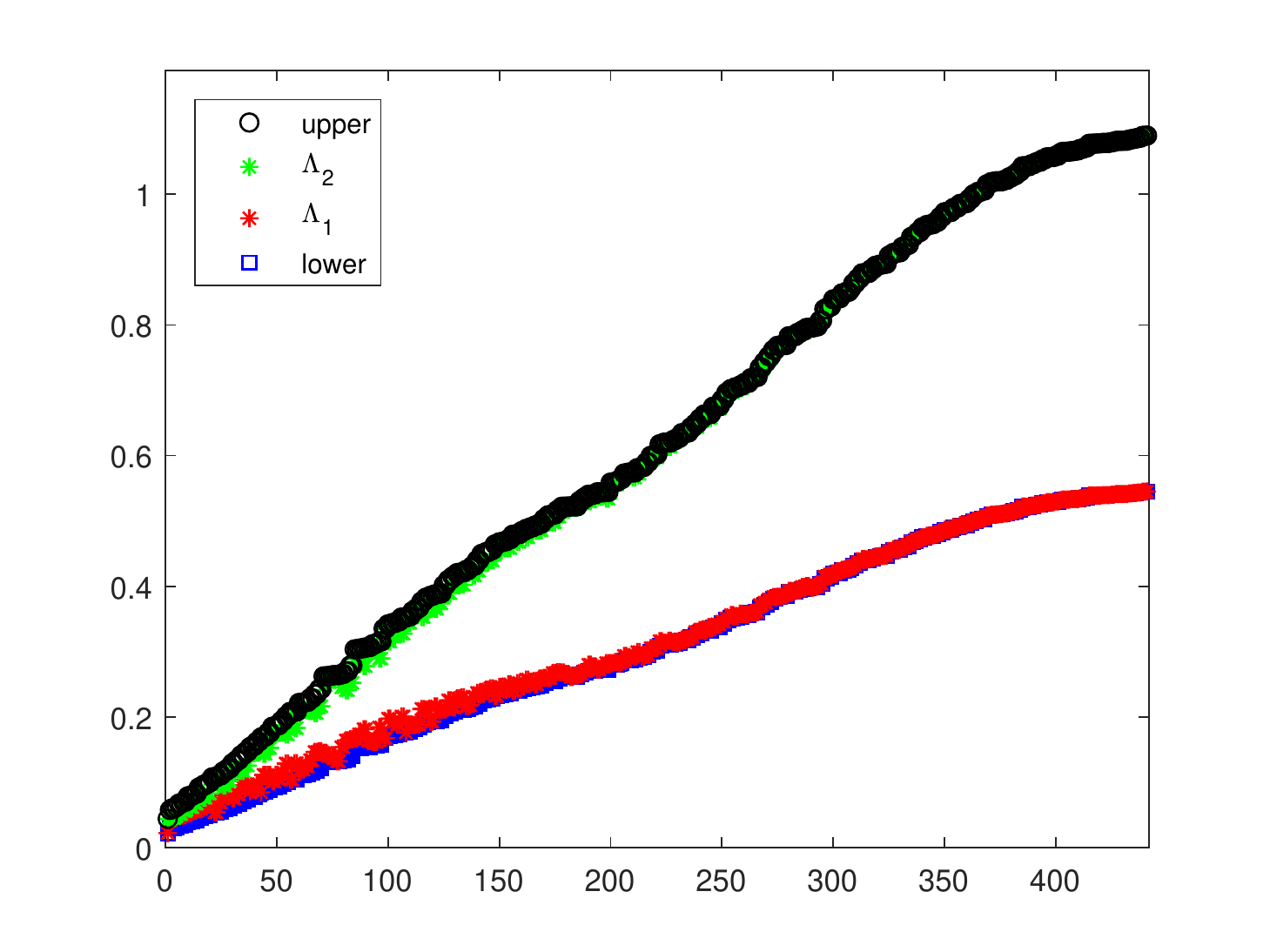}
		\subcaption{full sampling}
	\end{subfigure}
	\begin{subfigure}[c]{.47\textwidth}
		\includegraphics[width=\textwidth]{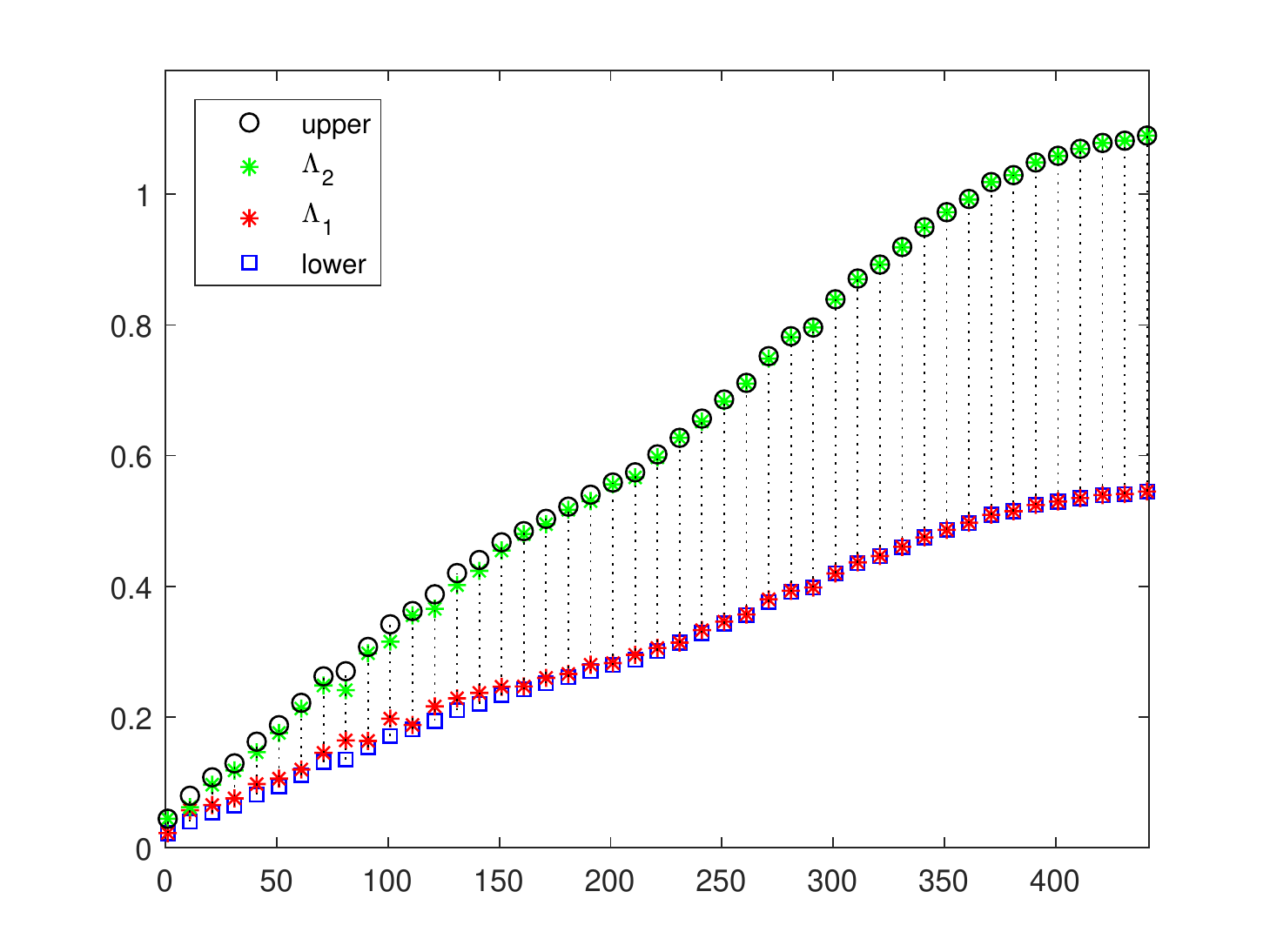}
		\subcaption{1:10 subsampling}
	\end{subfigure}
	\caption{2D case. Pointwise comparison of the lower bound $\beta \laplacian_p$ (\textcolor{blue}{$\mysquare$}) and the upper bound $\alpha \laplacian_p$ (\textcolor{black}{$\mycircle$}) with the eigenvalue functions $\lambda_i(f^{p,\alpha,\beta})$, $i=1,2$ (\textcolor{red}{$\ast$},\textcolor{green}{$\ast$}) for $n=20$, $p=3$, $\alpha=1$, and $\beta=0.5$.}\label{bounds2D305}
\end{figure}

\begin{figure}[t!]
	\centering
	\begin{subfigure}[c]{.47\textwidth}
		\includegraphics[width=\textwidth]{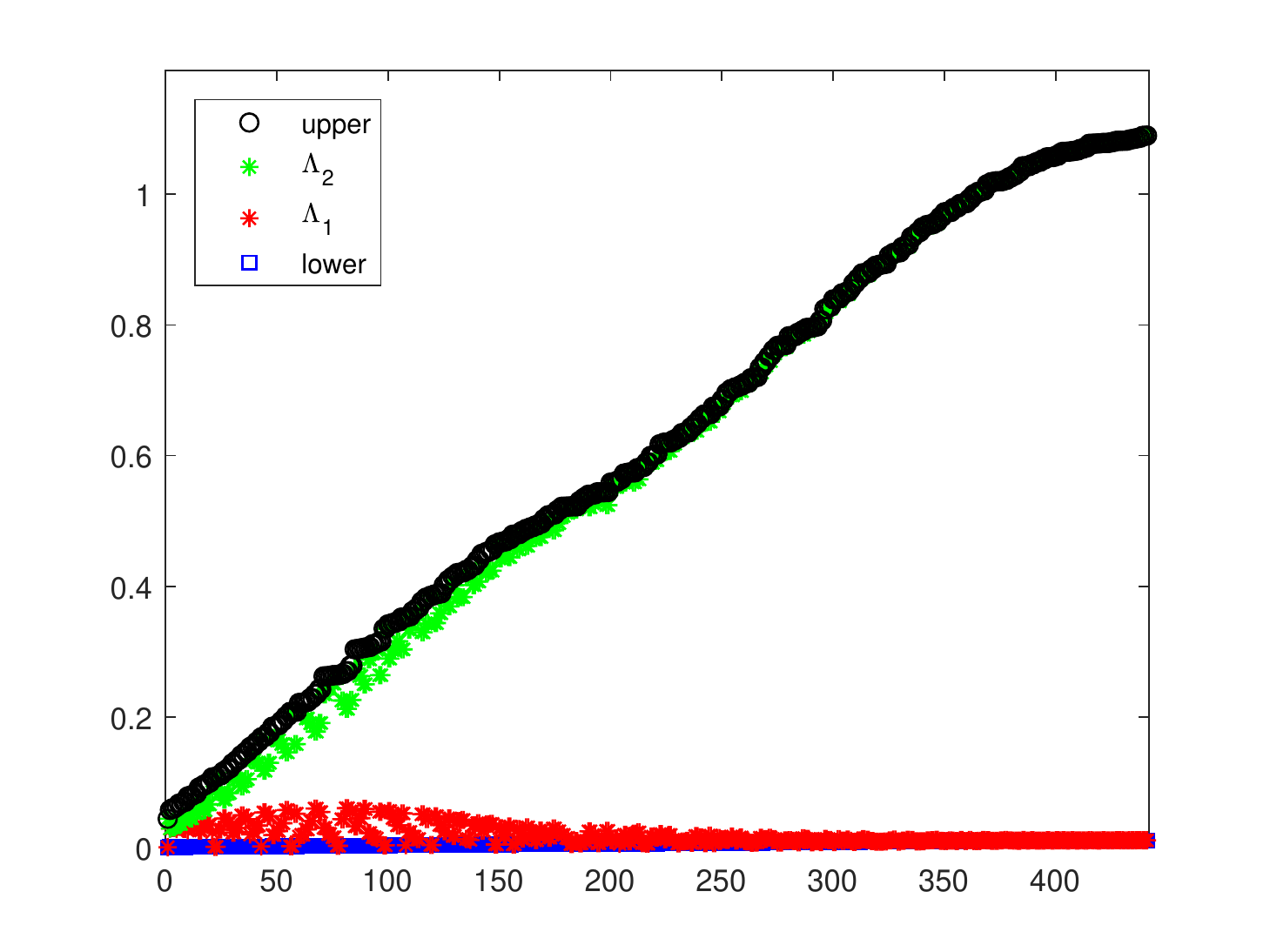}
		\subcaption{full sampling}
	\end{subfigure}
	\begin{subfigure}[c]{.47\textwidth}
		\includegraphics[width=\textwidth]{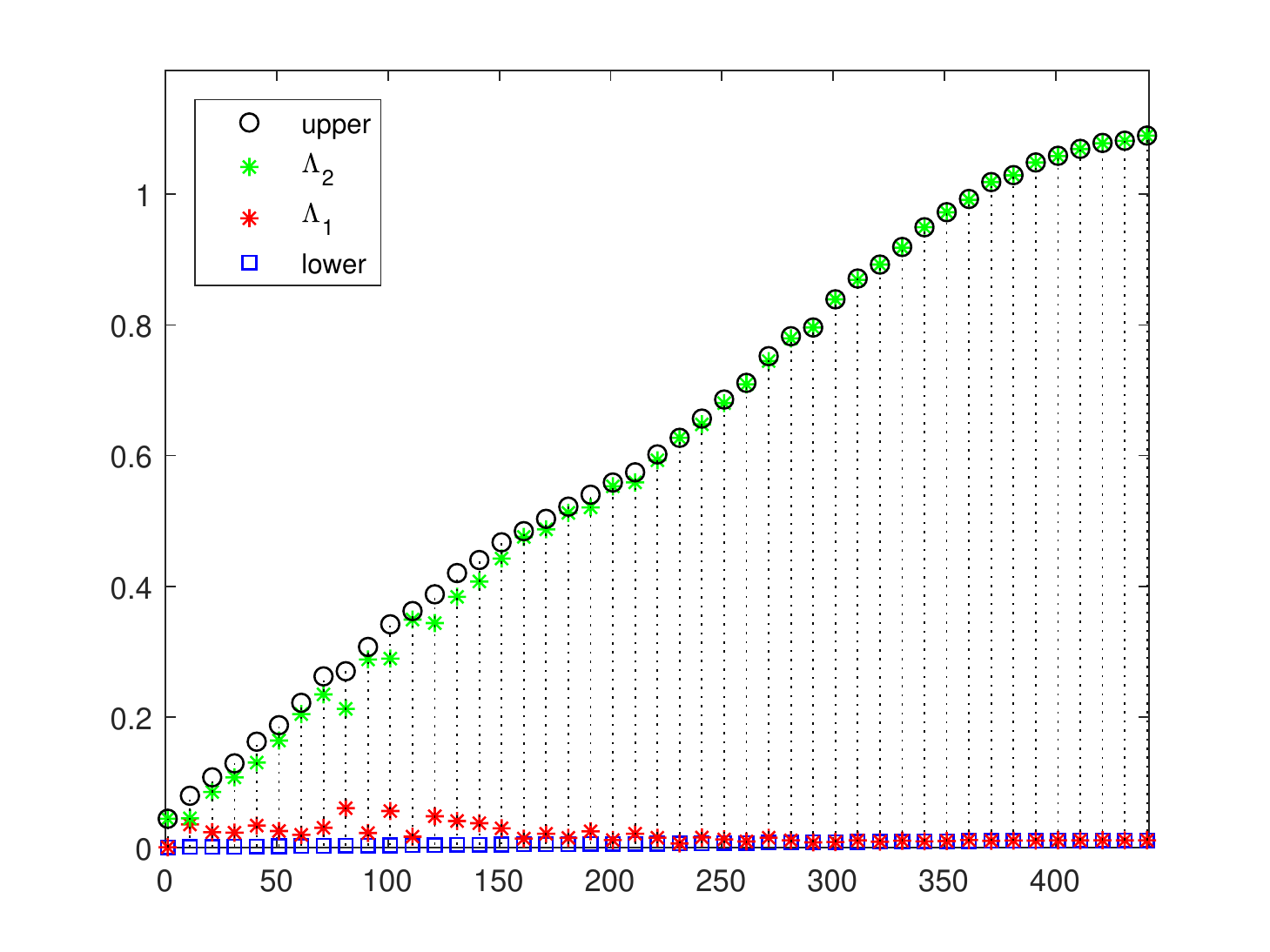}
		\subcaption{1:10 subsampling}
	\end{subfigure}
	\caption{2D case. Pointwise comparison of the lower bound $\beta \laplacian_p$ (\textcolor{blue}{$\mysquare$}) and the upper bound $\alpha \laplacian_p$ (\textcolor{black}{$\mycircle$}) with the eigenvalue functions $\lambda_i(f^{p,\alpha,\beta})$, $i=1,2$ (\textcolor{red}{$\ast$},\textcolor{green}{$\ast$}) for $n=20$, $p=3$, $\alpha=1$, and $\beta=0.01$.}\label{bounds2D3001}
\end{figure}

\begin{figure}[t!]
	\centering
	\begin{subfigure}[c]{.47\textwidth}
		\includegraphics[width=\textwidth]{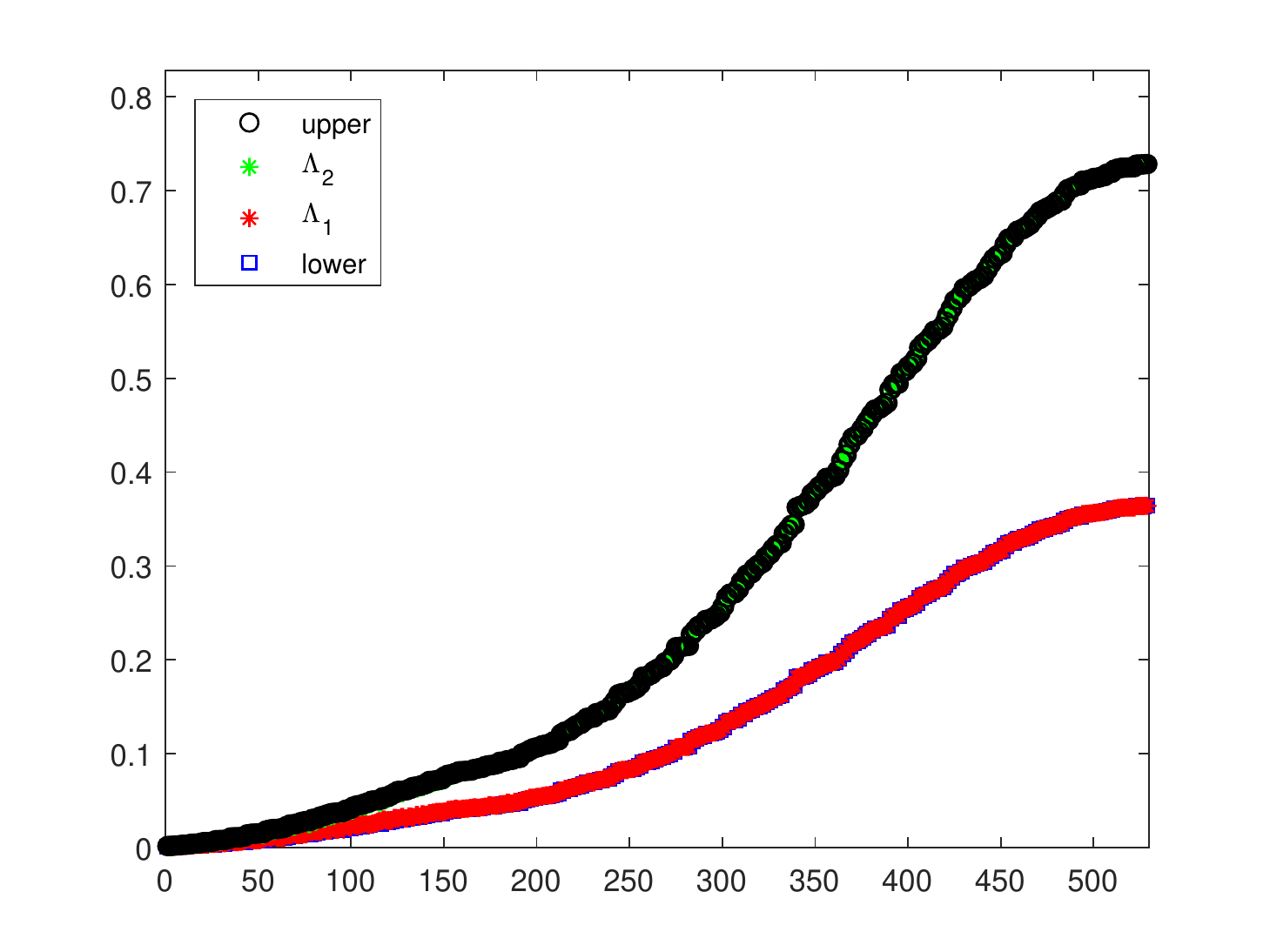}
		\subcaption{full sampling}
	\end{subfigure}
	\begin{subfigure}[c]{.47\textwidth}
		\includegraphics[width=\textwidth]{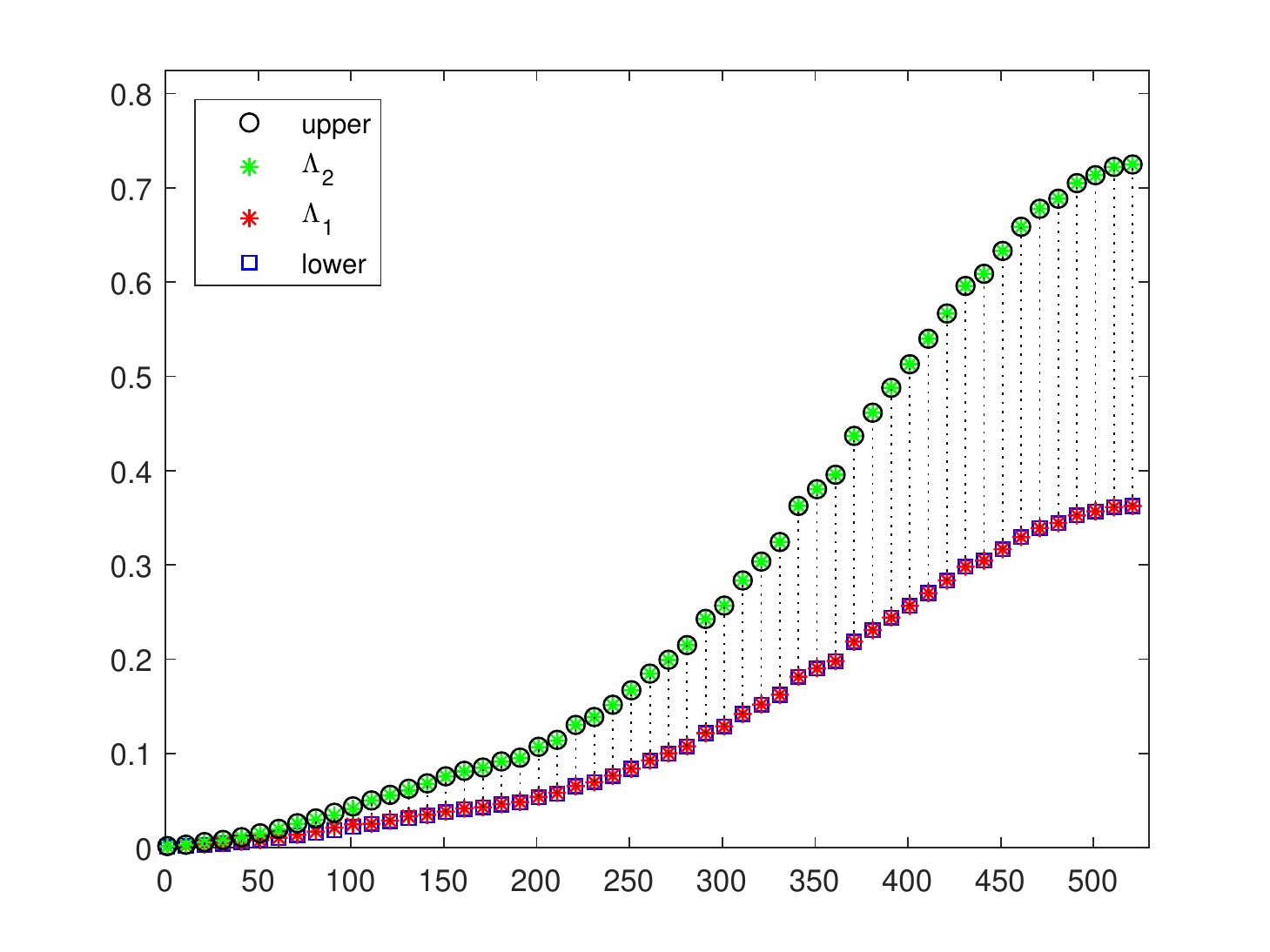}
		\subcaption{1:10 subsampling}
	\end{subfigure}
	\caption{2D case. Pointwise comparison of the lower bound $\beta \laplacian_p$ (\textcolor{blue}{$\mysquare$}) and the upper bound $\alpha \laplacian_p$ (\textcolor{black}{$\mycircle$}) with the eigenvalue functions $\lambda_i(f^{p,\alpha,\beta})$, $i=1,2$ (\textcolor{red}{$\ast$},\textcolor{green}{$\ast$}) for $n=20$, $p=5$, $\alpha=1$, and $\beta=0.5$.}\label{bounds2D505}
\end{figure}
\begin{figure}[t!]
	\centering
	\begin{subfigure}[c]{.47\textwidth}
		\includegraphics[width=\textwidth]{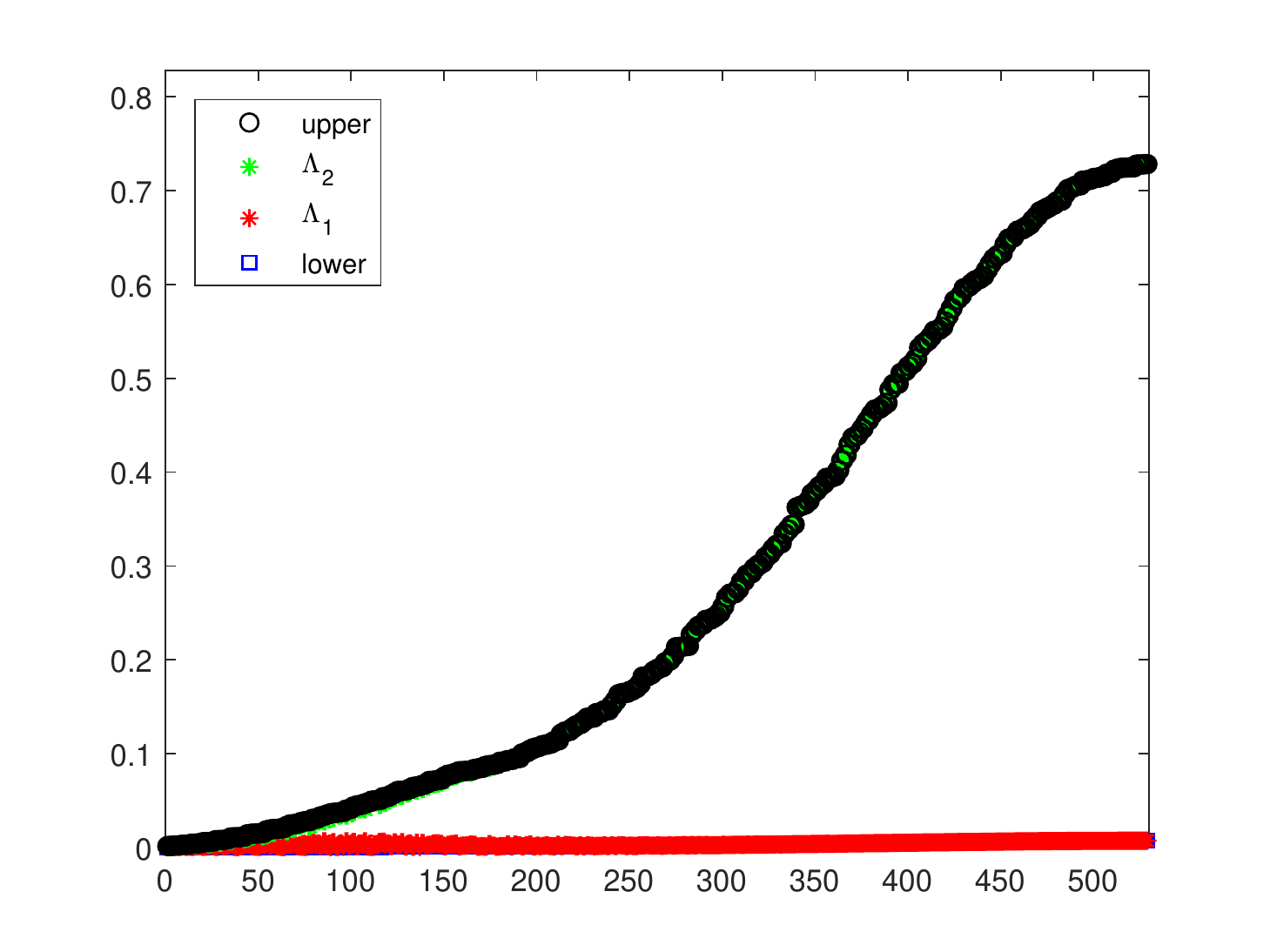}
		\subcaption{full sampling}
	\end{subfigure}
	\begin{subfigure}[c]{.47\textwidth}
		\includegraphics[width=\textwidth]{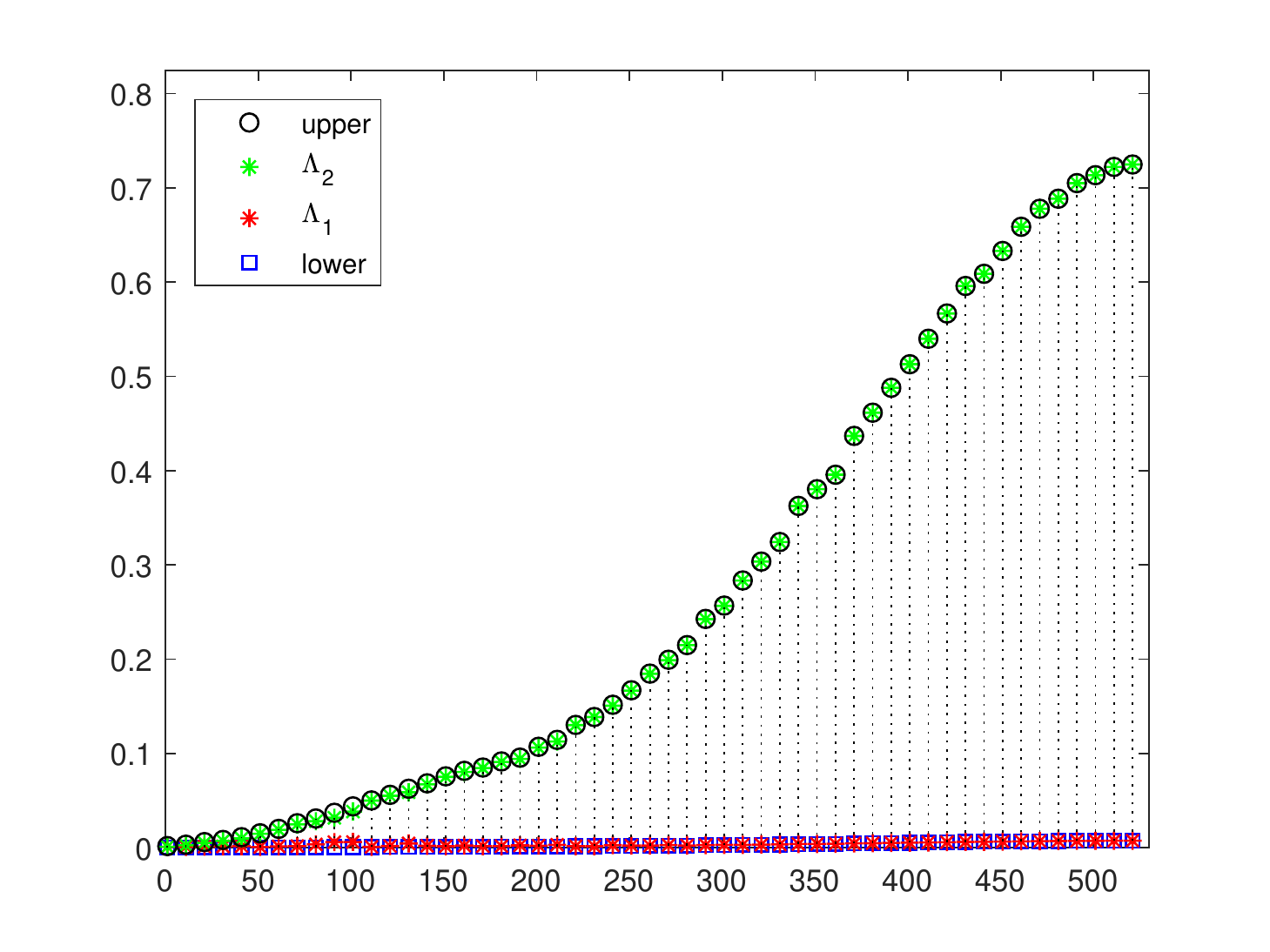}
		\subcaption{1:10 subsampling}
	\end{subfigure}
	\caption{2D case. Pointwise comparison of the lower bound $\beta \laplacian_p$ (\textcolor{blue}{$\mysquare$}) and the upper bound $\alpha \laplacian_p$ (\textcolor{black}{$\mycircle$}) with the eigenvalue functions $\lambda_i(f^{p,\alpha,\beta})$, $i=1,2$ (\textcolor{red}{$\ast$},\textcolor{green}{$\ast$}) for $n=20$, $p=5$, $\alpha=1$, and $\beta=0.01$.}\label{bounds2D5001}
\end{figure}

\begin{figure}[t!]
	\centering
	\begin{subfigure}[c]{.47\textwidth}
		\includegraphics[width=\textwidth]{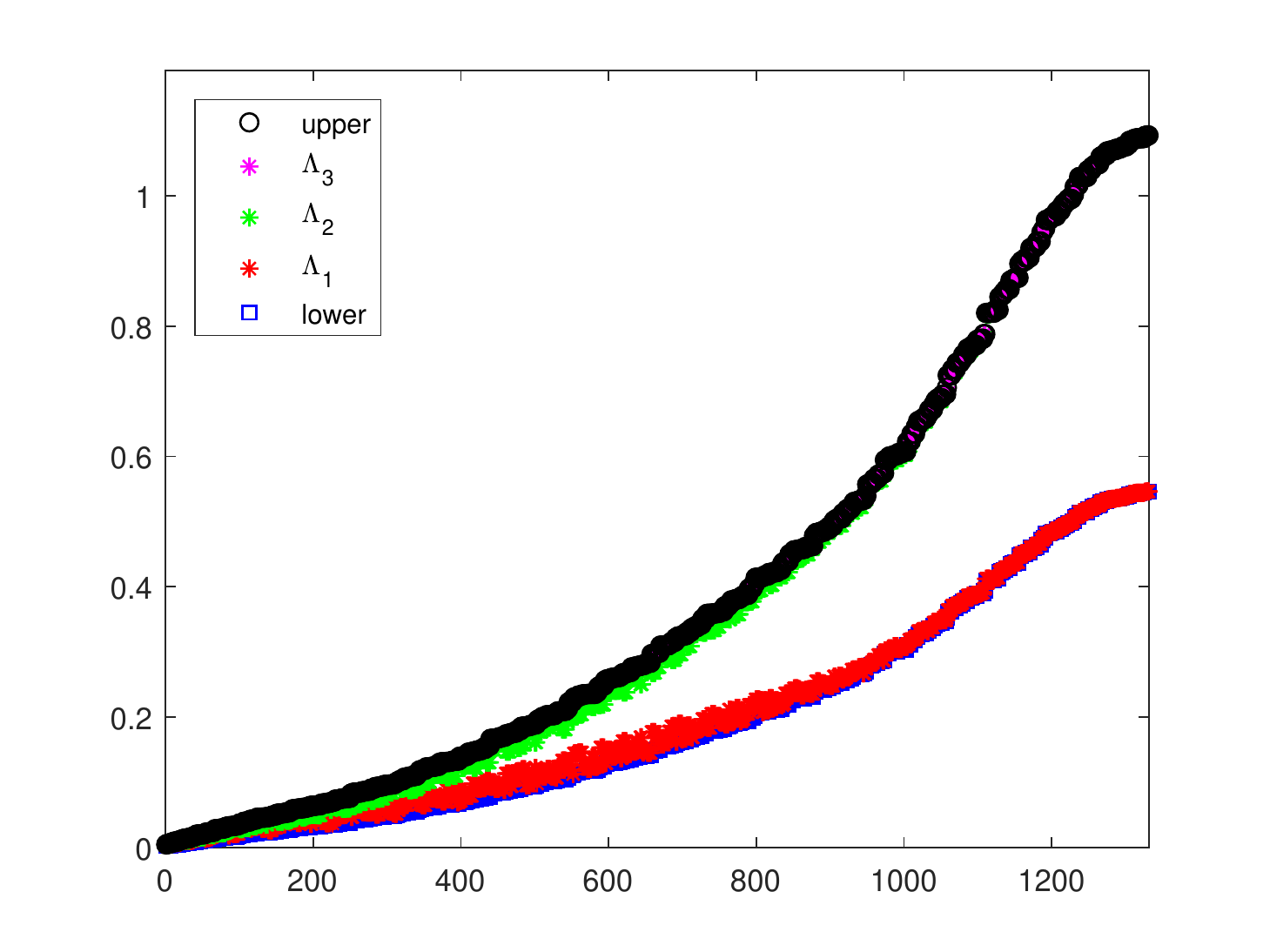}
		\subcaption{full sampling}
	\end{subfigure}
	\begin{subfigure}[c]{.47\textwidth}
		\includegraphics[width=\textwidth]{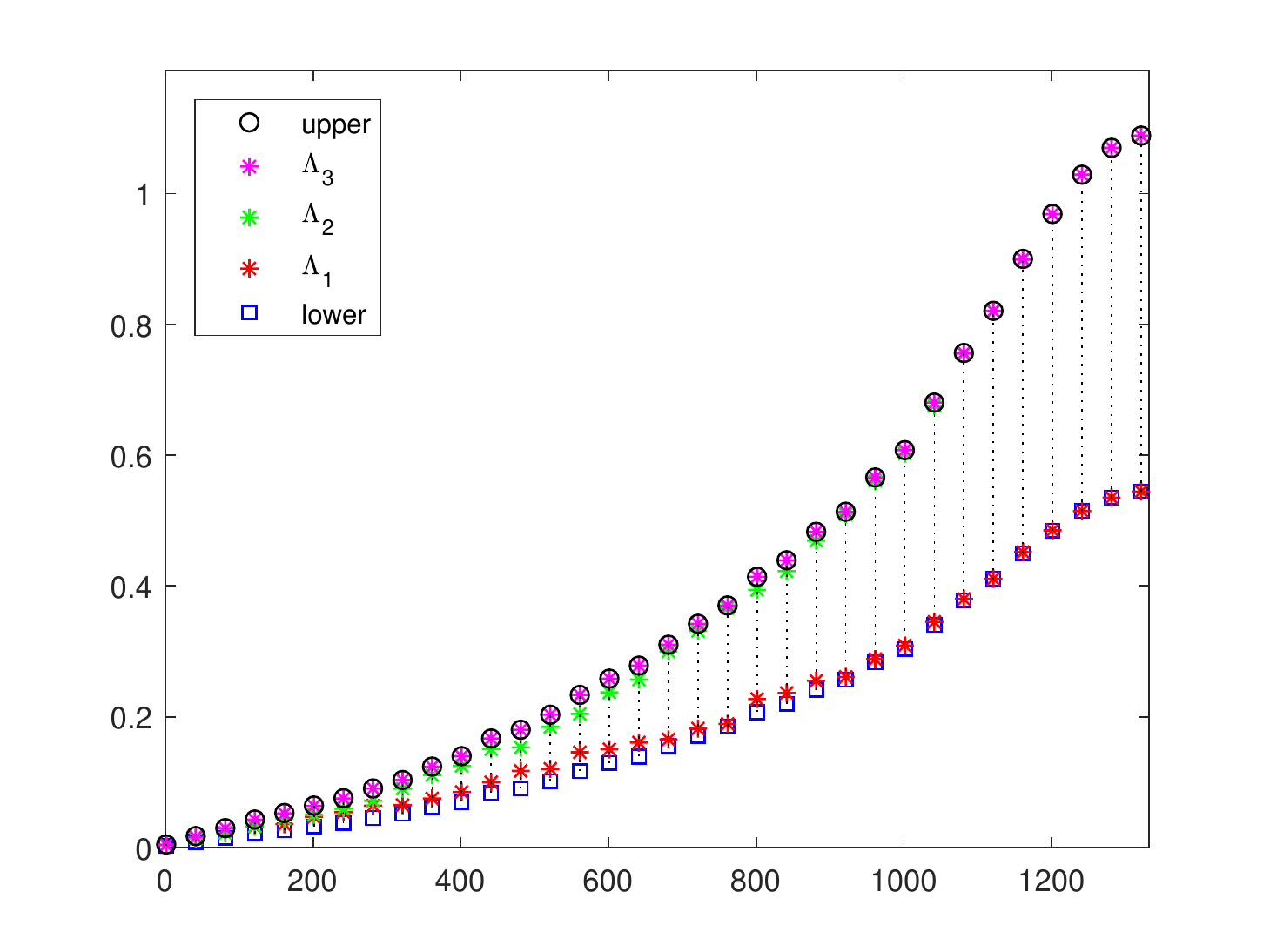}
		\subcaption{1:40 subsampling}
	\end{subfigure}
	\caption{3D case. Pointwise comparison of the lower bound $\beta \laplacian_p$ (\textcolor{blue}{$\mysquare$}) and the upper bound $\alpha \laplacian_p$ (\textcolor{black}{$\mycircle$}) with the eigenvalue functions $\lambda_i(f^{p,\alpha,\beta})$, $i=1,2,3$ (\textcolor{red}{$\ast$},\textcolor{green}{$\ast$},\textcolor{magenta}{$\ast$}) for $n=10$, $p=3$, $\alpha=1$, and $\beta=0.5$.
	}\label{bounds305}
\end{figure}

\begin{figure}[t!]
	\centering
	\begin{subfigure}[c]{.47\textwidth}
		\includegraphics[width=\textwidth]{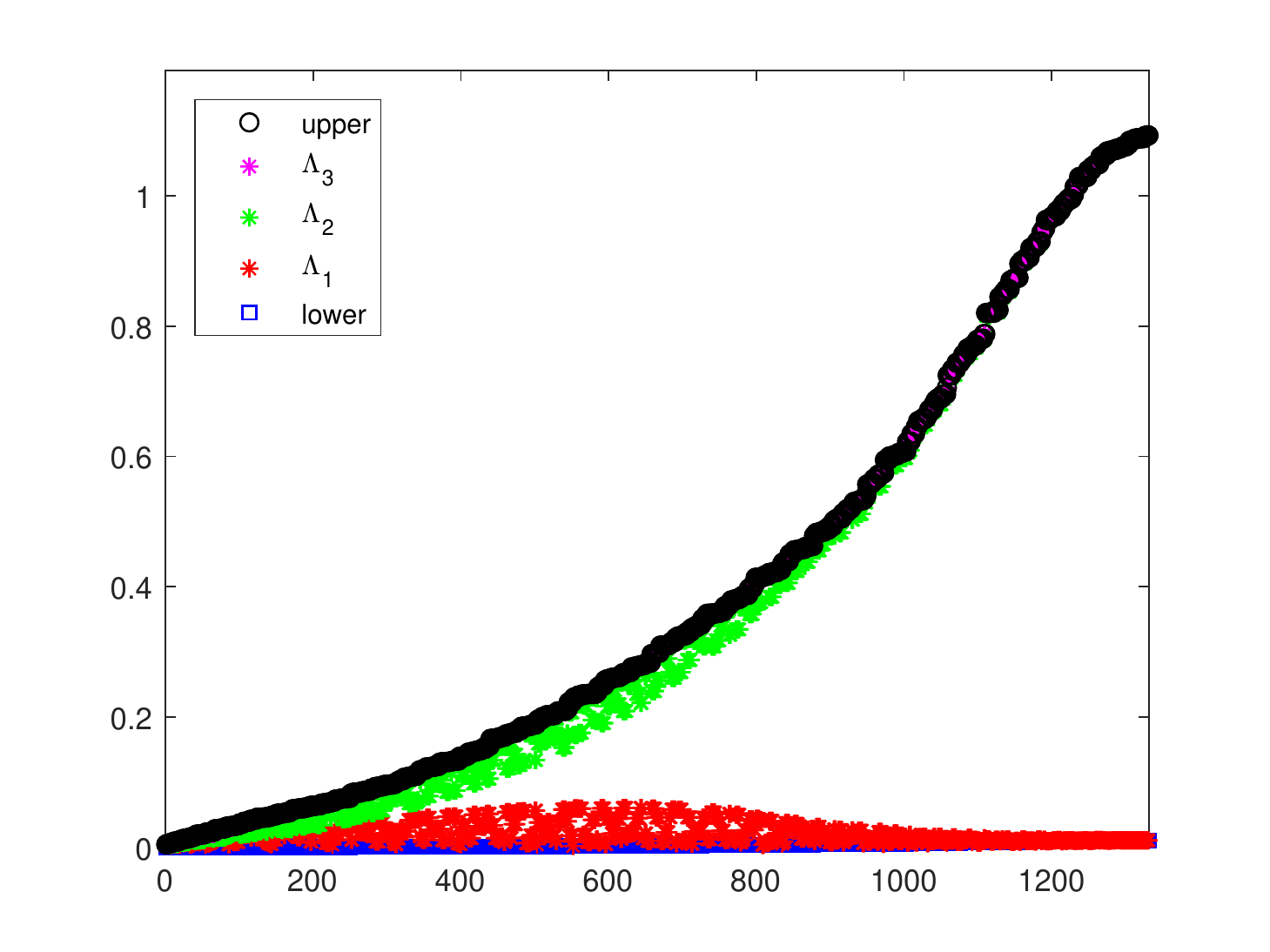}
		\subcaption{full sampling}
	\end{subfigure}
	\begin{subfigure}[c]{.47\textwidth}
		\includegraphics[width=\textwidth]{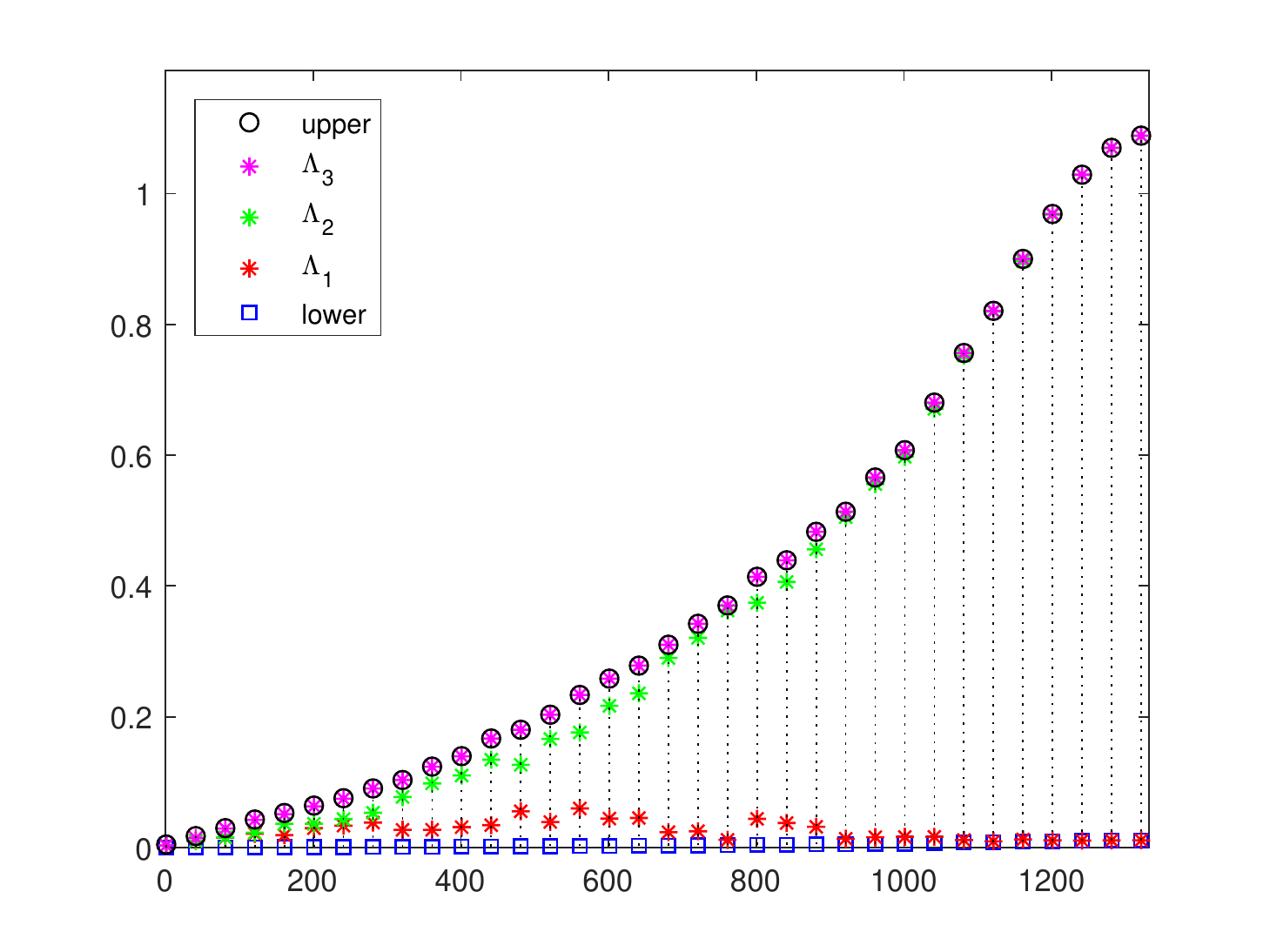}
		\subcaption{1:40 subsampling}
	\end{subfigure}
	\caption{3D case. Pointwise comparison of the lower bound $\beta \laplacian_p$ (\textcolor{blue}{$\mysquare$}) and the upper bound $\alpha \laplacian_p$ (\textcolor{black}{$\mycircle$}) with the eigenvalue functions $\lambda_i(f^{p,\alpha,\beta})$, $i=1,2,3$ (\textcolor{red}{$\ast$},\textcolor{green}{$\ast$},\textcolor{magenta}{$\ast$}) for $n=10$, $p=3$, $\alpha=1$, and $\beta=0.01$.}\label{bounds3001}
\end{figure}

\begin{figure}[t!]
	\centering
	\begin{subfigure}[c]{.47\textwidth}
		\includegraphics[width=\textwidth]{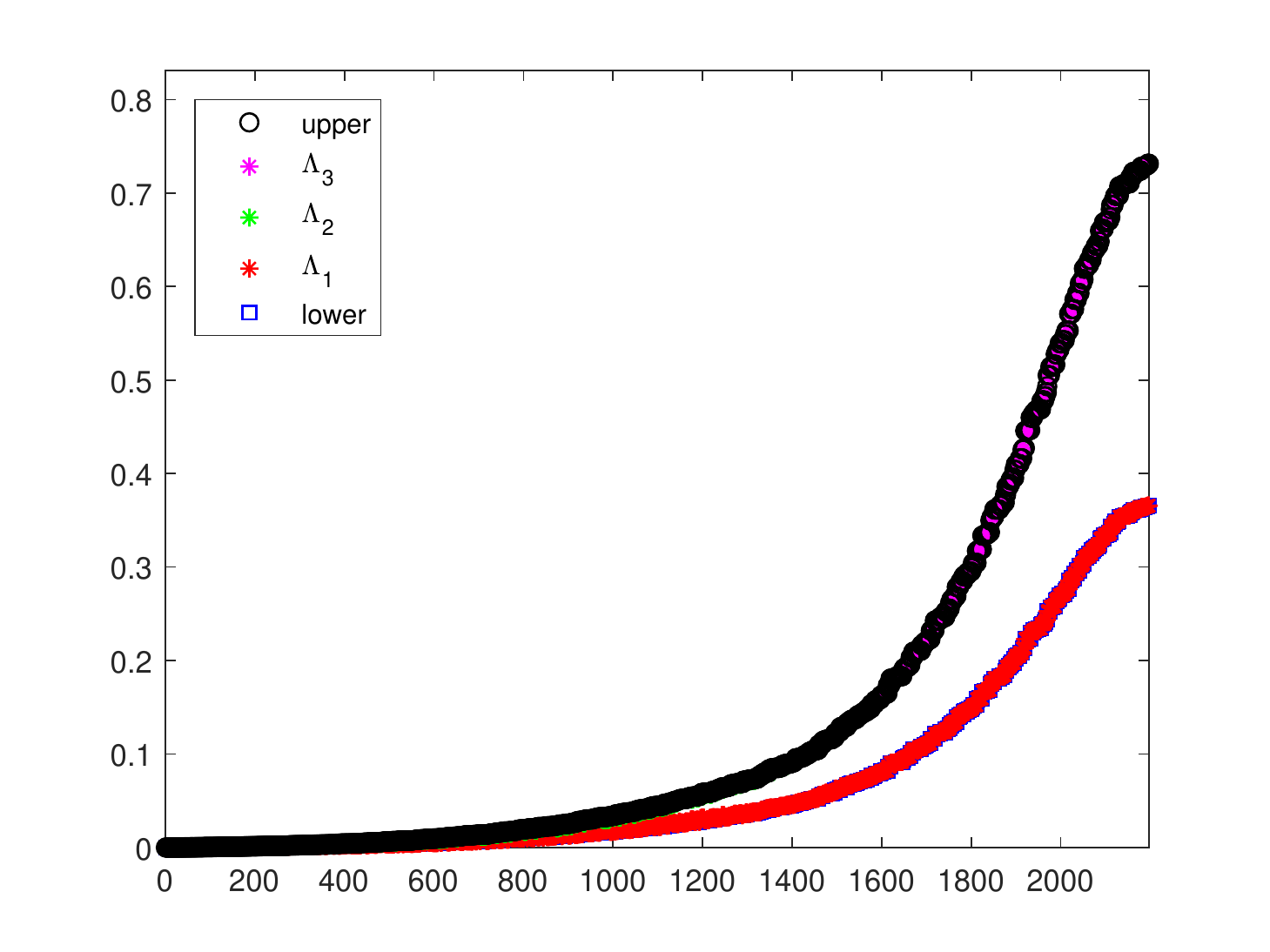}
		\subcaption{full sampling}
	\end{subfigure}
	\begin{subfigure}[c]{.47\textwidth}
		\includegraphics[width=\textwidth]{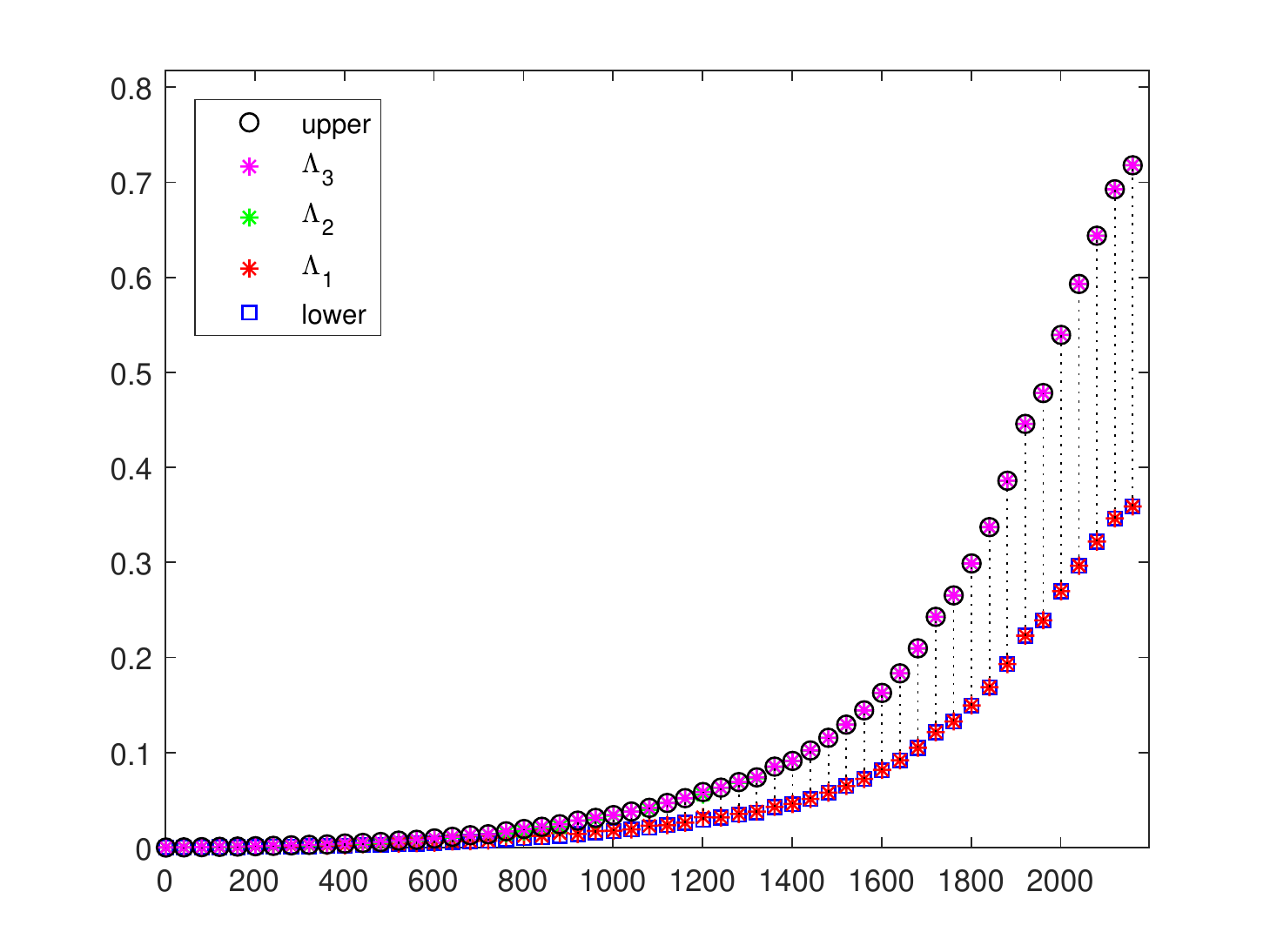}
		\subcaption{1:40 subsampling}
	\end{subfigure}
	\caption{3D case. Pointwise comparison of the lower bound $\beta \laplacian_p$ (\textcolor{blue}{$\mysquare$}) and the upper bound $\alpha \laplacian_p$ (\textcolor{black}{$\mycircle$}) with the eigenvalue functions $\lambda_i(f^{p,\alpha,\beta})$, $i=1,2,3$ (\textcolor{red}{$\ast$},\textcolor{green}{$\ast$},\textcolor{magenta}{$\ast$}) for $n=10$, $p=5$, $\alpha=1$, and $\beta=0.5$.}\label{bounds505}
\end{figure}
\begin{figure}[t!]
	\centering
	\begin{subfigure}[c]{.47\textwidth}
		\includegraphics[width=\textwidth]{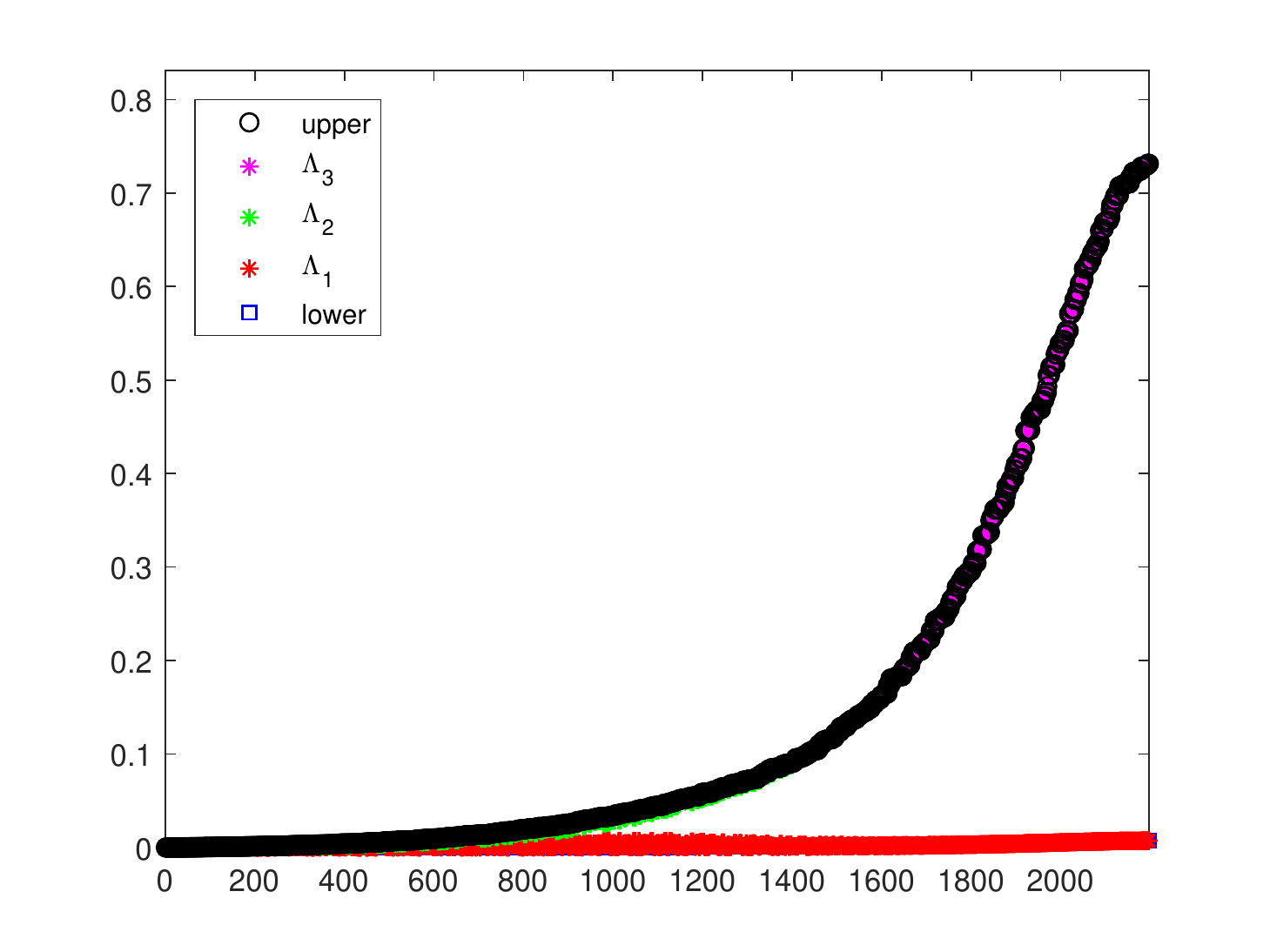}
		\subcaption{full sampling}
	\end{subfigure}
	\begin{subfigure}[c]{.47\textwidth}
		\includegraphics[width=\textwidth]{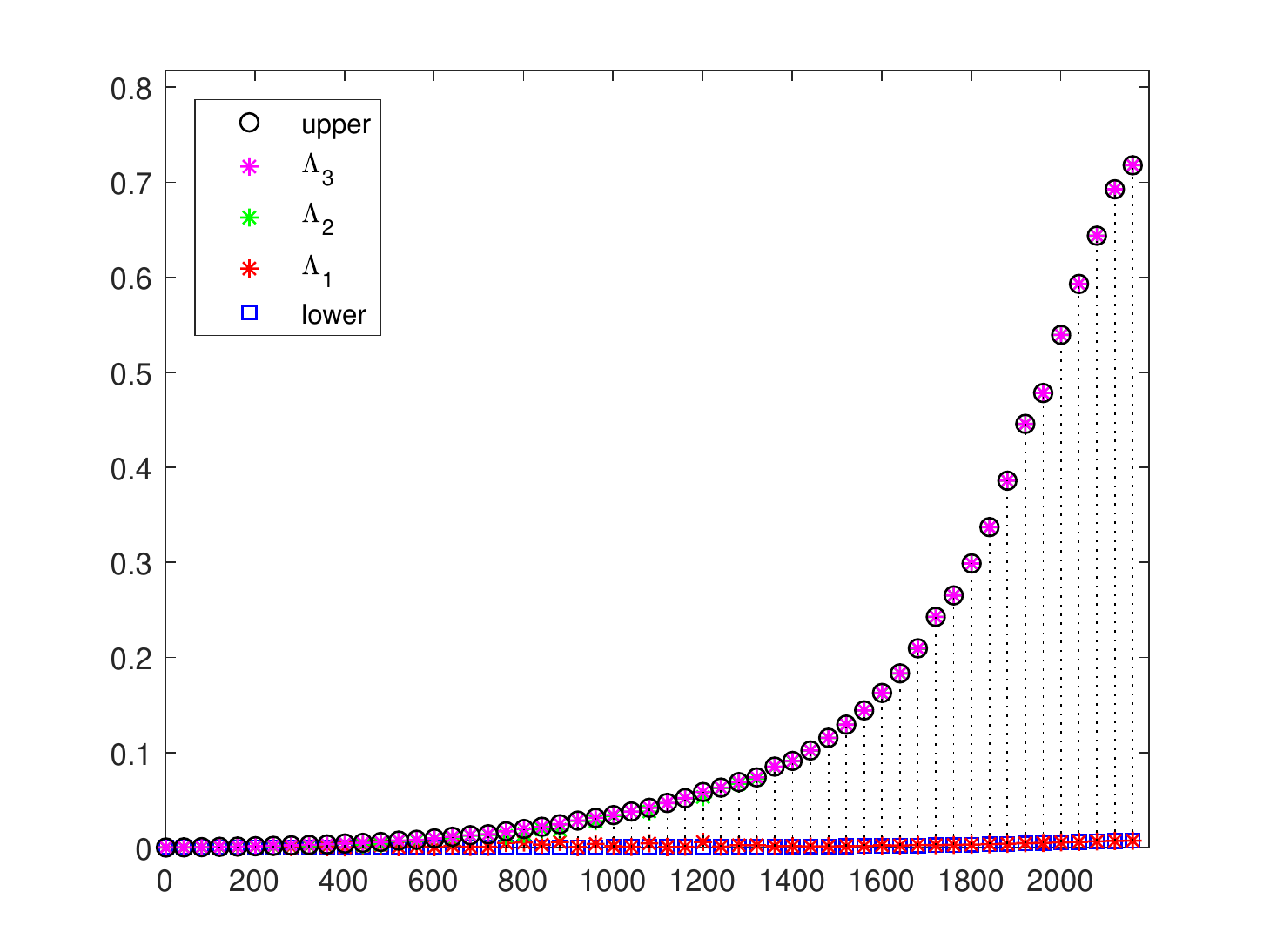}
		\subcaption{1:40 subsampling}
	\end{subfigure}
	\caption{3D case. Pointwise comparison of the lower bound $\beta \laplacian_p$ (\textcolor{blue}{$\mysquare$}) and the upper bound $\alpha \laplacian_p$ (\textcolor{black}{$\mycircle$}) with the eigenvalue functions $\lambda_i(f^{p,\alpha,\beta})$, $i=1,2,3$ (\textcolor{red}{$\ast$},\textcolor{green}{$\ast$},\textcolor{magenta}{$\ast$}) for $n=10$, $p=5$, $\alpha=1$, and $\beta=0.01$.}\label{bounds5001}
\end{figure}

\paragraph{Spectral bounds}
We now verify numerically the sharpness of the theoretical bounds provided in \eqref{eq:behavior_eig2D} and  \eqref{eq:behavior_eig3D} for the eigenvalue functions $\lambda_i(f^{p,\alpha,\beta})$, $i=1,\ldots,d$ and $d=2,3$.
For $d=2$, we compare the values in the sets $\Lambda_i$, $i=1,2$ with the values in the set $\Delta$ multiplied by $\alpha$ and $\beta$, to obtain the lower and upper bounds. For all sets we keep the ordering of $\Gamma$ to be the one corresponding to the ascending sorting of $\Delta$. As shown in Figures \ref{bounds2D305}(a)--\ref{bounds2D5001}(a), relation \eqref{eq:behavior_eig2D} holds for fixed $n=20$, $\alpha=1$, and varying $p,\beta$ in the sets $\{3,5\}$, $\{0.5,0.01\}$, respectively. In Figures \ref{bounds2D305}(b)--\ref{bounds2D5001}(b), we verify again relation \eqref{eq:behavior_eig2D}, but now only a subsampling (1 out of 10 points) of the sets $\Lambda_i$, $i=1,2$ and $\Delta$ is depicted. As mentioned in Remark \ref{rmk:sharp2D}, we see a very good match between the lower bound and $\lambda_1(f^{p,\alpha,\beta})$, and between the upper bound and $\lambda_2(f^{p,\alpha,\beta})$.

Similarly, for $d=3$, we compare the values in the sets $\Lambda_i$, $i=1,2,3$ with the values in the set $\Delta$ multiplied by $\alpha$ and $\beta$. We assume the ordering of the grid $\Gamma$ to be the one corresponding to the ascending sorting of $\Delta$. As shown in Figures \ref{bounds305}(a)--\ref{bounds5001}(a), relation \eqref{eq:behavior_eig3D} holds for fixed $n=10$, $\alpha=1$, and varying $p,\beta$ in the sets $\{3,5\}$, $\{0.5,0.01\}$, respectively. In Figures \ref{bounds305}(b)--\ref{bounds5001}(b), we check again relation \eqref{eq:behavior_eig3D}, but now only a subsampling (1 out of 40 points) of the sets $\Lambda_i$, $i=1,2,3$ and $\Delta$ is depicted.
From these examples, we conclude that both the lower and upper bounds are sharp (see Remark \ref{rmk:sharp3D}).

\section{Ad hoc solvers for IgA curl-div matrices}\label{sec:proposals}
In this section, we focus on the fast solution of
\begin{equation}\label{eq:system}
\mathcal{A}^{p,\alpha,\beta}_{\nn}\vecu=\bb,
\end{equation}
with $\mathcal{A}^{p,\alpha,\beta}_{\nn}$ defined in \eqref{eq:matr_Amu} for $d=2$ and in \eqref{eq:matr_Amu3D} for $d=3$.
We will exploit the spectral information on the coefficient matrix-sequence $\{\mathcal{A}^{p,\alpha,\beta}_{\nn}\}_{\nn}$ illustrated in the previous section to design ad hoc solvers for \eqref{eq:system}. In particular, we extend to our case the successful multi-iterative method for solving IgA discretizations of elliptic problems applied in \cite{SINUM-IgA,serra2,cmame2}. We end the section with several numerical experiments.

\subsection{A multi-iterative approach: multigrid plus $p$-independent PCG (or PGMRES) as smoother}\label{sub:multi}
We discuss a vector extension to the curl-div problem of the multi-iterative method proposed in \cite{serra2} for solving the linear systems coming from the IgA discretization of second-order elliptic problems. More precisely, we propose a strategy made up of the following building blocks:
\begin{itemize}
	\item a V-cycle with a $d\times d$ block multi-linear interpolation prolongator at each level,
	\item one Gauss--Seidel pre-smoothing iteration at each level, a few post-smoothing iterations of a preconditioned conjugate gradient (PCG) or of a preconditioned generalized minimal residual (PGMRES) method at the finest level whose preconditioner is chosen as a direct sum of $d$-level Toeplitz matrices generated by
	$$g_p(\theta_1,\ldots,\theta_d):=\prod_{i=1}^d\mfm_{{p-1}}(\theta_i),$$
	and one Gauss--Seidel post-smoothing iteration at the other levels.
\end{itemize}
The previous building blocks are strongly guided by the knowledge of the eigenvalue functions of the symbol of the matrix-sequence $\{\mathcal{A}^{p,\alpha,\beta}_{\nn}\}_{\nn}$.
Their properties described in Corollaries \ref{cor:zeros2D} and \ref{cor:zeros3D} imply the small eigenvalues of $\mathcal{A}^{p,\alpha,\beta}_{\nn}$ to be related to subspaces of both low and high frequencies. Mimicking the idea from \cite{serra2} justifies also in our case the use of a so-called \emph{multi-iterative method} (see \cite{Serra93}), that is a method made up of different basic iterative solvers having complementary spectral behavior. In particular, as already pointed out, we focus on a V-cycle multigrid which is able to cope with the standard ill-conditioning in the low frequencies, combined with a PCG (or PGMRES) post-smoothing at the finest level whose preconditioner works in the subspace of high frequencies. As confirmed by the numerical examples in Section \ref{sub:numer_mg}, the resulting method is optimal, i.e., has a convergence rate independent of the matrix-size and is robust with respect to the degree.

In the remaining part of this section, we describe in detail our V-cycle multigrid solver, and in particular how to build the prolongator and the preconditioner for the PCG (or PGMRES) smoother. We end with an intuitive spectral motivation why the proposed preconditioner works.

\paragraph{Multigrid idea}
Let $A\vecu=\bb$ be the linear system we want to solve, with $\vecu,\bb\in\mathbb{R}^N$ and $A\in \mathbb{R}^{N\times N}$ SPD matrix. Fix $\lev+1$ integers $N=:N_0>N_1>\cdots>N_\lev>0$, where $0<\lev<N$ denotes the maximum number of levels we decided to use. A multigrid method requires the following ingredients:
\begin{itemize}
	\item appropriate smoothers $S_\ell$, $\tilde{S}_\ell$, and the corresponding smoothing steps $s_\ell$, $\tilde{s}_\ell$ for every level $\ell=0,\ldots,\lev-1$;
	\item restriction operators $R_\ell:\mathbb{R}^{N_\ell}\rightarrow\mathbb{R}^{N_{\ell+1}}$  and prolongation operators
	$P_\ell:\mathbb{R}^{N_{\ell+1}}\rightarrow\mathbb{R}^{N_\ell}$ to transfer a quantity between levels $\ell$ and $\ell+1$, for $\ell=0,\ldots,\lev-1$;
	\item a hierarchy of matrices $A_\ell\in\mathbb{R}^{N_\ell\times N_\ell}$, $\ell=0,\ldots,\lev$ (with $A_{0}:=A$).
\end{itemize}
For solving (iteratively) a linear system of the form
$A_\ell\vecu_\ell=\bb_\ell$ at level $0\le\ell\le L-1$,
a single multigrid iteration in the \emph{$V$-cycle version} consists of the following steps. Given an initial approximation $\vecu_\ell^{(0)}$,
\begin{enumerate}
	\item $s_\ell$ pre-smoothing steps are performed with $S_\ell$, resulting in the new approximation $\vecu_\ell^{(1)}$;
	\item this approximation is corrected using the coarser level (\emph{coarse grid correction}):
	the residual $\rr_\ell=\bb_\ell-A_\ell\vecu_\ell^{(1)}$ is restricted to the coarser level, $\rr_{\ell+1}=R_\ell\rr_\ell$, and is used to build the error equation
	\begin{equation*}
	A_{\ell+1}\ee_{\ell+1}=\rr_{\ell+1};
	\end{equation*}
	then, this system is solved approximately by a single recursive call of the multigrid method, and its solution $\ee_{\ell+1}$ is prolongated back to the finer level, $\ee_\ell=P_\ell\ee_{\ell+1}$, which is used to update the approximation to $\vecu_\ell^{(2)}=\vecu_\ell^{(1)}+\ee_\ell$.
	\item  $\tilde{s}_\ell$ post-smoothing steps are performed with $\tilde{S}_\ell$ and results in the improved approximation $\vecu_\ell^{(3)}$.
\end{enumerate}
We focus on the case $R_\ell:=(P_\ell)^T$ and $A_{\ell+1}:=(P_\ell)^T A_\ell P_\ell$, the so-called \emph{Galerkin approach}.

\paragraph{Choice of the prolongator in our multigrid}
We now search for an appropriate prolongator for the multigrid in order to address our specific linear system. To this end, we follow the approach in \cite{serra2} for second order elliptic PDEs by exploiting the fact that the IgA matrices of interest can be considered as perturbations of Toeplitz matrices.

Let us start by recalling what is a typical choice for the prolongator in the multigrid setting for $d$-level Toeplitz matrices. Fixing $\mm_\ell:=((m_\ell)_1,\ldots,(m_\ell)_d)\in\mathbb{N}^d$, $\mm_{\ell+1}<\mm_\ell$, $\ell=0,\ldots,\lev-1$ with odd components, it was suggested in \cite{ADS2,ADS,Fior-Serra-2,Serra93} to consider prolongators of the form
\begin{equation}\label{eq:prolong}
P_{\mm_\ell}:=T_{\mm_\ell}(q_\ell)\ (K_{\mm_\ell})^T,
\end{equation}
where $q_\ell$ is a nonnegative $d$-variate trigonometric polynomial and
\begin{equation}\label{eq:Kut}
K_{\mm_\ell}:=K_{(m_\ell)_{1}}\otimes\cdots\otimes K_{(m_\ell)_{d}},
\end{equation}
with $K_{(m_\ell)_i}\in\mathbb{R}^{(m_{\ell+1})_i\times (m_\ell)_i}$, $(m_{\ell+1})_i\le (m_\ell)_i$ so-called \emph{cutting matrices}. In this view, since the IgA matrices in \cite{serra2} were perturbations of $d$-level Toeplitz matrices, one opted in \cite{serra2} for the same prolongator as in \eqref{eq:prolong} with
\begin{equation}\label{eq:pol}
q_\ell(\ttheta):=q(\ttheta):=\prod_{i=1}^d(1+\cos(\theta_i)),
\end{equation}
and $(K_{\mm_\ell})^T$ as in \eqref{eq:Kut} with
\begin{equation*}
K_{(m_\ell)_i}:=
\begin{bmatrix}
0 & 1 & 0 &   &        &    & &\\
&   & 0 & 1 & 0       &    & & \\
&   &   & \ddots  & \ddots & \ddots   & & \\
&   &   &   &        & 0  & 1 & 0
\end{bmatrix} \in {\mathbb R}^{(m_{\ell+1})_i\times (m_\ell)_i}.
\end{equation*}
It can be shown that such $P_{\mm_\ell}$  admits the following recursive expression
\begin{equation*}
P_{\mm_\ell}=P_{(m_\ell)_{1}}\otimes\cdots\otimes P_{(m_\ell)_{d}},
\end{equation*}
with $(m_\ell)_i=\frac{(m_0)_i-2^\ell+1}{2^\ell}$, $i=1,\ldots,d$, $\ell=0,\ldots,\lev-1$, and
\begin{equation}\label{eq:projector-1d}
P_{(m_\ell)_i}:=\frac{1}{2}
\begin{bmatrix}
1 & 2 & 1 &   &        &    & &\\
&   & 1 & 2 & 1       &    & & \\
&   &   & \ddots  & \ddots & \ddots   & & \\
&   &   &   &        & 1  & 2 & 1
\end{bmatrix}^T \in {\mathbb R}^{(m_\ell)_i\times(m_{\ell+1})_i}.
\end{equation}

A multigrid for $d$-level block Toeplitz matrices has been proposed in \cite{Huk}. As observed in that paper, when the generating function is an $s\times s$ diagonal matrix-valued function, a multigrid
for the whole matrix can be seen as $s$ independent multigrids for $d$-level Toeplitz matrices with scalar-valued symbols. Formally, let $f:[-\pi,\pi]^d\rightarrow\mathbb{C}^{s\times s}$ be defined as
\begin{equation*}
f(\ttheta):=\begin{bmatrix}
f_{1,1}(\ttheta) &  \\
& f_{2,2}(\ttheta) \\
& & \ddots \\
& & & f_{s,s}(\ttheta)
\end{bmatrix},
\end{equation*}
where each $f_{i,i}:[-\pi,\pi]^d\rightarrow \mathbb{C}$ has only a single isolated zero in $[-\pi,\pi]^d$ of order (at most) $2$. Then, we can define $s$ multigrid methods with prolongators as in \eqref{eq:prolong}, one for each $f_{i,i}$, choosing polynomials like in \eqref{eq:pol}.

In order to define the prolongator in our setting, let us focus on the $d\times d$ block diagonal matrix
\begin{equation}\label{diag}
\mathcal{D}^{p,\alpha,\beta}_{\nn}:=
\begin{bmatrix}
(\mathcal{A}^{p,\alpha,\beta}_{\nn})^{(1,1)} \\
& \ddots \\
& & (\mathcal{A}^{p,\alpha,\beta}_{\nn})^{(d,d)}
\end{bmatrix},
\end{equation}
by collecting the diagonal blocks of the matrix ${\mathcal{A}^{p,\alpha,\beta}_{\nn}}$. We recall that each of these blocks are perturbations of $d$-level Toeplitz matrices generated by
$$
f^{p,\alpha,\beta}_{1,1}, \ldots, f^{p,\alpha,\beta}_{d,d},
$$
respectively. Therefore, we define the following $d\times d$ block prolongator
\begin{equation}\label{eq:2by2prol}
\mathcal{P}_{\mm_\ell}:=\begin{bmatrix}
P_{\mm_\ell} \\
& \ddots \\
& & P_{\mm_\ell}
\end{bmatrix}
=I_d\otimes P_{(m_\ell)_1} \otimes \cdots \otimes P_{(m_\ell)_d},
\end{equation}
where $I_d$ is the $d\times d$ identity matrix, $P_{(m_\ell)_i}$ is defined in \eqref{eq:projector-1d}, and $d=2,3$. We expect a multigrid with the above prolongator and with a standard smoother (e.g., Gauss--Seidel or Jacobi smoother) to have a convergence rate independent of the matrix-size, not only for $\mathcal{D}^{p,\alpha,\beta}_{\nn}$ but also for $\mathcal{A}^{p,\alpha,\beta}_{\nn}$ when $|\alpha-\beta|\ll1$. On the other hand, the prolongator in \eqref{eq:2by2prol} (combined with any standard smoother) will lead to a convergence rate worsening with $p$. This drawback is addressed and overcome by choosing a $p$-independent PCG (or PGMRES) as smoother at the finest level.

\paragraph{$p$-independent PCG (or PGMRES) as post-smoother}
We now discuss the preconditioner used for the PCG (or PGMRES) smoother. In particular, we focus here on its construction and on how to solve efficiently the associated linear system. Afterwards, we will discuss why it copes with the ill-conditioning introduced in the high frequencies for large~$p$.

As mentioned in the beginning of the section, we define the preconditioner as
\begin{equation} \label{eq:kron_precond}
\mathcal{T}_{\nn}^{p}:=
\begin{bmatrix}
T_{\mm}(g_p) \\
& \ddots \\
& & T_{\mm}(g_p)
\end{bmatrix}
= I_d\otimes T_{n_1+p-2}(\mfm_{p-1}) \otimes \cdots \otimes T_{n_d+p-2}(\mfm_{p-1}),
\end{equation}
where $I_d$ is again the $d\times d$ identity matrix.
Such a preconditioner is easy to construct because
\begin{equation*}
(T_{n_l+p-2}(\mfm_{p-1}))_{i,j}=
\begin{cases}
\phi_{2p-1}(p-i+j), & \text{if } |i-j|<p, \\
0, &  \text{otherwise},
\end{cases}
\end{equation*}
i.e., its entries are nothing but evaluations of cardinal B-splines. Moreover, due to its tensor-product nature, the preconditioner \eqref{eq:kron_precond} is easily solvable. Indeed, we have
\begin{equation*}
\left(\mathcal{T}_{\nn}^{p}\right)^{-1}=
\begin{bmatrix}
(T_{\mm}(g_p))^{-1} \\
& \ddots \\
& & (T_{\mm}(g_p))^{-1}
\end{bmatrix},
\end{equation*}
where, by the properties of Kronecker product,
\begin{equation*}
(T_{\mm}(g_p))^{-1} = (T_{n_1+p-2}(\mfm_{p-1}))^{-1}\otimes \cdots\otimes (T_{n_d+p-2}(\mfm_{p-1})^{-1},
\end{equation*}
and each $(T_{n_l+p-2}(\mfm_{p-1}))^{-1}$ can be solved by means of an LU factorization which is optimal for banded matrices, i.e., linear in the matrix size (and quadratic in the bandwidth). Therefore, the computational cost for solving a linear system with coefficient matrix \eqref{eq:kron_precond} is linear in the total matrix size $d\prod_{l=1}^d(n_l+p-2)$.

\paragraph{Spectral motivation for the choice of preconditioner}
As explained before, we aimed for a preconditioner that works in the subspace of high frequencies. This means that it should remove the numerical zeros of the symbol $f^{p,\alpha,\beta}$ described in Corollaries \ref{cor:zeros2D} and \ref{cor:zeros3D} for $d=2,3$.
This can be formally justified with the help of the GLT axioms as follows.

Let us first look at the permuted matrices
$\Pi_{\mm}\,\mathcal{A}^{p,\alpha,\beta}_{\nn}\,\Pi_{\mm}^T$, so that we deal with a low rank perturbation of $d$-level block Toeplitz matrices. From Section \ref{sec:spectral} we know that $\{\Pi_{\mm}\, \mathcal{A}^{p,\alpha,\beta}_{\nn}\,\Pi_{\mm}^T\}_{\nn}$ is a GLT sequence with matrix-valued symbol $f^{p,\alpha,\beta}$. Furthermore, our permuted preconditioner is the $d$-level block Toeplitz matrix generated by $g_p I_d$, where $I_d$ is the $d\times d$ identity matrix.
Hence, by applying {\bf GLT3} for Toeplitz sequences generated by matrix-valued functions, $\{T_{\mm}(g_p I_d)\}_{\nn}$ is a GLT sequence with symbol $g_p I_d$. Consequently, by applying {\bf GLT2} and {\bf GLT1} (again their block version), we conclude that the preconditioned sequence
\begin{equation*}
\left\{(T_{\mm}(g_p I_d))^{-1}\left(\Pi_{\mm}\, \mathcal{A}^{p,\alpha,\beta}_{\nn}\,\Pi_{\mm}^T\right)\right\}_{\nn}
\end{equation*}
is still a GLT sequence with symbol
\begin{equation*}
\frac{f^{p,\alpha,\beta}(\theta_1,\ldots,\theta_d)}{g_p(\theta_1,\ldots,\theta_d)}.
\end{equation*}
Like in \cite{SINUM-IgA,serra2} for the preconditioned Laplacian symbol, an elementary analytical study of the above ratio (using item \ref{eq:sp} of Lemma \ref{lem:h-f-e}) shows that the numerical zeros of $f^{p,\alpha,\beta}$ at the points $(\theta_1,\ldots,\theta_d)$, for any $\theta_i=\pm\pi$, $i\in \{1,\ldots,d\}$
and for $p$ being large, are removed.
Hence, the preconditioned symbol keeps only the zero at $(0,\ldots,0)$, which can be treated effectively by a standard V-cycle multigrid, if the parameter $\beta$ is not too small.

Finally, we have to take into account that we work with the original coefficient matrix $\mathcal{A}^{p,\alpha,\beta}_{\nn}$ and not with its permuted version $\Pi_{\mm}\,\mathcal{A}^{p,\alpha,\beta}_{\nn}\,\Pi_{\mm}^T$. However, this permutation does not affect the argument, and hence, this explains why the proposed preconditioner is responsible for the robustness with respect to $p$ of the global method; it is the part that removes the ill-conditioning in the high frequency subspaces.

\begin{remark}
	The proposed multi-iterative method is able to cut the ill-conditioning of the linear system \eqref{eq:system} with respect to the matrix-size and the degree of the approximation, but another source of ill-conditioning should be taken into account, i.e., the one related to the parameters $\alpha,\beta$ of the problem. In this view, we adopt the classical strategy of using the proposed multigrid-type method as preconditioner for the CG method. As shown in the next numerical section, the result is indeed a valid attempt to guarantee robustness with respect to $\alpha,\beta$. The study of more sophisticated strategies will be subject of further research.
\end{remark}


\subsection{Numerical examples}\label{sub:numer_mg}

In the following, we test the effectiveness of the multi-iterative method introduced in Section \ref{sub:multi} used as a stand-alone method (label ``MIM'') or as preconditioner for the CG method (label ``$P_{\rm MIM}$'') for solving the linear system \eqref{eq:system}.
In addition, we check the performance of a Jacobi-type preconditioning (label ``$P_{\rm WL}$'') given by one iteration of our multigrid applied to the $d\times d$-block diagonal matrix $\mathcal{D}^{p,\alpha,\beta}_{\nn}$ defined in \eqref{diag} for $d=2,3$. The subscript ``${\rm WL}$'' stands for Weighted Laplacian and is used to recall that the diagonal blocks of $\mathcal{D}^{p,\alpha,\beta}_{\nn}$ are some sort of scalar Laplacian incorporating the weights $\alpha,\beta$; then, from our spectral analysis we expect such a preconditioner to improve reasonably the conditioning with respect to all the involved parameters as well.

The specific multigrid involved in all our proposals ($P_{\rm WL}$, $P_{\rm MIM}$, MIM) is defined by the following setting:
\begin{itemize}
	\item a V-cycle with the $d\times d$ block prolongator \eqref{eq:2by2prol} at each level (the number of recursion levels is given by $\lev=\log_2(n+p-1)$);
	\item one Gauss--Seidel pre-smoothing iteration at each level;
	$p$ post-smoothing iterations of the PCG (for $P_{\rm WL}$) or PGMRES (for $P_{\rm MIM}$ and MIM) at the finest level whose preconditioner is chosen as $\mathcal{T}_{\nn}^{p}$ defined in \eqref{eq:kron_precond}; and one Gauss--Seidel post-smoothing iteration at the other levels.
\end{itemize}

\renewcommand\arraystretch{1.2}
\begin{table}[t!]
	\centering
	\begin{tabular}{|c|cccc|c|cccc|c|cccc|}
		\hline
		\multicolumn{5}{|c|}{$p=1$} &\multicolumn{5}{c|}{$p=2$} &\multicolumn{5}{c|}{$p=3$} \\ \hline
		$n$ & $P_{{\rm WL}}$ & $P_{{\rm MIM}}$ & MIM & CG & $n$ & $P_{{\rm WL}}$ & $P_{{\rm MIM}}$ & MIM & CG & $n$ & $P_{{\rm WL}}$ & $P_{{\rm MIM}}$ & MIM & CG \\ \hline
		16  &   22  &   7  &   21  &   57  &   15  &   23  &   6  &   15  &   40  &   14  &   23  &   5  &   12  &   48 \\
		32  &   24  &   7  &   21  &  123  &   31  &   24  &   6  &   15  &   77  &   30  &   24  &   5  &   12  &   73 \\
		64  &   26  &   7  &   21  &  252  &   63  &   25  &   6  &   15  &  153  &   62  &   25  &   6  &   14  &  150 \\
		128  &   26  &   7  &   21  &  519  &  127  &   26  &   6  &   16  &  312  &  126  &   26  &   6  &   15  &  311 \\
		\hline
		\hline
		\multicolumn{5}{|c|}{$p=4$} &\multicolumn{5}{c|}{$p=5$} &\multicolumn{5}{c|}{$p=6$} \\ \hline
		$n$ & $P_{{\rm WL}}$ & $P_{{\rm MIM}}$ & MIM & CG & $n$ & $P_{{\rm WL}}$ & $P_{{\rm MIM}}$ & MIM & CG & $n$ & $P_{{\rm WL}}$ & $P_{{\rm MIM}}$ & MIM & CG \\ \hline
		13  &   23  &   5  &   11  &   82  &   12  &   22  &   5  &   10  &  125  &   11  &   22  &   6  &   11  &  206 \\
		29  &   24  &   5  &   12  &  113  &   28  &   24  &   5  &   11  &  206  &   27  &   24  &   5  &   12  &  333 \\
		61  &   25  &   5  &   13  &  167  &   60  &   25  &   5  &   13  &  267  &   59  &   25  &   6  &   13  &  475 \\
		125  &   26  &   6  &   14  &  330  &  124  &   26  &   6  &   14  &  383  &  123  &   26  &   6  &   14  &  620 \\
		\hline
	\end{tabular}
	\caption{2D case. Number of iterations required by $P_{\rm WL}$, $P_{\rm MIM}$, MIM, and the CG method for $\alpha=1$, $\beta=0.1$}\label{tab01}
	\bigskip
	\centering
	\begin{tabular}{|c|cccc|c|cccc|c|cccc|}
		\hline
		\multicolumn{5}{|c|}{$p=1$} &\multicolumn{5}{c|}{$p=2$} &\multicolumn{5}{c|}{$p=3$} \\ \hline
		$n$ & $P_{{\rm WL}}$ & $P_{{\rm MIM}}$ & MIM & CG & $n$ & $P_{{\rm WL}}$ & $P_{{\rm MIM}}$ & MIM & CG & $n$ & $P_{{\rm WL}}$ & $P_{{\rm MIM}}$ & MIM & CG \\ \hline
		16  &   42  &   18  &  122  &  115  &   15  &   63  &   16  &  114  &   91  &   14  &   67  &   15  &   94  &   91 \\
		32  &   57  &   19  &  140  &  291  &   31  &   67  &   17  &  116  &  207  &   30  &   68  &   16  &   94  &  201 \\
		64  &   67  &   21  &  152  &  685  &   63  &   71  &   18  &  116  &  427  &   62  &   72  &   17  &  102  &  419 \\
		128  &   73  &   21  &  156  & 1438  &  127  &   74  &   18  &  119  &  867  &  126  &   75  &   19  &  110  &  872 \\
		\hline
		\hline
		\multicolumn{5}{|c|}{$p=4$} &\multicolumn{5}{c|}{$p=5$} &\multicolumn{5}{c|}{$p=6$} \\ \hline
		$n$ & $P_{{\rm WL}}$ & $P_{{\rm MIM}}$ & MIM & CG & $n$ & $P_{{\rm WL}}$ & $P_{{\rm MIM}}$ & MIM & CG & $n$ & $P_{{\rm WL}}$ & $P_{{\rm MIM}}$ & MIM & CG \\ \hline
		13  &   65  &   15  &   88  &  107  &   12  &   61  &   15  &   80  &  149  &   11  &   60  &   16  &   88  &  223 \\
		29  &   65  &   15  &   87  &  207  &   28  &   64  &   16  &   85  &  292  &   27  &   64  &   17  &   80  &  436 \\
		61  &   70  &   17  &   96  &  435  &   60  &   68  &   17  &   92  &  451  &   59  &   68  &   17  &   89  &  680 \\
		125  &   73  &   19  &  106  &  922  &  124  &   72  &   18  &  103  &  982  &  123  &   72  &   18  &  100  & 1039 \\
		\hline
	\end{tabular}
	\caption{2D case. Number of iterations required by $P_{\rm WL}$, $P_{\rm MIM}$, MIM, and the CG method for $\alpha=1$, $\beta=0.01$}\label{tab001}
\end{table}

For our 2D test example, we consider problem \eqref{eq:curl-div} with $d=2$ and the manufactured solution
\begin{eqnarray*}
	{\bm u}=\left[\begin{array}{c}
		u^1(x_1, x_2)\\
		u^2(x_1, x_2)\\
	\end{array}\right]
	=
	\left[\begin{array}{c}
		\sin(2\pi x_1(1-x_1)x_2(1-x_2))\\
		\cos(2\pi x_1(1-x_1)x_2(1-x_2))-1\\
	\end{array}\right]
\end{eqnarray*}
with $(x_1, x_2)\in (0,1)^2$. Moreover, we fix $n_1=n_2=n$.
We compare the number of iterations required by our proposals with the number of CG iterations in Tables \ref{tab01}--\ref{tab001}. Note that in all tables we choose $n$ so to keep the same matrix-size for each $p$. As stopping criterion, we use $\|\rr^{(k)}\|_2/\|\rr^{(0)}\|_2<10^{-7}$, where $\rr^{(k)}$ is the residual vector after $k$ iterations. The initial guess is always chosen to be the zero vector.

In Table \ref{tab01} we fix $\alpha=1,\beta=0.1$. As expected, while the number of CG iterations increases both in the matrix-size and the degree, all our ad hoc proposals are optimal and robust. However, we clearly notice that $P_{{\rm MIM}}$ is outperforming $P_{{\rm WL}}$ and MIM. In other words, the multi-iterative method used as preconditioner is the most robust  strategy with respect to all parameters of the problem. 
This behavior is even more evident when $\beta$ decreases to $0.01$ as confirmed by the results in Table \ref{tab001}.
Finally, note that for both $\beta=0.1,0.01$ the iterations required by MIM for $p=1$ is higher than for the other degrees. This could be justified by the fact that for $p=1$, the preconditioner $\mathcal{T}_{\nn}^{p}$ used at the finest level post-smoothing step is nothing but the identity matrix.

\begin{table}[t!]
	\centering
	\begin{tabular}{|c|ccc|c|ccc|c|ccc|}
		\hline
		\multicolumn{4}{|c|}{$p=1$} &\multicolumn{4}{c|}{$p=2$} &\multicolumn{4}{c|}{$p=3$} \\ \hline
		$n$ & $P_{{\rm MIM}}$ & MIM & CG & $n$ & $P_{{\rm MIM}}$ & MIM & CG & $n$ & $P_{{\rm MIM}}$ & MIM & CG \\ \hline
		8  &   7  &   19  &   27  &    7  &   6  &   14  &   29  &    6  &   5  &   10  &   55 \\
		16  &   8  &   22  &   56  &   15  &   6  &   15  &   38  &   14  &   5  &   11  &   64 \\
		32  &   8  &   23  &  120  &   31  &   6  &   15  &   74  &   30  &   5  &   12  &   76 \\
		\hline
		\hline
		\multicolumn{4}{|c|}{$p=4$} &\multicolumn{4}{c|}{$p=5$} &\multicolumn{4}{c|}{$p=6$} \\ \hline
		$n$ & $P_{{\rm MIM}}$ & MIM & CG & $n$ & $P_{{\rm MIM}}$ & MIM & CG & $n$ & $P_{{\rm MIM}}$ & MIM & CG \\ \hline
		5  &   6  &   13  &   79  &    4  &      &       &       &    3  &      &       &      \\
		13  &   5  &   10  &   93  &   12  &   6  &   15  &  161  &   11  &   8  &   15  &  267 \\
		29  &   5  &   11  &  114  &   28  &   5  &   11  &  212  &   27  &   6  &   14  &  353 \\
		\hline
	\end{tabular}
	\caption{3D case. Number of iterations required by $P_{\rm MIM}$, MIM, and the CG method for $\alpha=1$,
		$\beta=0.1$}\label{tab3D01}
	\bigskip
	\centering
	\begin{tabular}{|c|ccc|c|ccc|c|ccc|}
		\hline
		\multicolumn{4}{|c|}{$p=1$} &\multicolumn{4}{c|}{$p=2$} &\multicolumn{4}{c|}{$p=3$} \\ \hline
		$n$ & $P_{{\rm MIM}}$ & MIM & CG & $n$ & $P_{{\rm MIM}}$ & MIM & CG & $n$ & $P_{{\rm MIM}}$ & MIM & CG \\ \hline
		8  &   12  &   61  &   38  &    7  &   15  &  106  &   46  &    6  &   13  &   88  &   72 \\
		16  &   17  &  121  &  114  &   15  &   16  &  112  &  103  &   14  &   14  &   87  &  102 \\
		32  &   20  &  158  &  292  &   31  &   16  &  120  &  205  &   30  &   15  &   94  &  202 \\
		\hline
		\hline
		\multicolumn{4}{|c|}{$p=4$} &\multicolumn{4}{c|}{$p=5$} &\multicolumn{4}{c|}{$p=6$} \\ \hline
		$n$ & $P_{{\rm MIM}}$ & MIM & CG & $n$ & $P_{{\rm MIM}}$ & MIM & CG & $n$ & $P_{{\rm MIM}}$ & MIM & CG \\ \hline
		5  &   16  &   72  &  114  &    4  &       &       &       &    3  &       &       &      \\
		13  &   15  &   83  &  142  &   12  &   17  &   88  &  229  &   11  &   19  &  105  &  360 \\
		29  &   15  &   85  &  214  &   28  &   15  &   81  &  316  &   27  &   17  &   80  &  487 \\
		\hline
	\end{tabular}
	\caption{3D case. Number of iterations required by $P_{\rm MIM}$, MIM, and the CG method for $\alpha=1$, $\beta=0.01$}\label{tab3D001}
\end{table}

The story is similar in the 3D case.
As test example, we consider problem \eqref{eq:curl-div} with $d=3$ and the manufactured solution
\begin{eqnarray*}
	{\bm u}=\left[\begin{array}{c}
		u^1(x_1, x_2,x_3)\\
		u^2(x_1, x_2,x_3)\\
		u^3(x_1,x_2,x_3)\\
	\end{array}\right]
	=
	\left[\begin{array}{c}
		\sin(2\pi x_1(1-x_1)x_2(1-x_2)x_3(1-x_3))\\
		\cos(2\pi x_1(1-x_1)x_2(1-x_2)x_3(1-x_3))-1\\
		u^1(x_1, x_2,x_3)+u^2(x_1, x_2,x_3)\\
	\end{array}\right]
\end{eqnarray*}
with $(x_1, x_2,x_3)\in (0,1)^3$. Again, we fix $n_1=n_2=n_3=n$. We compare the number of iterations required by $P_{{\rm MIM}}$ and MIM with the number of CG iterations in Tables \ref{tab3D01}--\ref{tab3D001}. Also in 3D, the strategy $P_{{\rm MIM}}$ is the most robust with respect to all parameters of the problem.

\section{Conclusions}\label{sec:end}

In this paper, we focused on the B-spline approximation of a parameter-dependent prototype of an {Alfv\'en}-like operator, which appears as a subproblem in the complex MHD model, both in a 2D and 3D setting. In particular, we analyzed the spectral distribution of the corresponding coefficient matrices with the aim of designing fast iterative solvers for them.

By employing tools coming from the GLT theory, the study of the coefficient matrices has emphasized a rich spectral structure with a critical dependence on the several physical and approximation  parameters.
Such spectral information has been used  in order to design suitable iterative solvers for the corresponding linear systems, which result in a combination of multigrid technique and preconditioned Krylov solvers.


%

The case of operators with variable coefficients and complex geometries has not been discussed here but is of course imperative to face real-life MHD problems and need to be addressed in subsequent steps. We just remark that efficient solvers for the constant coefficient case on the reference unit square/cube domain are important not only as a starting step towards the treatment of fully general problems, but also because they often serve as efficient preconditioners for the most general setting, at least for not too 
extreme configurations (see, e.g., \cite{serra2,cmame2}).

\section*{Acknowledgment}
This work was partially supported
by the MIUR-DAAD Joint Mobility 2017 Programme through the project ``ATOMA'' and
by the Mission Sustainability Programme of the University of Rome ``Tor Vergata'' through the project ``IDEAS''. Moreover, C. Manni, M. Mazza, S. Serra-Capizzano, and H. Speleers are members of the INdAM Research group GNCS (Gruppo Nazionale per il Calcolo Scientifico), which partially supported this work as well.
Special thanks go to Prof. Franco Brezzi for the fruitful discussion at the IMACS 2016 conference.

%





\end{document}